\documentclass[12pt,a4paper]{article}

\usepackage{exscale,makeidx,amssymb,amsmath,bbm,color,psfrag}
\usepackage{amsmath,amscd,mathrsfs}
\usepackage{graphicx}
\usepackage{enumerate}
\usepackage{amsthm,amsfonts}
\usepackage[in]{fullpage}
\usepackage{subfigure}
\usepackage[normalem]{ulem}
\usepackage[colorlinks=true,urlcolor=blue,citecolor=blue]{hyperref}
\usepackage{etoolbox}

\newtheorem{theorem}{Theorem}[section]
\newtheorem{definition}[theorem]{Definition}
\newtheorem{proposition}[theorem]{Proposition}
\newtheorem{corollary}[theorem]{Corollary} 
\newtheorem{lemma}[theorem]{Lemma}

\renewcommand{\oe}{\mathrm{oe}}
\newcommand{\LL}{\mathbb{L}}

\newcommand{\dghp}{\mathrm{d}_{\mathrm{GHP}}}

\newcommand{\gh}{\mathrm{GH}}
\newcommand{\ghp}{\mathrm{GHP}}
\newcommand{\dhau}{\mathrm{d}_{\mathrm{H}}}

\newcommand{\N}{\mathbb{N}}
\newcommand{\R}{\mathbb{R}}
\newcommand{\Z}{\mathbb{Z}}

\newcommand{\M}{\mathbb{M}}

\renewcommand{\AA}{\mathcal{A}}

\newcommand{\ind}{\mathbbm{1}}

\newcommand{\convdist}{\ensuremath{\stackrel{d}{\rightarrow}}}

\newcommand{\fl}[1]{\ensuremath{\lfloor #1 \rfloor}}

\newcommand{\diam}{\mathrm{diam}}
\newcommand{\E}[1]{\ensuremath{\mathbb{E}\left[#1\right]}}

\newcommand{\p}[1]{\ensuremath{\mathbb{P}\left(#1\right)}}

\newcommand{\pran}[1]{\left(#1\right)}
\newcommand{\eps}{\ensuremath{\epsilon}}

\providecommand{\Ell}{}
\renewcommand{\Ell}{\ensuremath{L}}

\definecolor{tocc}{rgb}{0,0,0.5}
\newcommand{\refP}[1]{Proposition~\ref{#1}}
\newcommand{\refL}[1]{Lemma~\ref{#1}}

\newcommand{\refS}[1]{Section~\ref{#1}}

\newcommand{\MM}{\mathbb{M}}

\newcommand{\bX}{\mathbf{X}}
\newcommand{\bY}{\mathbf{Y}}
\newcommand{\bZ}{\mathbf{Z}}
\newcommand{\dist}{\mathrm{dist}}

\newcommand\textcolour\textcolor
\newcommand{\eqdist}{\stackrel{\mathrm{d}}{=}}
\newcommand{\dgh}{\ensuremath{\mathrm{d}_{\mathrm{GH}}}}

\newcommand{\rg}{$\R$-graph~}
\newcommand{\rX}{\ensuremath{\mathrm{X}}}
\newcommand{\rY}{\ensuremath{\mathrm{Y}}}

\newcommand{\skel}{\ensuremath{\mathrm{skel}}}

\newcommand{\core}{\ensuremath{\mathrm{core}}}
\newcommand{\conn}{\ensuremath{\mathrm{conn}}}
\newcommand{\len}{\ensuremath{\mathrm{len}}}
\newcommand{\im}{\ensuremath{\mathrm{Im}}}
\newcommand{\cut}{\ensuremath{\mathrm{cut}}}
\newcommand{\dis}{\ensuremath{\mathrm{dis}}}
\newcommand{\cB}{\mathcal{B}}

\newcommand{\cP}{\mathcal{P}}

\newcommand{\cM}{\mathcal{M}}

\newcommand{\rZ}{\ensuremath{\mathrm{Z}}}

\newcommand{\bbM}{\ensuremath{\mathbb{M}}}

\def\build#1_#2^#3{\mathrel{\mathop{\kern 0pt#1}\limits_{#2}^{#3}}}

\usepackage[refpage,noprefix,intoc]{nomencl}							
\makenomenclature											
\title{The scaling limit of the minimum spanning tree of the complete graph\vspace{-0.1cm}}
\author{L. Addario-Berry\thanks{Department of Mathematics and Statistics, McGill University}, N. Broutin\thanks{Projet RAP, Inria Rocquencourt-Paris}, C. Goldschmidt\thanks{Department of Statistics and Lady Margaret Hall, University of Oxford},  G. Miermont\thanks{UMPA, \'Ecole Normale Sup\'erieure de Lyon}\vspace{-0.1cm}}
\date{January 8, 2013}

\begin{document}

\maketitle

\vspace{-1cm}
\begin{abstract}
Consider the minimum spanning tree (MST) of the complete graph
  with $n$ vertices, when edges are assigned independent 
  random weights. Endow this tree with the graph distance
  renormalized by $n^{1/3}$ and with the uniform measure on its vertices.
  We show that the resulting space converges in distribution as $n\to\infty$ to a random
  measured metric space in the Gromov--Hausdorff--Prokhorov topology. We additionally 
  show that the limit is a random binary $\R$-tree and has
  Minkowski dimension $3$ almost surely. In particular, its law is
  mutually singular with that of the Brownian continuum random tree or
  any rescaled version thereof. Our approach relies on a coupling between
  the MST problem and the Erd\H{o}s--R\'enyi random
  graph. 
  We exploit the explicit description of the scaling limit of the Erd\H{o}s--R\'enyi random 
  graph in the so-called critical window, established in \cite{ABBrGo09a}, and   
  provide a similar description of the scaling limit for a ``critical minimum spanning forest'' contained within the MST. 
\end{abstract}

\vspace{-1cm}
\begin{figure}[h!]
  \begin{center}
\scalebox{-0.41}[0.41]{\includegraphics[trim=2.5cm 2.5cm 2cm 2.5cm, clip=true,angle=180]{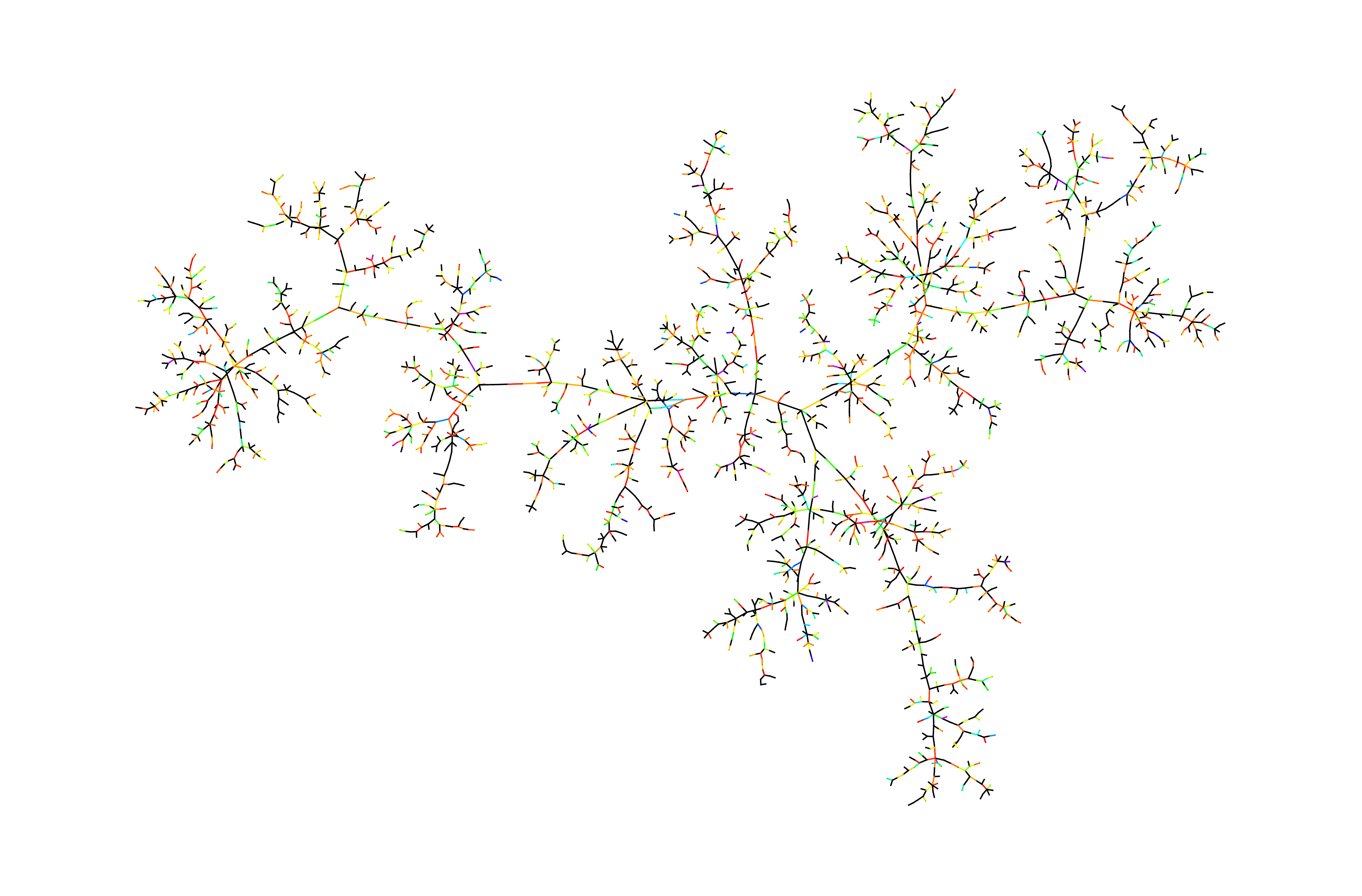}}
 \end{center}
\end{figure}
\vspace{-0.7cm}

\noindent Figure 1: A simulation of the minimum spanning tree on
  $K_{3000}$. Black edges have weights less than $1/3000$; for coloured
  edges, weights increase as colours vary from from red to purple.

\newpage

\addtocounter{figure}{1}

\begingroup
\hypersetup{linkcolor=tocc}
\tableofcontents
\endgroup

\section{Introduction}

\subsection{A brief history of minimum spanning trees} 
The minimum spanning tree (MST) problem is one of the first and
foundational problems in the field of combinatorial optimisation. In
its initial formulation by Bor\r{u}vka \cite{boruvka26jistem}, one is
given distinct, positive edge weights (or {\em lengths}) for $K_n$,
\nomenclature[Kn]{$K_n$}{Complete graph on $\{1,\ldots,n\}$.}
the complete graph on vertices labelled by the elements of
$\{1,\ldots,n\}$. Writing $\{w_e,e \in E(K_n)\}$ for this collection
\nomenclature[Eg]{$E(G)$}{Set of edges of the graph $G$.}
of edge weights, one then seeks the unique connected subgraph $T$ of
$K_n$ with vertex set $V(T) = \{1,\ldots,n\}$ 
\nomenclature[Vg]{$V(G)$}{Set of vertices of the graph $G$.}
that minimizes the {\em
  total length}
\begin{equation}\label{eq:totallength}
\sum_{e \in E(T)} w_e\, .
\end{equation}
Algorithmically, the MST problem is among the easiest in combinatorial
optimisation: procedures for building the MST are both easily
described and provably efficient. The most widely known MST algorithms
are commonly called {\em Kruskal's algorithm} and {\em Prim's
  algorithm}.\footnote{Both of these names are misnomers or, at the
  very least, obscure aspects of the subject's development; see Graham
  and Hell \cite{graham1985history} or Schriver
  \cite{schriver05history} for careful historical accounts.} Both
procedures are important in this work; as their descriptions are
short, we provide them immediately.

\medskip
\noindent {\sc Kruskal's algorithm}: start from a forest of $n$
isolated vertices $\{1,\ldots,n\}$. At each step, add the unique edge
of smallest weight joining two distinct components of the current
forest. Stop when all vertices are connected.

\medskip
\noindent {\sc Prim's algorithm}: fix a starting vertex $i$. At each
step, consider all edges joining the component currently containing
$i$ with its complement, and from among these add the unique edge of
smallest weight. Stop when all vertices are connected.  \medskip

Unfortunately, efficient procedures for constructing MST's do not
automatically yield efficient methods for understanding the typical
structure of the resulting objects. To address this, a common approach
in combinatorial optimisation is to study a procedure by examining how
it behaves when given random input; this is often called {\em average
  case} or {\em probabilistic analysis.}

The probabilistic analysis of MST's dates back at least as far as
Beardwood, Halton and Hammersley \cite{beardwood59shortest} who
studied the \emph{Euclidean MST} of $n$ points in $\R^d$.  Suppose
that $\mu$ is an absolutely continuous measure on $\R^d$ with
bounded support, and let $(P_i, i \ge 1)$ be i.i.d.\ samples from
$\mu$.  For edge $e=\{i,j\}$, take $w_e$ to be the Euclidean distance
between $P_i$ and $P_j$.  Then there exists a constant $c=c(\mu)$ such
that if $X_n$ is the total length of the minimum spanning tree, then
\[
\frac{X_n}{n^{(d-1)/d}} \stackrel{\mathrm{a.s.}}{\rightarrow} c. 
\]
This \emph{law of large numbers for Euclidean MST's} is the
jumping-off point for a massive amount of research: on more general
laws of large numbers
\cite{steele81subadditive,steele1988growth,yukich1996ergodic,penrose2003weak},
on central limit theorems
(\cite{alexander1996rsw,kesten1996central,lee1997central,lee1999central,penrose2001central,yukich2000asymptotics},
and on the large-$n$ scaling of various other ``localizable''
functionals of random Euclidean MST's
(\cite{penrose1996random,penrose1998extremes,penrose1999strong,steele1987number,kozma2006minimal}. (The
above references are representative, rather than exhaustive. The books
of Penrose \cite{penrose2003random} 
and of Yukich
\cite{yukich1998probability} are comprehensive compendia of the known
results and techniques for such problems.)

From the perspective of Bor\r{u}vka's original formulation, the most
natural probabilistic model for the MST problem may be the
following. Weight the edges of the complete graph $K_n$ with
independent and identically distributed (i.i.d.) random edge weights
$\{W_e: e \in E(K_n)\}$ whose common distribution $\mu$ is atomless
and has support contained in $[0,\infty)$, and let $\MM^n$ be the
resulting random MST. 
\nomenclature[Mn]{$\mathbb{M}^n$}{Minimum spanning tree of $K_n$ (as a graph).}
The conditions on $\mu$ ensure that all edge
weights are positive and distinct. Frieze \cite{frieze85mst} showed
that if the common distribution function $F$ is differentiable
at $0^+$ and $F'(0^+)>0$, then the total weight $X_n$ satisfies
\begin{equation}\label{eq:knlln}
F'(0^+)\cdot \E{X_n} \to \zeta(3),
\end{equation}
whenever the edge weights have finite mean. It is also known that
$F'(0^+)\cdot X_n \stackrel{\mathrm{p}}{\to} \zeta(3)$ without any
moment assumptions for the edge weights
\cite{frieze85mst,steele1987frieze}.  Results analogous to
(\ref{eq:knlln}) have been established for other graphs, including the
hypercube \cite{penrose98mst}, high-degree expanders and graphs of
large girth \cite{beveridge1998random}, and others
\cite{frieze89mst,frieze2000note}.

Returning to the complete graph $K_n$, Aldous \cite{aldous90mst}
proved a distributional convergence result corresponding to
(\ref{eq:knlln}) in a very general setting where the edge weight
distribution is allowed to vary with $n$, extending earlier, related
results \cite{timofeev88finding,avram92minimum}. Janson
\cite{janson95mst} showed that for i.i.d.\ Uniform$[0,1]$ edge weights
on the complete graph, $n^{1/2} (X_n - \zeta(3))$ is asymptotically
normally distributed, and gave an expression for the variance that was
later shown \cite{janson06mst} to equal $6\zeta(4)-4\zeta(3)$.

If one is interested in the {\em graph theoretic} structure of the
tree $\MM^n$ rather than in information about its edge weights, the
choice of distribution $\mu$ is irrelevant. To see this, observe that
the behaviour of both Kruskal's algorithm and Prim's algorithm is
fully determined once we order the edges in increasing order of
weight, and for any distribution $\mu$ as above, ordering the edges by
weight yields a uniformly random permutation of the edges. We are thus
free to choose whichever distribution $\mu$ is most convenient, or
simply to choose a uniformly random permutation of the edges. Taking
$\mu$ to be uniform on $[0,1]$ yields a particularly fruitful
connection to the now-classical {\em Erd\H{o}s--R\'enyi random graph
  process}. This connection has proved fundamental to the detailed
understanding of the global structure of $\MM^n$ and is at the heart
of the present paper, so we now explain it.

Let the edge weights $\{W_e: e \in E(K_n)\}$ be i.i.d.\ Uniform$[0,1]$
random variables. The Erd\H{o}s--R\'enyi graph process
$(\mathbb{G}(n,p),0 \le p \le 1)$ is the increasing graph process
obtained by letting $\mathbb{G}(n,p)$ have vertices $\{1,\ldots,n\}$
and edges $\{e \in E(K_n): W_e \le p\}$.\footnote{Later, it will be
  convenient to allow $p \in \R$, and we note that the definition of
  $\mathbb{G}(n,p)$ still makes sense in this case.} For fixed $p$,
each edge of $K_n$ is independently present with
probability $p$. Observing the process as $p$ increases from zero to
one, the edges of $K_n$ are added one at a time in exchangeable random
order. This provides a natural coupling with the behaviour of
Kruskal's algorithm for the same weights, in which edges are {\em
  considered} one at a time in exchangeable random order, and added
precisely if they join two distinct components. More precisely, for $0
< p < 1$ write $\mathbb{M}(n,p)$ for the subgraph of the MST $\MM^n$
with edge set $\{e \in E(\MM^n): W_e \le p\}$. Then for every $0 < p <
1$, the connected components of $\mathbb{M}(n,p)$ and of
$\mathbb{G}(n,p)$ have the same vertex sets.

In their foundational paper on the subject \cite{erdos60evolution},
Erd\H{o}s and R\'enyi described the {\em percolation phase transition}
for their eponymous graph process. They showed that for $p=c/n$ with
$c$ fixed, if $c < 1$ (the {\em subcritical} case) then
$\mathbb{G}(n,p)$ has largest component of size $O(\log n)$ in
probability, whereas if $c > 1$ (the {\em supercritical} case) then
the largest component of $\mathbb{G}(n,p)$ has size
$(1+o_p(1))\gamma(c) n$, where 
$\gamma(c)$ is the survival probability of a Poisson$(c)$ branching process. 
They also showed that for $c>1$,
all components aside from the largest have size $O(\log n)$ in
probability.

In view of the above coupling between the graph process and Kruskal's
algorithm, the results of the preceding paragraph strongly suggest
that ``most of'' the global structure of the MST $\MM^n$ should
already be present in the largest component of $\mathbb{M}(n,c/n)$,
for any $c > 1$. In order to understand $\MM^n$, then, a natural
approach is to delve into the structure of the forest
$\mathbb{M}(n,p)$ for $p \sim 1/n$ (the {\em near-critical} regime)
and, additionally, to study how the components of this forest attach
to one another as $p$ increases through the near-critical regime.  In
this paper, we use such a strategy to show that after suitable
rescaling of distances and of mass, the tree $\MM^n$, viewed as a
measured metric space, converges in distribution to a random compact
measured metric space $\mathscr{M}$ of total mass measure one, which
is a {\em random real tree} in the sense of
\cite{evans2006rpr,legall06rrt}.

The space $\mathscr{M}$ is the scaling limit of the minimum spanning
tree on the complete graph.  It is binary and its mass measure is
concentrated on the leaves of $\mathscr{M}$. The space $\mathscr{M}$
shares all these features with the first and most famous random real
tree, the Brownian continuum random tree, or CRT
\cite{aldouscrt91,aldouscrtov91,aldouscrt93,legall06rrt}. However,
$\mathscr{M}$ is not the CRT; we rule out this possibility 
by showing that $\mathscr{M}$ almost surely has Minkowski dimension 3. 
Since the CRT has both Minkowski dimension 2
and Hausdorff dimension 2, this shows that the law of $\mathscr{M}$ is mutually singular with that of the CRT, 
or any rescaled version thereof.

The remainder of the introduction is structured as follows. First,
Section~\ref{sec:overview}, below, we provide the precise statement of
our results. Second, in Section~\ref{sec:proof_idea} we provide an
overview of our proof techniques. Finally, in
Section~\ref{sec:history}, we situate our results with respect to the
large body of work by the probability and statistical physics
communities on the convergence of minimum spanning trees, and briefly
address the question of universality.

\subsection{The main results of this paper}\label{sec:overview}
Before stating our results, a brief word on the spaces in which we
work is necessary.  We formally introduce these spaces in
Section~\ref{sec:metrics}, and here only provide a brief
summary. First, let $\mathcal{M}$ be the set of measured
isometry-equivalence classes of compact measured metric
spaces, and let $\dghp$ denote the Gromov--Hausdorff--Prokhorov
distance on $\cM$; the pair $(\cM,\dghp)$ forms a Polish space. 

We wish to think of $\MM^n$ as an element of $(\cM,\dghp)$.  In order
to do this, we introduce a measured metric space $M^n$ obtained from
$\MM^n$ by rescaling distances by $n^{-1/3}$ and assigning mass $1/n$
to each vertex.  The main contribution of this paper is the following
theorem. 
\nomenclature[Mnaaa]{$M^n,\mathscr{M}$}{$M^n$ is the measured metric space version of $\mathbb{M}^n$; $\mathscr{M}$ is its GHP limit.}
\begin{theorem}\label{thm:main-0}
  There exists a random, compact measured metric space $\mathscr{M}$ 
  such that, as $n \to \infty$, 
\[
M^n \convdist \mathscr{M}\, 
\]
in the space $(\cM,\dghp)$.  The limit $\mathscr{M}$ is a random
$\R$-tree. It is almost surely binary, and its mass measure is
concentrated on the leaves of $\mathscr{M}$.  Furthermore, almost
surely, the Minkowski dimension of $\mathscr{M}$ exists and is
equal to $3$.
\end{theorem}

A consequence of the last statement is that $\mathscr{M}$ is not a
rescaled version of the Brownian CRT $\mathscr{T}$, 
\nomenclature[T]{$\mathscr{T}$}{Brownian CRT.}
in the sense that
for any non-negative random variable $A$, the laws of $\mathscr{M}$
and the space $\mathscr{T}$, in which all distances are multiplied by
$A$, are mutually singular. Indeed, the Brownian tree has Minkowski
dimension $2$ almost surely. The assertions of
Theorem~\ref{thm:main-0} are contained within the union of
Theorems~\ref{thm:big_statement} and~\ref{thm:basic_properties} and
Corollary~\ref{cor:notCRT}, below.

In a preprint \cite{Louigi} posted simultaneously with the current work, 
the first author of this paper shows that
the {\em unscaled} tree $\mathbb{M}^n$, when rooted at vertex $1$, converges in the local weak sense to a 
random infinite tree, and that this limit almost surely has cubic volume growth. 
The results of \cite{Louigi} form a natural complement to 
Theorem~\ref{thm:main-0}.

As mentioned earlier, we approach the study of $M^n$ and of its
scaling limit $\mathscr{M}$ via a detailed description of the graph
$\mathbb{G}(n,p)$ and of the forest $\mathbb{M}(n,p)$, for $p$ near
$1/n$. As is by this point well-known, it turns out that the right
scaling for the ``critical window'' is given by taking $p=1/n +
\lambda/n^{4/3}$, for $\lambda \in \R$, and for such $p$, the largest
components of $\mathbb{G}(n,p)$ typically have size of order $n^{2/3}$
and possess a bounded number of cycles
\cite{luczak1994srg,aldous97}. Adopting this parametrisation, for
$\lambda \in \R$ write
\[
(\mathbb{G}^{n,i}_{\lambda},i \ge 1)
\]
for the components of $\mathbb{G}(n,1/n + \lambda/n^{4/3})$ listed in
decreasing order of size (among components of equal size, list
components in increasing order of smallest vertex label, say). For
each $i \ge 1$, we then write $G_{\lambda}^{n,i}$ for the measured
metric space obtained from $\mathbb{G}^{n,i}_{\lambda}$ by rescaling
distances by $n^{-1/3}$ and giving each vertex mass $n^{-2/3}$, and
let
\[
G_{\lambda}^n = (G^{n,i}_{\lambda},i \ge 1). 
\]
We likewise define a sequence $(\mathbb{M}^{n,i}_{\lambda},i \ge 1)$
of graphs, and a sequence $M_\lambda^n = (M^{n,i}_{\lambda},i \ge 1)$
of measured metric spaces, starting from
$\mathbb{M}(n,1/n+\lambda/n^{4/3})$ instead of 
$\mathbb{G}(n,1/n+\lambda/n^{4/3})$. 

In order to compare sequences $\bX = (\rX_i,i \ge 1)$ of elements of
$\cM$ (i.e., elements of $\cM^{\N}$), we let
$\LL_p$, for $p\geq 1$, be the set of sequences $\bX \in \cM^{\N}$ with
\[
\sum_{i \ge 1} \diam(\rX_i)^p +\sum_{i\geq 1}\mu_i(X_i)^p< \infty\, , 
\] 
and for two such sequences $\bX=(\rX_i,i \ge 1)$ and $\bX' = (\rX'_i,i \ge
1)$, we let 
\[
\dist^p_{\ghp}(\bX,\bX')=\left( \sum_{i \ge 1}\dghp(\rX_i,\rX'_i)^p \right)^{1/p}\, .
\]
The resulting metric space $(\LL_p,\dist^p_{\ghp})$ is a Polish space. 

The second main result of this
paper is the following (see Theorems \ref{thm:msf-converge} and
\ref{thm:big_statement} below). 
\begin{theorem}\label{thm:msf-converge-0}
Fix $\lambda \in \R$. Then there exists a random sequence $\mathscr{M}_{\lambda}=(\mathscr{M}_{\lambda}^i,i \ge 1)$ of compact measured metric spaces, 
such that as $n \to \infty$,
\begin{equation}\label{eq:abbgr1}
M_{\lambda}^n \convdist \mathscr{M}_{\lambda}
\end{equation}
in the space $(\mathbb{L}_4, \mathrm{dist}_{\ghp}^4)$. Furthermore,
let $\hat{\mathscr{M}}_\lambda^1$ be the first term
$\mathscr{M}_{\lambda}^1$ of the limit sequence $\mathscr{M}_\lambda$,
with its measure renormalized to be a probability. Then as
$\lambda \to \infty$, $\hat{\mathscr{M}}_\lambda^1$ converges in
distribution to $\mathscr{M}$ in the space $(\mathcal{M},\dghp)$.
\end{theorem}

\subsection{An overview of the proof}\label{sec:proof_idea}
 Theorem 1 of \cite{ABBrGo09a} states that for each $\lambda \in \R$, there is a random sequence \[
\mathscr{G}_{\lambda} = (\mathscr{G}_{\lambda}^i, i \ge 1)
\]
of compact measured metric spaces, such that 
\begin{equation}\label{eq:abbr}
G_{\lambda}^n \convdist \mathscr{G}_{\lambda},
\end{equation}
in the space $(\LL_4,\dist^4_{\ghp})$. (Theorem 1 of \cite{ABBrGo09a}
is, in fact, slightly weaker than this because the metric spaces there
are considered without their accompanying measures, but it is easily
strengthened; see Section~\ref{sec:proofs}.) The limiting spaces are
similar to $\R$-trees; we call them {\em $\R$-graphs}. In
Section~\ref{sec:RtreesRgraphs} we define $\R$-graphs and develop a
decomposition of $\R$-graphs analogous to the classical ``core  
and kernel'' decomposition for finite connected graphs (see, e.g.,
\cite{janson00random}). 
We expect this generalisation of the theory of $\R$-trees to find further 
applications. 
The main results of \cite{ABBrGo09b} provide
precise distributional descriptions of the cores and kernels of the
components of $\mathscr{G}_{\lambda}$.

It turns out that, having understood the distribution of
$G_{\lambda}^n$, we can access the distribution of $M_{\lambda}^n$ by
using a minimum spanning tree algorithm called {\em cycle
  breaking}. This algorithm finds the minimum weight spanning tree of
a graph by listing edges in {\em decreasing} order of weight, then
considering each edge in turn and removing it if its removal leaves
the graph connected.

Using the convergence in (\ref{eq:abbr}) and an analysis of the cycle breaking algorithm, we will establish Theorem~\ref{thm:msf-converge-0}. 
The sequence $\mathscr{M}_{\lambda}$ is constructed from $\mathscr{G}_{\lambda}$ by a continuum analogue of the cycle breaking procedure. Showing that the continuum analogue of cycle breaking is well-defined and commutes with the appropriate limits is somewhat involved; this is the subject of Section~\ref{sec:ckcc}. 

For fixed $n$, the process $(M^{n,1}_{\lambda},\lambda\in \R)$ is
eventually constant, and we note that $\M^n = \lim_{\lambda \to \infty} \M^{n,1}_{\lambda}$
In order to establish that $M^n$ converges in distribution in the space $(\cM,\dghp)$ as $n \to \infty$, we rely on two ingredients. First, the convergence
in (\ref{eq:abbgr1}) 
is strong enough to imply that the first component $M_\lambda^{n,1}$ converges in distribution as
$n\to\infty$ to a limit $\mathscr{M}^1_\lambda$ in the space $(\cM,\dghp)$. 

Second, the results in \cite{abbr09} entail
Lemma~\ref{lem:cauchy-property2}, which in particular implies that for any $\eps > 0$, 
\begin{equation}\label{eq:abbr09b}
  \lim_{\lambda \to \infty} \limsup_{n \to \infty} \p{ \dgh(M^{n,1}_{\lambda},M^n) \geq \eps} =0. 
\end{equation}
This is enough to prove a version of our main result for the metric spaces without their measures.  In Lemma~\ref{lem:mst-lim}, below, we strengthen this statement. 
Let $\hat{M}^{n,1}_{\lambda}$ be the measured metric space obtained from $M^{n,1}_{\lambda}$ by rescaling so that the total mass is one (in $M^{n,1}_{\lambda}$ we gave each vertex mass $n^{-2/3}$; now we give each vertex mass $|V(\mathbb{M}^{n,1}_{\lambda})|^{-1}$). 
We show that for any $\eps > 0$, 
\begin{equation}\label{eq:mst-lim}
  \lim_{\lambda \to \infty} \limsup_{n \to \infty} \p{ \dghp(\hat{M}^{n,1}_{\lambda},M^n) \geq \eps} =0. 
\end{equation}
Since $(\cM,\dghp)$ is a complete, separable space, the so-called
principle of accompanying laws 
entails
that
$$M^n\convdist \mathscr{M}$$ 
in the space $(\cM,\dghp)$ for some limiting random measured metric space
$\mathscr{M}$ which is thus the scaling limit of the minimum spanning
tree on the complete graph. Furthermore, still as a consequence of the
principle of accompanying laws, $\mathscr{M}$ is also the limit in
distribution of $\mathscr{M}^1_\lambda$ as $\lambda\to\infty$ in the
space $(\cM,\dghp)$.

For fixed $\lambda \in \R$, we will see that each component of
$\mathscr{M}_{\lambda}$ is almost surely binary. Since $\mathscr{M}$
is compact and (if the measure is ignored) is an increasing limit of
$\mathscr{M}^1_{\lambda}$ as $\lambda \to \infty$, it will follow that
$\mathscr{M}$ is almost surely binary.

To prove that the mass measure is concentrated on the leaves of
$\mathscr{M}$, we use a result of {\L}uczak \cite{luczak90component}
on the size of the largest component in the barely supercritical
regime. This result in particular implies that for all $\epsilon > 0$,
\[
\lim_{\lambda \to \infty} \limsup_{n \to \infty} \p{\left|
    \frac{|V(\mathbb{M}^{n,1}_{\lambda})|}{2\lambda n^{2/3}} -
    1\right| > \eps} = 0.
\]
Since $\mathbb{M}^{n,1}_{\infty}$ has $n$ vertices, it follows
that for any $\lambda \in \R$, the proportion of the mass of
$M^{n,1}_{\infty}$ ``already present in $M^{n,1}_{\lambda}$'' is
asymptotically negligible. But (\ref{eq:abbr09b}) tells us that for
$\lambda$ large, with high probability every point of
$M^{n,1}_{\infty}$ not in $M^{n,1}_{\lambda}$ has distance $o_{\lambda
  \to \infty}(1)$ from a point of $M^{n,1}_{\lambda}$, so has distance
$o_{\lambda \to \infty}(1)$ from a leaf of $M^{n,1}_{\infty}$. Passing
this argument to the limit, it will follow that $\mathscr{M}$ almost
surely places all its mass on its leaves.

The statement on the Minkowski dimension of $\mathscr{M}$ depends
crucially on an explicit description of the components of
$\mathscr{G}_\lambda$ from \cite{ABBrGo09b}, which allows us to estimate
the number of balls needed to cover $\mathscr{M_\lambda^1}$. Along
with a refined version of \eqref{eq:abbr09b}, which yields an
estimate of the distance between $\mathscr{M}^1_\lambda$ and
$\mathscr{M}$, we are able to obtain bounds on the
covering number of $\mathscr{M}$.

This completes our overview, and we now proceed
with a brief discussion of related work, before turning to details.

\subsection{Related work}\label{sec:history}
In the majority of the work on convergence of MST's, inter-point
distances are chosen so that the edges of the MST have constant
average length (in all the models discussed above, the average edge
length was $o(1)$).  For such weights, the limiting object is
typically a non-compact infinite tree or forest. As detailed above, the study
bifurcates into the ``geometric'' case in which the points lie in a
Euclidean space $\R^d$, and the ``mean-field'' case where the
underlying graph is $K_n$ with i.i.d\ edge weights. In both cases, a
standard approach is to pass directly to an infinite underlying graph
or point set, and define the minimum spanning tree (or forest)
directly on such a point set.

It is not {\em a priori} obvious how to define the minimum spanning
tree, or forest, of an infinite graph, as neither of the algorithms
described above are necessarily well-defined (there may be no smallest
weight edge leaving a given vertex or component). However, it is known
\cite{aldous92asymptotics} that given an infinite locally finite graph
$G=(V,E)$ and distinct edge weights $\mathbf{w} = \{w_e,e \in E\}$,
the following variant of Prim's algorithm is well-defined and builds a
forest, each component of which is an infinite tree.

\medskip
\noindent {\sc Invasion percolation}: for each $v \in V$, run Prim's
algorithm starting from $v$ and call the resulting set of edges
$E_v$. Then let $\mathrm{MSF}(G,\mathbf{w})$ be the graph with
vertices $V$ and edges $\bigcup_{v \in V} E_v$.  \medskip

The graph $\mathrm{MSF}(G,\mathbf{w})$ is also described
by the following rule, which is conceptually based on the coupling
between Kruskal's algorithm and the percolation process, described
above. For each $r > 0$, let $G_r$ be the subgraph with edges $\{e \in
E: w_e \le r\}$.  Then an edge $e = uv \in E$ with $w_e = r$ is an
edge of $\mathrm{MSF}(G,\mathbf{w})$ if and only if $u$ and $v$ are in
distinct components of $G_r$ and one of these components is finite.

The latter characterisation again allows the MSF to be studied by
coupling with a percolation process. This connection was exploited by
Alexander and Molchanov \cite{alexander94percolation} in their proof
that the MSF almost surely consists of a single, one-ended tree for
the square, triangular, and hexagonal lattices with i.i.d.\
Uniform$[0,1]$ edge weights and, later, to prove the same result for
the MSF of the points of a homogeneous Poisson process in $\R^2$
\cite{alexander1995percolation}. Newman \cite{newman1997topics} has
also shown that in lattice models in $\mathbb{R}^d$, the critical
percolation probability $\theta(p_c)$ is equal to $0$ if and only if
the MSF is a proper forest (contains more than one tree). Lyons, Peres
and Schramm \cite{lyons2006minimal} developed the connection with
critical percolation. Among several other results, they showed that if
$G$ is any Cayley graph for which $\theta(p_c(G))=0$, then the
component trees in the MSF all have one end almost surely, and that
almost surely every component tree of the MSF itself has percolation
threshold $p_c=1$. (See also \cite{timar06ends} for subsequent work on
a similar model.) For two-dimensional lattice models, more detailed
results about the behaviour of the so-called ``invasion percolation
tree'', constructed by running Prim's algorithm once from a fixed
vertex, have also recently been obtained
\cite{damron2008relations,damron2010outlets}.

In the mean-field case, one common approach is to study the MST or MSF
from the perspective of local weak convergence
\cite{aldous2004omp}. This leads one to investigate the minimum
spanning forest of Aldous' {\em Poisson-weighted infinite tree}
(PWIT). Such an approach is used implicitly in \cite{mcdiarmid97mst}
in studying the first $O(\sqrt{n})$ steps of Prim's algorithm on
$K_n$, and explicitly in \cite{addario2009invasion} to relate the
behaviour of Prim's algorithm on $K_n$ and on the PWIT. Aldous
\cite{aldous90mst} establishes a local weak limit for the tree
obtained from the MST of $K_n$ as follows. Delete the (typically
unique) edge whose removal minimizes the size of the component
containing vertex $1$ in the resulting graph, then keep only the
component containing $1$.

Almost nothing is known about {\em compact} scaling limits for whole
MST's. In two dimensions, Aizenman, Burchard, Newman and Wilson
\cite{aizenman1999scaling} have shown tightness for the family of
random sets given by considering the subtree of the MST connecting a
finite set of points (the family is obtained by varying the set of
points), either in the square, triangular or hexagonal lattice, or in
a Poisson process. They also studied the properties of subsequential
limits for such families, showing, among other results, that any
limiting ``tree'' has Hausdorff dimension strictly between $1$ and
$2$, and that the curves connecting points in such a tree are almost
surely H\"{o}lder continuous of order $\alpha$ for any $\alpha <
1/2$. Recently, Garban, Pete, and Schramm \cite{garban10mst} announced
that they have proved the existence of a scaling limit for the MST in
2D lattice models. The MST is expected to be invariant under scalings,
rotations and translations, but not conformally invariant, and to have
no points of degree greater than four.  In the mean-field case,
however, we are not aware of any previous work on scaling limits for
the MST. In short, the scaling limit $\mathscr{M}$ that we identify in
this paper appears to be a novel mathematical object. It is one of the
first MST scaling limits to be identified, and is perhaps the first
scaling limit to be identified for any problem from combinatorial
optimisation.

We expect $\mathscr{M}$ to be a universal object: the MST's of a wide
range of ``high-dimensional'' graphs should also have $\mathscr{M}$ as
a scaling limit. By way of analogy, we mention two facts. First, Peres
and Revelle~\cite{peresrevelle} have shown the following universality
result for {\em uniform} spanning trees (here informally stated). Let
$\{G_n\}$ be a sequence of vertex transitive graphs of size tending to
infinity. Suppose that (a) the uniform mixing time of simple random
walk on $G_n$ is $o(G_n^{1/2})$, and (b) $G_n$ is sufficiently
``high-dimensional'', in that the expected number of meetings between
two random walks with the same starting point, in the first
$|G_n|^{1/2}$ steps, is uniformly bounded. Then after a suitable
rescaling of distances, the spanning tree of $G_n$ converges to the
CRT in the sense of finite-dimensional distributions. Second, under a
related set of conditions, van der Hofstad and Nachmias
\cite{van2012hypercube} have very recently proved that the largest
component of critical percolation on $G_n$ in the barely supercritical
phase has the same scaling as in the Erd\H{o}s--R\'enyi graph process
(we omit a precise statement of their result as it is rather
technical, but mention that their conditions are general enough to
address the notable case of percolation on the hypercube). However, a
proof of an analogous result for the MST seems, at this time, quite
distant. As will be seen below, our proof requires detailed control on
the metric and mass structure of all components of the Kruskal process
in the critical window and, for the moment, this is not available for
any other models.

\section{Metric spaces and types of convergence}\label{sec:metrics}

The reader may wish to simply skim this section on a first reading, referring back to it as needed.

\subsection{Notions of convergence} 
\label{sec:notions-convergence}

\subsubsection*{Gromov--Hausdorff distance}

Given a metric space $(X,d)$, we write $[X,d]$ for the isometry class
of $(X,d)$, 
\nomenclature[Xd]{$[X,d]$}{Isometry class of the metric space $(X,d)$.} 
and frequently use the notation $\mathrm{X}$ for either
\nomenclature[X]{$\mathrm{X}$}{A metric space, possibly decorated with measures and/or points.}
$(X,d)$ or $[X,d]$ when there is no risk of ambiguity. 
For a metric space $(X,d)$ we write $\diam((X,d))=\sup_{x,y \in X} d(x,y)$,
\nomenclature[Diam]{$\diam((X,d))$}{Equal to $\sup_{x,y \in X} d(x,y)$.}
which may be infinite. 

Let $\mathrm X = (X,d)$ and $\mathrm X' = (X',d')$ be metric spaces.
If $C$ is a subset of $X\times X'$, the {\em distortion} $\dis(C)$ is
defined by
\[
\dis(C) = \sup\{|d(x,y)-d'(x',y')|: (x,x') \in C, (y,y') \in C\}.
\]
\nomenclature[Dis]{$\mathrm{dis}(C)$}{Distortion of the correspondence $C$; equal to $\sup\{\lvert d(x,y)-d'(x',y')\rvert: (x,x') \in C, (y,y') \in C\}$.}
A \emph{correspondence} $C$ between $X$ and $X'$ is a measurable
subset of $X \times X'$ such that for every $x \in X$, there exists
$x' \in X'$ with $(x,x') \in C$ and vice versa.  
Write $C(X,X')$ for
the set of correspondences 
between $X$ and $X'$. 
\nomenclature[Cxx]{$C(X,X')$}{Set of correspondences between $X$ and $X'$.}
The Gromov--Hausdorff
distance $\dgh(\rX,\rX')$ between the isometry classes of $(X,d)$ and
$(X',d')$ is
\[
\dgh(\mathrm X,\mathrm X') = \frac{1}{2}\inf\{ \dis(C): C \in C(X,X')
\},
\]
and there is a correspondence which achieves this infimum. 
\nomenclature[Dgh]{$\dgh(\mathrm X,\mathrm X')$}{Gromov--Hausdorff distance between $\rX$ and $\rX'$; equal to $\frac{1}{2}\inf\{ \dis(C): C \in C(X,X')\}$. See same section for $\dgh^k(\mathrm X, \mathrm X')$.}
(In fact,
since our metric spaces are assumed separable, the requirement that
the correspondence be measurable is not strictly necessary.) It can be
verified that $\dgh$ is indeed a distance and, writing $\mathring{\cM}$ for
the set of isometry classes of compact metric spaces, that
$(\mathring{\cM},\dgh)$ is itself a complete separable metric space.
\nomenclature[Mdgh]{$(\mathring{\cM},\dgh)$}{Set of isometry classes of compact metric spaces with GH distance; see same section for $(\cM^{(k)},\dgh^k)$.}

Let $(X,d,(x_1,\ldots,x_k))$ and $(X',d',(x'_1,\ldots,x'_k))$ be
metric spaces, each with an ordered set of $k$ distinguished points
(we call such spaces {\em $k$-pointed metric spaces})\footnote{When
  $k=1$, we simply refer to pointed (rather than $1$-pointed) metric
  spaces, and write $(X,d,x)$ rather than $(X,d,(x))$}. We say that
these two $k$-pointed metric spaces are {\em isometry-equivalent} if
there exists an isometry $\phi:X\to X'$ such that $\phi(x_i)=x'_i$ for
every $i\in \{1,\ldots,k\}$.  As
before, we write $[X,d,(x_1,\ldots,x_k)]$ for the isometry equivalence
class of $(X,d,(x_1,\ldots,x_k))$, and denote either by $\rX$ when
there is little chance of ambiguity.

The {\em $k$-pointed Gromov--Hausdorff distance} is defined as
\[
\dgh^k(\mathrm X, \mathrm X') = \frac{1}{2}\inf\left\{\dis(C) :C \in
  C(X,X') \text{ such that } (x_i,x'_i) \in C, 1 \leq i\leq
  k\right\}.
\]
Much as above, the space $(\cM^{(k)},\dgh^k)$ of isometry classes of
$k$-pointed compact metric spaces is itself a complete separable metric
space. 

\subsubsection*{The Gromov--Hausdorff--Prokhorov distance}

A \emph{compact measured metric space} is a triple $(X,d,\mu)$ where
$(X,d)$ is a compact metric space and $\mu$ is a (non-negative) finite
measure on $(X,\cB)$, where $\cB$ is the Borel $\sigma$-algebra on
$(X,d)$. 
\nomenclature[Xdmu]{$(X,d,\mu),[X,d,\mu]$}{$(X,d,\mu)$ is a measured metric space; $\mu$ is a finite measure on $X$. $[X,d,\mu]$ is its measured isometry-equivalence class.}
Given a measured metric space
$(X,d,\mu)$, a metric space $(X',d')$ and a measurable function $\phi:
X \to X'$, we write $\phi_* \mu$ for the push-forward of the measure
$\mu$ to the space $(X',d')$.  
\nomenclature[Phimu]{$\phi_* \mu$}{Push-forward of measure $\mu$ under a map $\phi$.}
Two compact measured metric spaces
$(X,d,\mu)$ and $(X',d',\mu')$ are called \emph{isometry-equivalent}
if there exists an isometry $\phi:(X,d)\to (X',d')$ such that
$\phi_*\mu=\mu'$. The isometry-equivalence class of $(X,d,\mu)$ will be
denoted by $[X,d,\mu]$. Again, both will often be denoted by
$\mathrm{X}$ when there is little risk of ambiguity. If 
$\rX=(X,d,\mu)$ then we write $\mathrm{mass}(\rX)=\mu(X)$. 
\nomenclature[Massx]{$\mathrm{mass}(\rX)$}{For a measured metric space $\rX=(X,d,\mu)$, equal to $\mu(X)$.}

There are several natural distances on compact measured metric spaces
that generalize the Gromov--Hausdorff distance, see for instance
\cite{evanswinter,villani09,miermont09,AbDeHo12}. The presentation we
adopt is still different from these references, but closest in spirit
to \cite{AbDeHo12} since we are dealing with arbitrary finite measures
rather than just probability measures. In particular, it induces the
same topology as the compact Gromov--Hausdorff--Prokhorov metric of
\cite{AbDeHo12}.

If $(X,d)$ and $(X',d')$ are two metric spaces, let $M(X,X')$ be the set of
finite non-negative Borel measures on $X\times X'$. 
\nomenclature[Mxx]{$M(X,X')$}{Set of finite non-negative Borel measures on $X\times X'$.}
We will denote by $p,p'$ the canonical projections from $X\times
X'$ to $X$ and $X'$.  

Let $\mu$ and $\mu'$ be finite non-negative Borel measures on $X$ and $X'$
respectively.  The {\em discrepancy} of $\pi\in M(X,X')$ with respect
to $\mu$ and $\mu'$ is the quantity
$$D(\pi;\mu,\mu')=\|\mu-p_*\pi\|+\|\mu'-p'_*\pi\|\, ,$$
where $\|\nu\|$ is the total variation of the signed measure
$\nu$. 
\nomenclature[Dpimu]{$D(\pi;\mu,\mu')$}{Discrepancy of $\pi \in M(X,X')$.}
Note in particular that $D(\pi;\mu,\mu')\geq |\mu(X)-\mu'(X')|$, 
by the triangle inequality and the fact that $\|\nu\|\geq |\nu(1)|$, where $\nu(1)$ is the total
mass of $\nu$. If $\mu$ and $\mu'$ are probability
distributions (or have the same mass), a measure $\pi\in M(X,X')$ with
$D(\pi;\mu,\mu')=0$ is a coupling of $\mu$ and $\mu'$ in the standard
sense.

Recall that the Prokhorov distance between two finite non-negative
Borel measures $\mu$ and $\mu'$ on the \emph{same} metric space $(X,d)$ is given by 
$$\inf\{\eps>0:\mu(F)\leq
\mu'(F^\eps)+\eps\mbox{ and }\mu'(F)\leq \mu(F^\eps)+\eps\mbox{ for
  every closed }F\subseteq X \}\, .$$ 
An alternative distance, which generates the same topology but more easily extends to the 
setting where $\mu$ and $\mu'$ are measures on different metric spaces, is given by 
\[
\inf\left\{\eps>0: D(\pi; \mu, \mu') < \eps,
\pi(\{(x,x') \in X \times X: d(x,x')\geq \eps\})< \eps \text{ for some $\pi \in M(X,X)$} 
\right\}\, .
\]
To extend this, we replace the condition on $\{(x,x') \in X \times X: d(x,x')\geq \eps\}$ by 
an analogous condition on the measure of the set of pairs lying outside the correspondence. 
More precisely, let $\rX = (X,d,\mu)$ and $\rX' = (X',d',\mu')$ be measured metric
spaces. The Gromov--Hausdorff--Prokhorov distance between $\rX$ and
$\rX'$ is defined as
\[
\dghp(\rX,\rX') =   \inf\left\{\frac{1}{2}\dis(C)\vee
  D(\pi;\mu,\mu')\vee \pi(C^c)
\right\}\, ,
\]
the infimum being taken over all $C \in C(\rX, \rX')$ and $\pi\in
M(X,X')$.  
\nomenclature[Dghp]{$\dghp(\rX,\rX')$}{Gromov--Hausdorff--Prokhorov distance between 
$\rX$ and $\rX'$; see the same section for $\dghp^{k,l}(\rX,\rX')$.}
Here and elsewhere we write $x\vee y = \max(x,y)$ (and,
likewise, $x \wedge y = \min(x,y)$).  

Just as for $\dgh$, it can be verified that
$\dghp$ is a distance and that writing $\cM$ for the set of measured
isometry-equivalence classes of compact measured metric spaces,
$(\cM,\dghp)$ is a complete separable metric space (see, e.g., \cite{AbDeHo12}). 
\nomenclature[Mdghp]{$(\cM,\dghp)$}{Set of measured isometry-equivalence classes of compact measured metric spaces, with GHP distance; see same section for $(\cM^{k,l},\dghp^{k,l})$.}

Note that $\dghp((X,d,0),(X',d',0))=\dgh((X,d),(X',d'))$. In other words, 
the mapping $[X,d]\mapsto [X,d,0]$ is an isometric embedding of
$(\mathring{\mathcal{M}},\dgh)$ into $\mathcal{M}$, and we will sometimes abuse 
notation by writing $[X,d]\in \mathcal{M}$.  Note also that
$$
  \dgh(\rX,\rX')\vee
|\mu(X)-\mu'(X')|\leq \dghp(\rX,\rX') \leq
\frac{1}{2}(\diam(\rX)+\diam(\rX'))\vee (\mu(X)+\mu'(X'))\, .
$$
In particular, if $\rZ$ is the ``zero'' metric space consisting of a
single point with measure $0$, then 
\begin{equation}
  \label{eq:4}
\dghp(\rX,\rZ)=\frac{\diam(\rX)}{2}\vee \mu(X)\, ,\qquad \text{ for every
}\rX=[X,d,\mu]\, .
\end{equation}

Finally, we can define an analogue of $\dghp$ for measured
isometry-equivalence class of spaces of the form
$(X,d,\mathbf{x},\boldsymbol{\mu})$ where
$\mathbf{x}=(x_1,\ldots,x_k)$ are points of $X$ and
$\boldsymbol{\mu}=(\mu_1,\ldots,\mu_l)$ are finite Borel measures on
$X$. 
If $(X,d,\mathbf{x},\boldsymbol{\mu}),(X',d',\mathbf{x}',\boldsymbol{\mu}')$
are such spaces, whose measured, pointed isometry classes are denoted
by $\rX,\rX'$, we let 
$$\dghp^{k,l}(\rX,\rX')=
\inf \left\{ \frac{1}{2}\dis(C)\vee
  \max_{1\leq j\leq l} \left (D(\pi_j;\mu_j,\mu'_j)\vee
    \pi_j(C^c)\right) \right\}
$$
where the infimum is over all $C\in C(X,X')$ such that $(x_i,x'_i)\in
C, 1\leq i\leq k$ and all $\pi_j\in M(X,X'),1\leq j\leq l$.  Writing
$\cM^{k,l}$ for the set of measured isometry-equivalence classes of
compact metric spaces equipped with $k$ marked points and $l$ finite
Borel measures, we again obtain a complete separable metric space
$(\cM^{k,l},\dghp^{k,l})$.  We will need the following fact, which is in
essence \cite[Proposition 10]{miermont09}, except that we have to
take into account more measures and/or marks. This is a minor modification
of the setting of \cite{miermont09}, and the proof is similar.

\begin{proposition}
  \label{sec:grom-hausd-prokh}
  Let $\rX_n=(X_n,d_n,\mathbf{x}_n,\boldsymbol{\mu}_n)$ converge to
  $\rX_\infty=(X_\infty,d_\infty,\mathbf{x}_\infty,\boldsymbol{\mu}_\infty)$
  in $\cM^{k,l}$, and assume that the first measure $\mu_n^1$ of
  $\boldsymbol{\mu}_n$ is a probability measure for every $n\in
  \N\cup\{\infty\}$. Let $y_n$ be a random variable with distribution $\mu^1_n$, and let
  $\tilde{\mathbf{x}}_n=(x_n^1,\ldots,x_n^k,y_n)$. Then
  $(X_n,d_n,\tilde{\mathbf{x}}_n,\boldsymbol{\mu}_n)$ converges in
  distribution to
  $(X_\infty,d_\infty,\tilde{\mathbf{x}}_\infty,\boldsymbol{\mu}_\infty)$ in 
$\cM^{k+1,l}$. 
\end{proposition}

\subsubsection*{Sequences of metric spaces}
We now consider a natural metric on certain 
sequences of measured metric spaces. 
For $p\geq 1$ and $\bX=(\rX_i,i\geq 1),\bX'=(\rX'_i,i\geq 1)$ in $\cM^\N$, we let
\[
\dist^p_{\ghp}(\bX, \bX') = 
\pran{ \sum_{i \ge 1} \mathrm{d}_{\mathrm{GHP}}
  \big(\rX_i, \rX_i'\big)^p
}^{1/p}\, .
\] 
If $\bX \in \cM^{n}$ for some $n \in \N$, we consider $\bX$ as an
element of $\cM^{\N}$ by appending to $\bX$ an infinite sequence of
copies of the ``zero'' metric space $\rZ$. This allows us to use
$\mathrm{dist}^p_{\ghp}$ to compare sequences of metric spaces with
different numbers of elements, and to compare finite sequences with
infinite sequences. In particular, let $\bZ=(\rZ,\rZ,\ldots)$, and 
\[
\LL_p = \left\{\bX \in \cM^{\N}~:~ \dist^p_{\ghp}(\bX,\bZ)<\infty\right\}, 
\]
so, by \eqref{eq:4}, $\bX\in \LL_p$ if and only if the sequences
$(\diam(\rX_i),i\geq 1)$ and $(\mu_i(X_i),i\geq 1)$ are in
$\ell^p(\N)$. 
\nomenclature[Lp]{$(\LL_p,\dist^p_{\ghp})$}{Set of sequences of measured metric spaces, 
with the distance $\dist^p_{\ghp}$.}
We leave the reader to check that $(\LL_p,\dist^p_{\ghp})$ is a complete separable metric
space.

\subsection{Some general metric
  notions}\label{sec:some-general-metric}

Let $(X,d)$ be a metric space. For $x\in X$ and $r\geq 0$, we let
$B_r(x)=\{y\in X:d(x,y)<r\}$ and $\overline{B}_r(x)=\{y\in
X:d(x,y)\leq r\}$. We say $(X,d)$ is {\em degenerate} if $|X|=1$. 
As regards metric spaces, we mostly follow \cite{burago01} for our
terminology.
 
\subsubsection*{Paths, length, cycles}

Let $\mathcal{C}([a,b],X)$ be the set of continuous functions from
$[a,b]$ to $X$, hereafter called {\em paths with domain $[a,b]$} or {\em paths from $a$ to $b$}. 
\nomenclature[Cabx]{$\mathcal{C}([a,b],X)$}{Set of paths from $a$ to $b$.}
The image of a path is called an {\em arc}; it is a {\em simple arc} if the
path is injective. 
\nomenclature[Arc]{arc}{The image of a path; see same section for simple arc.}
If $f\in \mathcal{C}([a,b],X)$, the length of $f$
is defined by
\[
\len(f)=\sup\bigg\{\sum_{i=1}^kd(f(t_{i-1}),f(t_i)):k\geq 1, t_0,
t_1, \ldots,t_k \in [a,b], t_0\leq t_1\leq \ldots\leq t_k\bigg\}\, .
\]
If $\len(f)<\infty$, then the function $\varphi:[a,b]\to [0,\len(f)]$
\nomenclature[Lenf]{$\mathrm{len}(f)$}{The length of path $f$.}
defined by $\varphi(t)=\len(f|_{[a,t]})$ is non-decreasing and
surjective. The function $f\circ\varphi^{-1}$, where $\varphi^{-1}$ is
the right-continuous inverse of $\varphi$, is easily seen to be continuous,
and we call it the path $f$ {\em parameterized by arc-length}.

The {\em intrinsic distance} (or {\em intrinsic metric}) associated with $(X,d)$ is the function $d_l$
defined by 
$$d_l(x,y)=\inf\{\len(f):f\in \mathcal{C}([0,1],X), f(0)=x,f(1)=y\}\, .$$
The function $d_l$ need not take finite values. 
\nomenclature[Dl]{$d_l$}{The intrinsic distance associated with $(X,d)$ or with a subset $Y \subset X$.}
When
it does, then it defines a new distance on $X$ such that $d\leq
d_l$. The metric space $(X,d)$ is called {\em intrinsic} if $d=d_l$. 
Similarly, if $Y \subset X$ then the {\em intrinsic metric on $Y$} is given by 
\[
d_l(x,y)=\inf\{\len(f):f\in \mathcal{C}([0,1],Y), f(0)=x,f(1)=y\}\, .
\]

Given $x,y \in X$, a {\em geodesic between $x$ and $y$} 
(also called a {\em shortest path between $x$ and $y$}) 
is an
isometric embedding $f:[a,b] \to X$ such that $f(a)=x$ and $f(b)=y$
(so that obviously $\len(f)=b-a=d(x,y)$).  
\nomenclature[Geodesic]{geodesic}{Definition in text; see same section for geodesic arc, geodesic space.}
In this case, 
we call the image $\im(f)$ a {\em geodesic arc between $x$ and~$y$}.
  
  A metric space $(X,d)$ is called a geodesic space if for any
two points $x,y$ there exists a geodesic between $x$ and $y$. A
geodesic space is obviously an intrinsic space. If $(X,d)$ is compact,
then the two notions are in fact equivalent. Also note that for every
$x$ in a geodesic space and $r>0$, $\overline{B}_r(x)$ is the closure
of $B_r(x)$. Essentially all metric spaces $(X,d)$ that we consider in
this paper are in fact compact geodesic spaces.

A path $f\in \mathcal{C}([a,b],X)$ is a {\em local geodesic} between
$x$ and $y$ if $f(a)=x$, $f(b)=y$, and for any $t \in [a,b]$ there is
a neighborhood $V$ of $t$ in $[a,b]$ such that $f|_V$ is a
geodesic. 
\nomenclature[Local]{local geodesic}{Definition in text.}
It is then straightforward that $b-a=\len(f)$.  (Our
terminology differs from that of \cite{burago01}, where this would be
called a geodesic. We also note that we do not require $x$ and $y$ to be distinct.)

An {\em embedded cycle} is the image of a
continuous injective function $f:\mathbb{S}_1\to X$, where
$\mathbb{S}_1=\{z\in \mathbb{C}:|z|=1\}$. 
\nomenclature[cycle]{cycle}{Embedded cycle: image of continuous injective 
$f:\mathbb{S}_1\to X$; see same section for acyclic and unicyclic metric spaces.}
The length $\len(f)$ is the
length of the path $g:[0,1]\to X$ defined by $g(t)=f(e^{2\mathrm{i}\pi t})$ for
$0\leq t\leq 1$. It is easy to see that this length depends only on
the embedded cycle $c=\im(f)$ rather than its particular
parametrisation. We call it the {\em length} of the embedded
cycle, and write $\len(c)$ for this length. A metric space with no embedded cycle is called {\em acyclic},
and a metric space with exactly one embedded cycle is called {\em
  unicyclic}.

\subsubsection*{$\R$-trees and $\R$-graphs}\label{sec:r-trees-r}

A metric space $\rX=(X,d)$ is an {\em $\R$-tree} if it is an acyclic
geodesic metric space.  
\nomenclature[Rtree]{$\mathbb{R}$-tree}{Acyclic geodesic metric space.}
If $(X,d)$ is an $\R$-tree then for $x
\in T$, the {\em degree} $\deg_X(x)$ of $x$ is the number of connected
components of $X\setminus \{x\}$.  
\nomenclature[Degree]{$\deg_X(x)$}{Degree of $x$ in $\R$-graph or $\R$-tree $X$.}
A {\em leaf} is a point of degree
$1$; we let $\mathcal{L}(\rX)$ be the set of leaves of $\rX$.  
\nomenclature[Lx]{$\mathcal{L}(\rX)$}{Set of leaves of $\rX$.}  

A metric space $(X,d)$ is an {\em $\R$-graph} if it is locally an
$\R$-tree in the following sense. Note that by definition an
$\R$-graph is connected, being a geodesic space.

\begin{definition}\label{def:r-graph}
A compact geodesic metric space $(X,d)$ is an {\em
    $\R$-graph} if for every $x\in X$, there exists $\eps>0$ such that 
  $(B_\eps(x),d|_{B_\eps(x)})$ is an $\R$-tree. 
  \nomenclature[Rgraph]{$\mathbb{R}$-graph}{See Definition~\ref{def:r-graph}.}
\end{definition}

Let $\rX=(X,d)$ be an $\R$-graph and fix $x\in X$. The degree of $x$,
denoted by $\deg_X(x)$ and with values in $\N\cup \{\infty\}$, is
defined to be the degree of $x$ in $B_{\eps}(x)$ for every $\eps$
small enough so that $(B_\eps(x),d)$ is an $\R$-tree, and this
definition does not depend on a particular choice of $\eps$. If
$Y\subset X$ and $x\in Y$, we can likewise define the degree
$\deg_Y(x)$ of $x$ in $Y$ as the degree of $x$ in the $\R$-tree
$(B_\eps(x)\cap Y(x))\setminus \{x\}$, where $Y(x)$ is the connected
component of $Y$ that contains $x$, for any $\eps$ small
enough. Obviously, $\deg_Y(x)\leq \deg_{Y'}(x)$ whenever $Y\subset
Y'$.

Let
\nomenclature[Skel]{$\mathrm{skel}(\rX)$}{Points of degree at least two in an $\R$-graph $\rX$.}
\[\mathcal{L}(\rX)=\{x\in X:\deg_X(x)=1\}\, ,\qquad \skel(\rX)=\{x\in
X:\deg_X(x)\geq 2\}\, .
\]
 An element of $\mathcal{L}(\rX)$ is called a
{\em leaf of $\rX$}, and the set $\skel(\rX)$ is called the {\em skeleton of
  $\rX$}.  A point with degree at least $3$ is called a {\em branchpoint of $\rX$}.  
  \nomenclature[Branch]{branchpoint}{Point of degree at least three in an $\R$-graph $\rX$.} 
  We let $k(\rX)$ be the set of branchpoints of $\rX$.
  \nomenclature[Kx]{$k(\rX)$}{Set of branchpoints of $\R$-graph $\rX$.}
  If $\rX$ is, in fact, an $\R$-tree, then $\skel(\rX)$ is the set of
points whose removal disconnects the space, but this is not true in
general. Alternatively, it is easy to see that
\[
\skel(\rX)=\mathop{\bigcup_{x,y \in X}}_{c \in \Gamma(x,y)} c\setminus\{x,y\} 
\]
where for $x,y \in X$, $\Gamma(x,y)$ denotes the collection of all
geodesic arcs between $x$ and $y$. Since $(X,d)$ is
separable, this may be re-written as a countable union, and so there
is a unique $\sigma$-finite Borel measure $\ell$ on $X$ with
$\ell(\im(g)) = \len(g)$ for every injective path $g$, and such that
$\ell(X\setminus \skel(\rX))=0$. The measure $\ell$ is the Hausdorff
measure of dimension $1$ on $\rX$, and we refer to it as
the \emph{length measure} on $X$. 
\nomenclature[L]{$\ell$}{Length measure on an $\R$-graph $\rX$.}
If $(X,d)$ is an $\R$-graph then the
set $\{x\in X:\deg_X(x)\geq 3\}$ is countable (as is
classically the case for compact $\R$-trees), and hence this set has measure zero
under $\ell$.

\begin{definition}\label{def:core}
  Let $(X,d)$ be an $\R$-graph. Its {\em core}, denoted by
  $\core(\rX)$, is the union of all the simple arcs having both
  endpoints in embedded cycles of $\rX$. If it is non-empty, then
  $(\core(\rX),d)$ is an $\R$-graph with no leaves.
  \nomenclature[core]{$\mathrm{core}(\rX)$}{See Definition~\ref{def:core}.}
\end{definition}

The last part of this definition is in fact a proposition, which is
stated more precisely and proved below as Proposition
\ref{sec:skeleton-core-kernel}.  Since the core of $X$ encapsulates
all the embedded cycles of $X$, it is intuitively clear that when we
remove $\core(\rX)$ from $X$, we are left with a family of
$\R$-trees. This can be formalized as follows.  Fix $x\in X\setminus
\core(\rX)$, and let $f$ be a shortest path from $x$ to $\core(\rX)$, 
i.e., a geodesic from $x$ to $y \in \core(\rX)$, where $y \in \core(\rX)$ 
is chosen so that $\len(f)$ is minimum.
(recall that $\core(\rX$) is a closed subspace of $X$). This shortest
path is unique, otherwise we would easily be able to construct an embedded cycle
$c$ not contained in $\core(\rX)$, contradicting the
definition of $\core(\rX)$. Let $\alpha(x)$ be the endpoint of this path not equal to
$x$, which is thus the unique point of $\core(\rX)$ that is closest to
$x$. By convention, we let $\alpha(x)=x$ if $x\in \core(\rX)$. We call
$\alpha(x)$ the point of attachment of $x$.
\nomenclature[Alphax]{$\alpha(x)$}{Point of attachment of $x$ to $\core(\rX)$.}

\begin{proposition}\label{sec:structure-r-graphs-6}
  The relation $x\sim y\iff \alpha(x)=\alpha(y)$ is an equivalence
  relation on $X$. If $[x]$ is the equivalence class of $x$, then $([x],d)$ is a
  compact $\R$-tree. The equivalence class $[x]$ of a point $x\in
  \core(\rX)$ is a singleton if and only if 
  $\deg_X(x)=\deg_{\core(\rX)}(x)$. 
\end{proposition}

\proof The fact that $\sim$ is an equivalence relation is obvious. Fix
any equivalence class $[x]$. Note that $[x]\cap \core(\rX)$ contains only
the point $\alpha(x)$, so that $[x]$ is
connected and acyclic by definition. Hence, any two points of $[x]$ are
joined by a unique simple arc (in $[x]$). 
This path is moreover a
shortest path for the metric $d$, because a path starting and ending
in $[x]$, and visiting $X\setminus [x]$, must pass at least twice through
$\alpha(x)$ (if this were not the case, we could find an embedded cycle not
contained in $\core(\rX)$). The last statement is easy and left to the
reader. \endproof

\begin{corollary}\label{sec:r-trees-r-1}
If $(X,d)$ is an $\R$-graph, then $\core(\rX)$ is the maximal
closed subset of $X$ having only points of degree greater than or
equal to $2$. 
\end{corollary}

\proof If $Y$ is closed and strictly contains $\core(\rX)$, then we can find $x\in Y$ such
that $d(x,\core(\rX))=d(x,\alpha(x))>0$ is maximal. Then $Y\cap [x]$
is included in the set of points $y\in [x]$ such that the geodesic
arc from $y$ to $\alpha(x)$ does not pass through $x$. This set is an
$\R$-tree in which $x$ is a leaf, so $\deg_Y(x)\leq 1$. 
\endproof

Note that this characterisation is very close to the definition of the
core of a (discrete) graph. Another important structural component is
$\conn(\rX)$, the set of points of $\core(\rX)$ such that $X\setminus
\{x\}$ is connected. 
\nomenclature[conn]{$\mathrm{conn}(\rX)$}{Set of points of $\core(\rX)$ such that $X\setminus
\{x\}$ is connected.}
Figure \ref{fig:Rgraph} summarizes the preceding
definitions. The space $\conn(\rX)$ is not connected or closed in
general.  Clearly, a point of $\conn(\rX)$ must be contained in an
embedded cycle of $X$, but the converse is not necessarily true. A
partial converse is as follows.

\begin{proposition}\label{sec:skeleton-core-kernel-2}
  Let $x\in \core(\rX)$ have degree $\deg_X(x)=2$ and suppose $x$ is 
  contained in an embedded cycle of $X$. Then $x\in \conn(\rX)$.
\end{proposition}

\proof Let $c$ be an embedded cycle containing $x$. Fix $y, y'\in
X\setminus\{x\}$, and let $\phi,\phi'$ be geodesics from
$y,y'$ to their respective closest points $z,z'\in c$. Note that
$z$ is distinct from $x$ because otherwise, $x$ would have degree
at least $3$. Likewise, $z'\neq x$. 

Let $\phi''$ be a parametrisation of the arc of $c$ between $z$ and
$z'$ that does not contain $x$, then the concatenation of $\phi,\phi'$
and the time-reversal of the path $\phi''$ is a path from $y$ to $y'$,
not passing through $x$. Hence, $X\setminus \{x\}$ is connected.
\endproof

Let us now discuss the structure of $\core(\rX)$. Equivalently, we
need to describe $\R$-graphs with no leaves, because such graphs are
equal to their cores by Corollary \ref{sec:r-trees-r-1}.

A {\em graph with
  edge-lengths} is a triple $(V,E,(l(e),e\in E))$ where $(V,E)$ is a
finite connected multigraph, and $l(e)\in (0,\infty)$ for every $e\in
E$. With every such object, one can associate an $\R$-graph without
leaves, which is the {\em metric graph} obtained by viewing the edges
of $(V,E)$ as segments with respective lengths $l(e)$. Formally, this
$\R$-graph is the metric gluing of disjoint copies $Y^e$ of the real
segments $[0,l(e)],e\in E$ according to the graph structure of
$(V,E)$. We refer the reader to \cite{burago01} for
details on metric gluings and metric graphs. In Section
\ref{sec:RtreesRgraphs}, we will prove the following result.

\begin{theorem}\label{sec:structure-r-graphs-4}
  An $\R$-graph with no leaves is either a cycle, or is the metric
  gluing of a finite connected multigraph with edge-lengths in which
  all vertices have degree at least $3$. The associated multigraph,
  without the edge-lengths, is called the kernel of $X$, and denoted
  by $\ker(\rX)=(k(\rX),e(\rX))$.
  \nomenclature[Kerx]{$\ker(\rX)$}{Kernel $(k(\rX),e(\rX))$ of $\R$-graph $\rX$; see also Section~\ref{sec:core-metric-gluing}.}
\end{theorem}
The precise definition of $\ker(\rX)$, and the proof of Theorem~\ref{sec:structure-r-graphs-4}, both appear in Section~\ref{sec:core-metric-gluing}.

\begin{figure}[htb!]
\begin{center}
\includegraphics[scale=.9]{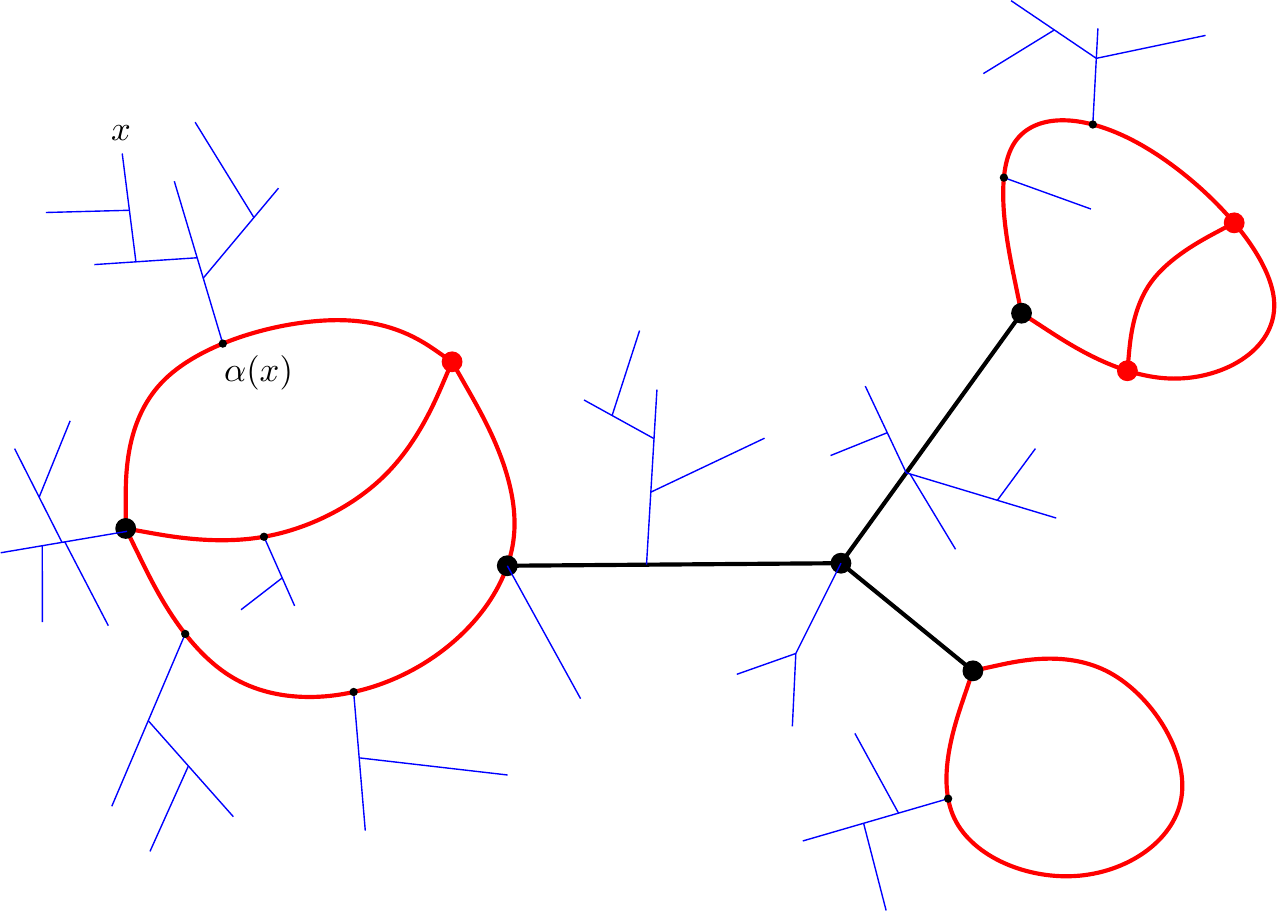}
\end{center}
\caption{An example of an $\R$-graph $(X,d)$, emphasizing the structural
  components. $\core(\rX)$ is in thick line (black and red), with
  $\conn(\rX)$ in red. The subtrees hanging from
  $\core(\rX)$ are in thin blue line. Kernel vertices are
  represented as large dots. An
  example of the projection $\alpha:X\to\core(\rX)$ is provided.}
\label{fig:Rgraph}
\end{figure}

For a connected multigraph $G=(V,E)$, the {\em surplus} $s(G)$ is
$|E|-|V|+1$. 
\nomenclature[Sg]{$s(G),s(\rX)$}{Surplus of $G$ and of $\rX$.}
For an \rg $(X,d)$, we let $s(\rX)=s(\ker(\rX))$ if
$\ker(\rX)$ is non-empty.  Otherwise, either $(X,d)$ is an $\R$-tree
or $\core(\rX)$ is a cycle. In the former case we set $s(\rX)=0$; in
the latter we set $s(\rX)=1$. Since the degree of every vertex in
$\ker(\rX)$ is at least $3$, we have $2|e(\rX)|=\sum_{v\in
  k(\rX)}\deg(v)\geq 3|k(\rX)|$, and so if $s(\rX)\geq 1$ we have
\begin{equation}
  \label{eq:5}
  |k(\rX)|\leq 2s(\rX)-2 \, ,
\end{equation}
with equality precisely if $\ker(\rX)$ is $3$-regular.

\section{Cycle-breaking in discrete and continuous graphs}
\label{sec:cycle-break-discr}

\subsection{The cycle-breaking algorithm}
\label{sec:cycle-break-algor-1}

Let $G=(V,E)$ be a finite connected multigraph. Let $\conn(G)$ be the
set of of all edges $e\in E$ such that $G\setminus e=(V,E\setminus
\{e\})$ is connected. 

If $s(G)>0$, then $G$ contains at least one cycle and $\conn(G)$ is
non-empty.  In this case, let $e$ be a uniform random edge in
$\conn(G)$, and let $K(G,\cdot)$ be the law of the multigraph
$G\setminus e$. 
\nomenclature[Kxcdot]{$K(G,\cdot)$}{Cycle-breaking Markov kernel on finite multigraph $G$.}
If $s(G)=0$, then $K(G,\cdot)$ is the Dirac mass at
$G$. By definition, $K$ is a Markov kernel from the set of graphs with
surplus $s$ to the set of graphs with surplus $(s-1)\vee 0$. Writing $K^n$ 
for the $n$-fold application of the kernel $K$, we have that $K^n(G,\cdot)$ does not depend
on $n$ for $n\geq s(G)$. We define the kernel
$K^\infty(G,\cdot)$ to be equal to this common value: a graph has law
$K^\infty(G,\cdot)$ if it is obtained from $G$ by repeatedly removing 
uniform non-disconnecting edges.

\begin{proposition}\label{sec:cycle-break-algor}
  The probability distribution $K^\infty(G,\cdot)$ is the law of the
  minimum spanning tree of $G$, when the edges $E$ are given
  exchangeable, distinct random edge-weights. 
\end{proposition}

\begin{proof}
  We prove by induction on the surplus of $G$ the stronger statement
  that $K^\infty(G,\cdot)$ is the law of the minimum spanning tree of
  $G$, when the weights of $\conn(G)$ are given exchangeable, distinct
  random edge-weights. For $s(G)=0$ the result is obvious. 

  Assume the result holds for every graph of surplus $s$, and let $G$
  have $s(G)=s+1$.  Let $e$ be the edge of $\conn(G)$ with maximal
  weight, and condition on $e$ and its weight. Then, note that the
  weights of the edges in $\conn(G)\setminus \{e\}$ are still in
  exchangeable random order, and the same is true of the edges of
  $\conn(G\setminus e)$. By the induction hypothesis, $K^s(G\setminus
  e,\cdot)$ is the law of the minimum spanning tree of $G\setminus
  e$. But $e$ is not in the minimum spanning tree of $G$, because by
  definition we can find a path between its endpoints that uses only
  edges having strictly smaller weights. Hence, $K^s(G\setminus
  e,\cdot)$ is the law of the minimum spanning tree of $G$. On the
  other hand, by exchangeability, the edge $e$ of $\conn(G)$ with largest
  weight is uniform in $\conn(G)$, so the unconditional law of a
  random variable with law $K^s(G\setminus e,\cdot)$ is  
  $K^{s+1}(G,\cdot)$.
\end{proof}

\subsection{Cutting the cycles of an 
  $\R$-graph}\label{sec:breaking-cycles-r}

There is a continuum analogue of the cycle-breaking algorithm in the
context of $\R$-graphs, which we now explain. Recall that $\conn(\rX)$
is the set of points $x$ of the $\R$-graph $\rX=(X,d)$ such that $x\in
\core(\rX)$ and $X\setminus \{x\}$ is connected. For $x\in
\conn(\rX)$, we let $(X_x,d_x)$ be the space $\rX$ ``cut at
$x$''. 
\nomenclature[Xxdx]{$(X_x,d_x)$}{$\rX=(X,d)$ or $\rX=(X,d,\mu)$ cut at the point $x \in X$; see also Section~\ref{sec:cutting-procedure}.}
Formally, it is the metric completion of $(X\setminus
\{x\},d_{X\setminus \{x\}})$, where $d_{X\setminus\{x\}}$ is the
intrinsic distance: $d_{X\setminus \{x\}}(y,z)$ is the
minimal length of a path from $y$ to $z$ that does not
visit $x$.
\begin{definition}\label{sec:cutting-cycles-an}
\nomenclature[Safely]{safely pointed}{See Definition~\ref{sec:cutting-cycles-an}.}
A point  $x\in X$ in a measured $\R$-graph $\rX=(X,d,\mu)$ is a
{\em regular point} if $x\in\conn(\rX)$, and moreover 
$\mu(\{x\})=0$ and $\deg_X(x)=2$. A marked space $(X,d,x,\mu)\in
\cM^{1,1}$, where $(X,d)$ is an $\R$-graph and $x$ is a regular point,
is called {\em safely pointed}. We say that a pointed $\R$-graph $(X,d,x)$ is safely pointed 
if $(X,d,x,0)$ is safely pointed. 
\end{definition}
If $x$ is a regular point then $\mu$ induces a measure (still denoted
by $\mu$) on the space $\rX_x$ with the same total mass. We will give a precise description of
the space $\rX_x=(X_x,d_x,\mu)$ in Section
\ref{sec:cutting-procedure}: in particular, it is a measured
$\R$-graph with $s(\rX_x)=s(\rX)-1$. 

Note that if $s(\rX)>0$ and if 
\nomenclature[L]{$L$}{Length measure restricted to $\conn(\rX)$.}
$$\Ell=\ell(\cdot\cap \conn(\rX))\, $$
is the length measure restricted to $\conn(\rX)$, then $\Ell$-almost
every point is regular. Also, $\Ell$ is a finite measure by Theorem~\ref{sec:structure-r-graphs-4}. Therefore, it makes sense to let
$\mathcal{K}(\rX,\cdot)$ be the law of $\rX_x$, where $x$ is a random point of 
$X$ with law $\Ell/\Ell(\conn(\rX))$.
\nomenclature[Kxcdot]{$\mathcal{K}(\rX,\cdot)$}{Cycle-breaking Markov kernel on $\R$-graph $\rX$. See same section and also Section~\ref{sec:cuttingrgs} for $\mathcal{K}$, $\mathcal{K}^n$ and $\mathcal{K}^{\infty}$.}
  If $s(\rX)=0$ we let
$\mathcal{K}(\rX,\cdot)=\delta_{\{\rX\}}$. Again, $\mathcal{K}$ is a Markov kernel from the set of 
measured $\R$-graphs with surplus $s$ to the set of measured $\R$-graphs of
surplus $(s-1)\vee0$, and $\mathcal{K}^n(\rX,\cdot)=\mathcal{K}^{s(\rX)}(\rX,\cdot)$ for
every $n\geq s(\rX)$: we denote this by $\mathcal{K}^\infty(\rX,\cdot)$.

In Section \ref{sec:ckcc} we will give details of the proofs of the aforementioned properties, 
as well as of the following crucial result.  For
$r\in (0,1)$ we let $\mathcal{A}_r$ be the set of measured $\R$-graphs
\nomenclature[Ar]{$\mathcal{A}_r$}{Set of ``$r$-uniformly elliptic $\R$-graphs''.}
with $s(\rX)\leq 1/r$ and whose core, seen as a graph with edge-lengths
$(k(\rX),e(\rX),(\ell(e),e\in e(\rX)))$, is such that 
$$\min_{e\in e(\rX)}\ell(e)\geq r\, ,\quad \mbox{ and }\quad \sum_{e\in e(\rX)}\ell(e)\leq 1/r$$
(if $s(\rX)=1$, this should be
understood as the fact that $\core(\rX)$ is a cycle with length in $[r,1/r]$.)

\begin{theorem}\label{sec:breaking-cycles-r-1}
Fix $r\in (0,1)$. Let $(X^n,d^n,\mu^n)$ be a sequence of measured
$\R$-graphs in $\mathcal{A}_r$, converging
as $n\to\infty$ to $(X,d,\mu)\in \mathcal{A}_r$ in
$(\mathcal{M},\dghp)$. Then $\mathcal{K}^\infty(\rX^n,\cdot)$ converges weakly
to $\mathcal{K}^\infty(\rX,\cdot)$. 
\end{theorem}

\subsection{A relation between the discrete and continuum
  procedures}\label{sec:relation-between-two}

We can view any finite connected multigraph $G=(V,E)$ as a metric
space $(V,d)$, where $d(u,v)$ is the least number of edges in any
chain from $u$ to $v$.  We may also consider the metric graph
$(m(G),d_{m(G)})$ associated with $G$ by treating edges as segments of
length $1$ (this is sometimes known as the \emph{cable system} for the graph $G$ \cite{varopoulos1985}). Then $(m(G),d_{m(G)})$ is an $\R$-graph. Note that
$\dgh((V,d),(m(G),d_{m(G)})) < 1$ and, in fact, $(m(G),d_{m(G)})$ contains an
isometric copy of $(V,d)$.  Also, temporarily writing $H$ for the {\em
  graph-theoretic} core of $G$, that is, the maximal subgraph of $G$
of minimum degree two, it is straightforwardly checked that
$\core(m(G))$ is isometric to $(m(H),d_{m(H)})$.

Conversely, let $(X,d)$ be an
$\R$-graph, and let $S_X$ be the set of points in $X$ with degree at
least three. We say that $(X,d)$ {\em has integer lengths} if all
local geodesics between points in $S_X$ have lengths in $\Z_+$.  Let
\[
v(\rX)= \{x \in X: d(x,S_X) \in \Z_{+}\},
\]
and note that if $(X,d)$ is compact and has integer lengths then
necessarily $|S_X| < \infty$ and $|v(\rX)| < \infty$.  The removal of
all points in $v(\rX)$ separates $X$ into a finite collection of
paths, each of which is either an open path of length one between two
points of $v(\rX)$, or a half-open path of length strictly less than
one between a point of $v(\rX)$ and a leaf.  Create an edge between
the endpoints of each such {\em open} path, and call the collection of
such edges $e(\rX)$. Then let
\[
g(\rX) = (v(\rX),e(\rX)); 
\]
we call the multigraph $g(\rX)$ the {\em graph corresponding to $\rX$} (see Figure~\ref{fig:graphspace}). 

\begin{figure}[htb]
\begin{center}$
\begin{array}{ccc}
\includegraphics[width=0.3\textwidth]{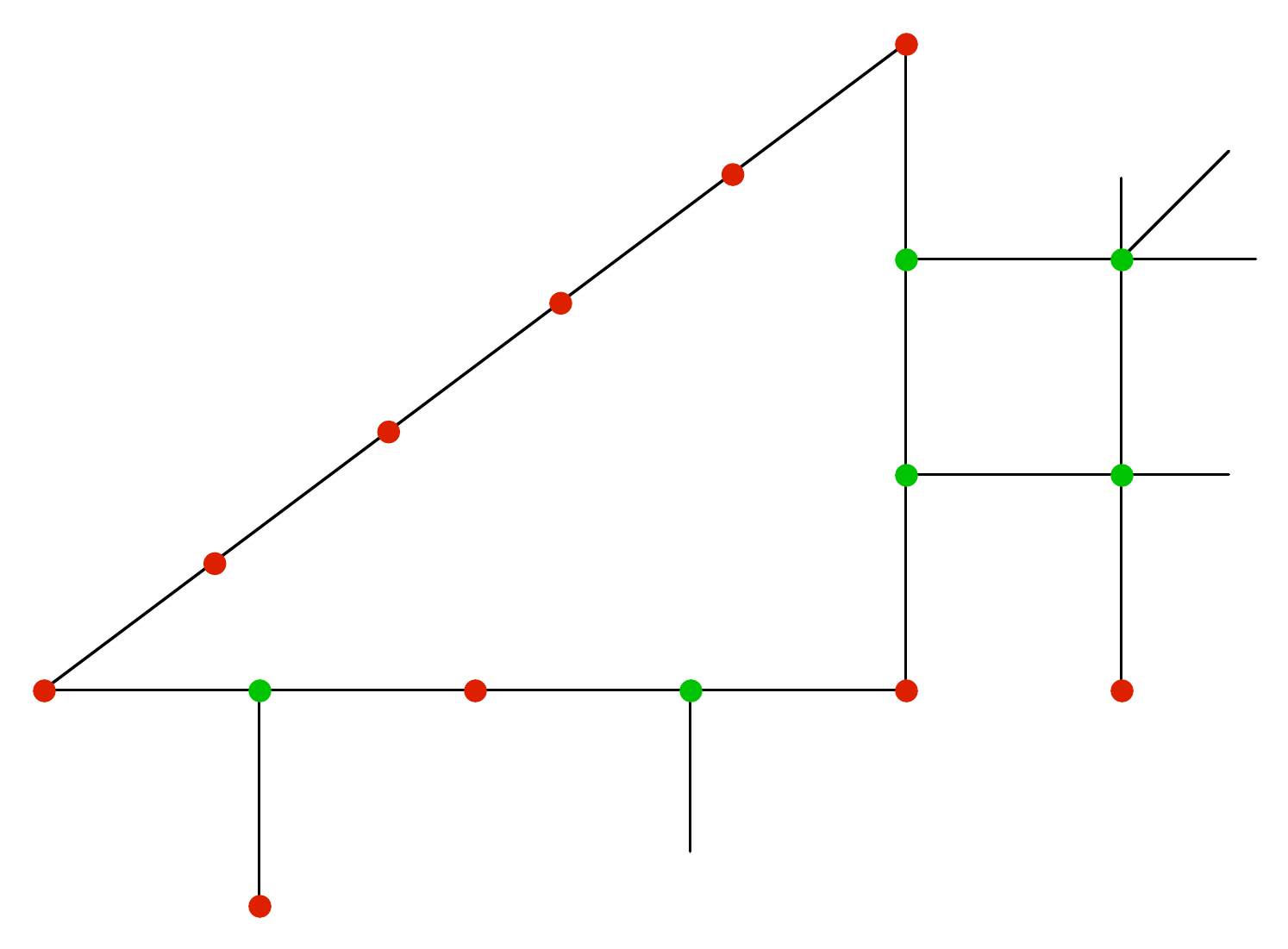} &
\includegraphics[width=0.3\textwidth]{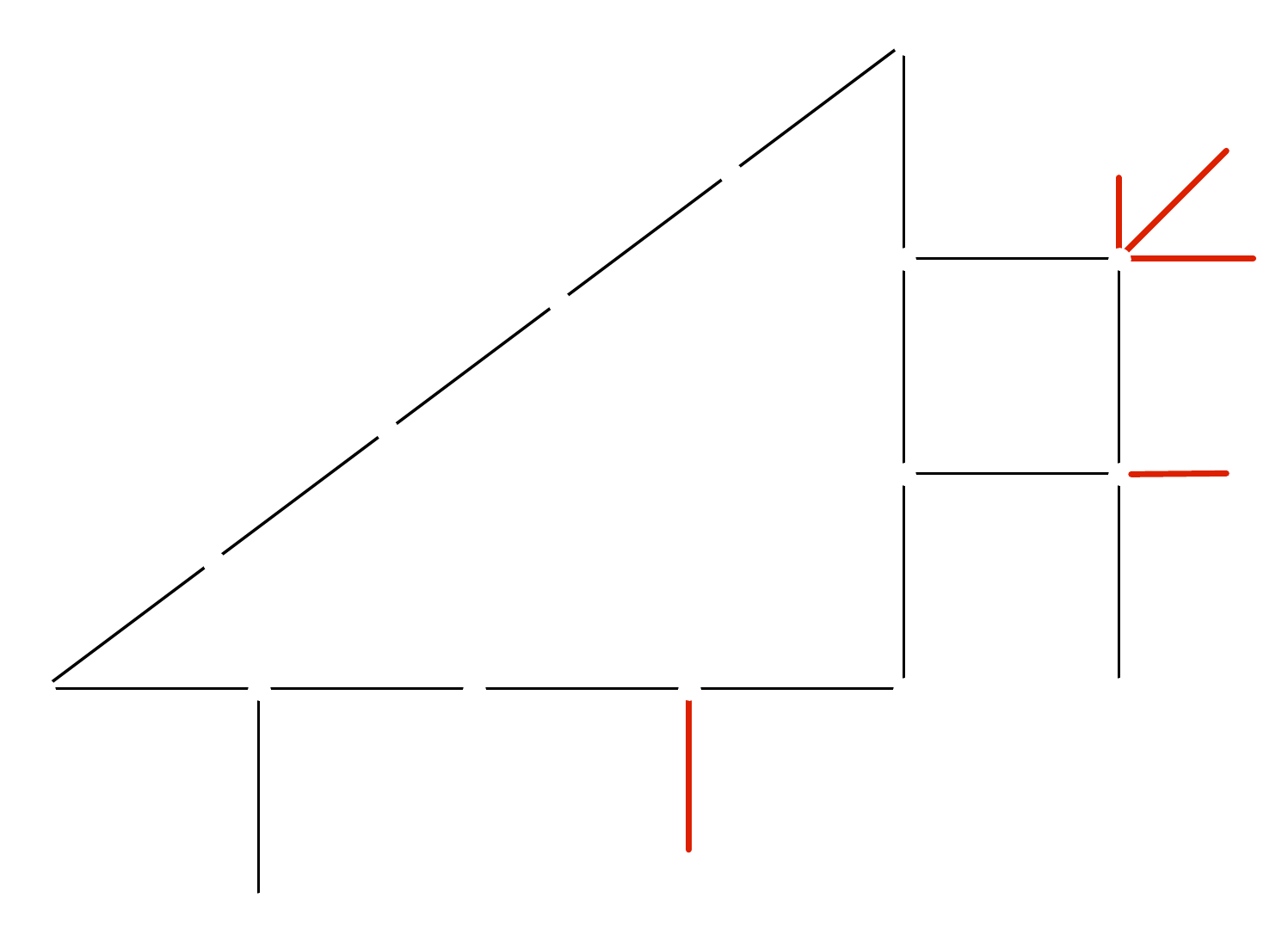} & 
\includegraphics[width=0.3\textwidth]{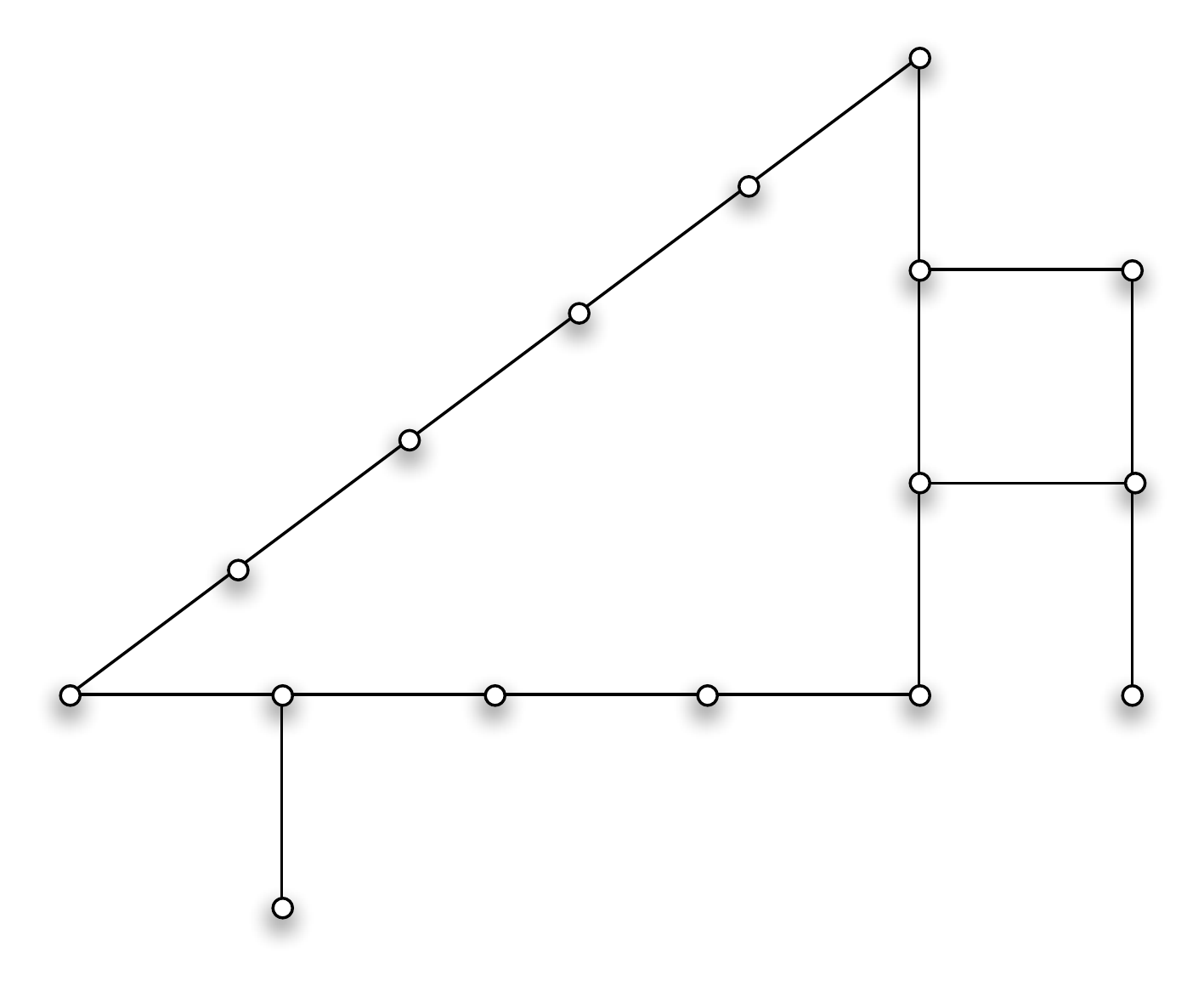}
\end{array}$
\end{center}
\caption{Left: an \rg with integer lengths. The points of degree at least three are marked green, and the remaining points of $v(\rX)$ are marked red. Centre: the collection of paths after the points of $v(\rX)$ are removed. The paths with non-integer length are drawn in red. Right: the graph $g(\rX)$. 
  \label{fig:graphspace}}
\end{figure}

Now, fix an \rg $(\rX,d)$ which has integer lengths and surplus
$s(\rX)$. Let $x_1,\dots, x_{s(\rX)}$ be the points sampled by the
successive applications of $\mathcal{K}$ to $\rX$: given $x_1,\ldots,x_i$, the
point $x_{i+1}$ is chosen according to $\Ell/\Ell(\rX)$ on
$\conn(\rX_{x_1,\ldots,x_i})$, where $\rX_{x_1,\ldots,x_i}$ is the
space $\rX$ cut successively at $x_1,x_2,\ldots,x_i$. Note that $x_i$
can also naturally be seen as a point of $X$ for $1\leq i\leq
s(\rX)$. Since the length measure of $v(\rX)$ is $0$, almost surely
$x_i\neq v(\rX)$ for all $1\le i\le s(\rX)$. Thus, each point $x_i$,
$1\le i\le s(\rX)$, falls in a path component of $\core(\rX)\setminus
v(\rX)$ which itself corresponds uniquely to an edge in $e_i\in
e(\rX)$. Note that the edges $e_i$, $1\le i\le s(\rX)$, are distinct by
construction.  Then let $g_0(\rX)=g(\rX)$, and for $1 \leq i \leq
s(\rX)$, write
\[
g_i(\rX) = (v(\rX),e(\rX)\setminus \{e_1,\ldots,e_i\}). 
\]
By construction, the graph $g_{i}(\rX)$ is connected and has surplus
precisely $s(\rX)-i$, and in particular $g_{s(\rX)}(\rX)$ is a
spanning tree of $g(\rX)$. Let $\cut(\rX)$ be the random \rg
resulting from the application of $\mathcal{K}^\infty$, that is obtained by
cutting $\rX$ at the points $x_1,\dots,x_{s(\rX)}$ in our setting.
\nomenclature[CutX]{$\mathrm{cut}(\rX)$}{Random $\R$-graph with distribution $\mathcal{K}^\infty(\rX,\cdot)$.}
\begin{proposition}\label{prop:cut-close}
We have $\dgh(\cut(\rX),g_{s(\rX)}(\rX)) < 1$. 
\end{proposition}
\begin{proof}
  First, notice that $g_{s(\rX)}(\rX)$ and $g(\cut(\rX))$ are
  isomorphic as graphs, so isometric as metric spaces. Also, as noted
  in greater generality at the start of the subsection, we
  automatically have $\dgh(\cut(\rX),g(\cut(\rX))) < 1$.
\end{proof}

\begin{proposition}\label{prop:cut-dist}
  The graph $g(\cut(\rX))$ is identical in distribution to the
  minimum-weight spanning tree of $g(\rX)$ when the edges of $e \in
  e(\rX)$ are given exchangeable, distinct random edge weights.
\end{proposition}
\begin{proof}
  When performing the discrete cycle-breaking on $g(\rX)$, the set of
  edges removed from $g(\rX)$ is identical in distribution to the set
  $\{e_1,\ldots,e_{s(\rX)}\}$ of edges that are removed from $g(\rX)$
  to create $g_{s(\rX)}(\rX)$, so $g_{s(\rX)}(\rX)$ has the same
  distribution as the minimum spanning tree by Proposition~\ref{sec:cycle-break-algor}.  Furthermore, as noted in the proof of
  the preceding proposition, $g_{s(\rX)}(\rX)$ and $g(\cut(\rX))$ are
  isomorphic.
\end{proof}

\subsection{Gluing points in $\R$-graphs} \label{subsec:gluing}

We end this section by mentioning the operation of gluing, which in
a vague sense is dual to the cutting operation. If $(X,d,\mu)$ is an
$\R$-graph  and $x,y$ are two distinct points of $X$, we let $\rX^{x,y}$
be the quotient metric space \cite{burago01} of $(X,d)$ by the smallest
equivalence relation for which $x$ and $y$ are equivalent. This space
is endowed with the push-forward of $\mu$ by the canonical
projection $p$. It is not difficult to see that $\rX^{x,y}$ is again an
$\R$-graph, and that the class of the point $z=p(x)=p(y)$ has degree
$\deg_{X^{x,y}}(z)=\deg_X(x)+\deg_X(y)$. Similarly, if
$\mathcal{R}$ is a finite set of unordered pairs $\{x_i,y_i\}$ with $x_i\neq
y_i$ in $X$, then one can identify $x_i$ and $y_i$ for each $i$, resulting
in an $\R$-graph $\rX^\mathcal{R}$. 

\section{Convergence of the MST} \label{sec:proofs}

We are now ready to state and prove the main results of this paper.
We begin by recalling from the introduction that we write $\MM^n$ for
the MST of the complete graph on $n$ vertices and $M^n$ for the
measured metric space obtained from $\MM^n$ by rescaling the graph
distance by $n^{-1/3}$ and assigning mass $1/n$ to each vertex.

\subsection{The scaling limit of the Erd\H{o}s--R\'enyi random graph}\label{sec:scaling-limit-erdhos}

Recall that $\mathbb{G}(n,p)$ is the Erd\H{o}s--R\'enyi random graph.
For $\lambda \in \R$, we write
\nomenclature[Glambdan]{$\mathbb{G}_{\lambda}^n,G_{\lambda}^n,\mathscr{G}_{\lambda}$}{Sequence of components $(\mathbb{G}^{n,i}_{\lambda},i \ge 1)$ of $\mathbb{G}(n,1/n + \lambda/n^{4/3})$; see same section for measured metric space version $G_{\lambda}^n$ and limit $\mathscr{G}_{\lambda}$.}
\[
\mathbb{G}_{\lambda}^n = (\mathbb{G}^{n,i}_{\lambda},i \ge 1)
\]
for the components of $\mathbb{G}(n,1/n + \lambda/n^{4/3})$ listed in
decreasing order of size (among components of equal size, list
components in increasing order of smallest vertex label, say). For
each $i \ge 1$, we then write $G_{\lambda}^{n,i}$ for the measured
metric space obtained from $\mathbb{G}^{n,i}_{\lambda}$ by rescaling
the graph distance by $n^{-1/3}$ and giving each vertex mass
$n^{-2/3}$, and let
\[
G_{\lambda}^n = (G^{n,i}_{\lambda},i \ge 1). 
\]

In a moment, we will state a scaling limit result for $G_{\lambda}^n$;
before we can do so, however, we must introduce the limit sequence of
measured metric spaces $\mathscr{G}_{\lambda} =
(\mathscr{G}^i_{\lambda}, i \ge 1)$.  We will do this somewhat
briefly, and refer the interested reader to \cite{ABBrGo09b,ABBrGo09a} for more
details and distributional properties.

First, consider the stochastic process $(W_{\lambda}(t), t \ge 0)$ defined by
\nomenclature[Wlambda]{$W_{\lambda}$}{Brownian motion with parabolic drift.}
\[
W_{\lambda}(t) := W(t) + \lambda t - \frac{t^2}{2},
\]
where $(W(t), t \ge 0)$ is a standard Brownian motion.  Consider the
excursions of $W_{\lambda}$ above its running minimum; in other words,
the excursions of
\[
B_{\lambda}(t) := W_{\lambda}(t) - \min_{0 \le s \le t} W_{\lambda}(s)
\]
above 0.  We list these in decreasing order of length as
$(\varepsilon^1, \varepsilon^2, \ldots)$ where, for $i \ge 1$,
$\sigma^i$ is the length of $\varepsilon^i$.  (We suppress the 
$\lambda$-dependence in the notation for readability.)
For definiteness, we shift the origin of each excursion to 0, so that
$\varepsilon^i: [0, \sigma^i] \to \R_+$ is a continuous function such
that $\varepsilon^i(0) = e^i(\sigma^i) = 0$ and $\varepsilon^i(x) > 0$
otherwise.

Now for $i \ge 1$ and for $x,x' \in [0,\sigma^i]$, define a
pseudo-distance via
\[
\hat{d}^i(x,x') = 2 \varepsilon^i(x) + 2 \varepsilon^i(x') - 4 \inf_{x
  \wedge x' \le t \le x \vee x'} \varepsilon^i(t).
\]
Declare that $x \sim x'$ if $\hat{d}^i(x,x') = 0$, so that $\sim$ is
an equivalence relation on $[0,\sigma^i]$.  Now let $\mathcal{T}^i =
[0,\sigma^i]/\!\!\sim$ and denote by $\tau^i: [0,\sigma^i] \to
\mathcal{T}^i$ the canonical projection. Then $\hat{d}^i$ induces a
distance on $\mathcal{T}^i$, still denoted by $\hat{d}^i$, and it is
standard (see, for example, \cite{legall06rrt}) that $(\mathcal{T}^i,
\hat{d}^i)$ is a compact $\R$-tree.  Write $\hat{\mu}^i$ for the
push-forward of Lebesgue measure on $[0,\sigma^i]$ by $\tau^i$, so
that $(\mathcal{T}^i, \hat{d}^i, \hat{\mu}^i)$ is a measured $\R$-tree
of total mass $\sigma^i$.

We now decorate the process $B_{\lambda}$ with the points of an
independent homogeneous Poisson process in the plane.  We can think of
the points which fall under the different excursions separately.  In
particular, to the excursion $\varepsilon^i$, we associate a finite
collection $\mathcal{P}^i = \{(x^{i,j}, y^{i,j}), 1 \le j \le s^i\}$
of points of $[0,\sigma^i] \times [0, \infty)$ which are the Poisson
points shifted in the same way as the excursion $\varepsilon^i$.  (For
definiteness, we list the points of $\cP^i$ in increasing order of first co-ordinate.)
Conditional on $\varepsilon^1, \varepsilon^2, \ldots$, the collections
$\mathcal{P}^1, \mathcal{P}^2, \ldots$ of points are independent.
Moreover, by construction, given the excursion $\varepsilon^i$, we have $s^i
\sim \text{Poisson}(\int_0^{\sigma^i} \varepsilon^i(t) \mathrm{d} t)$.
Let $z^{i,j} = \inf\{t \ge x^{i,j}: \varepsilon^i(t) = y^{i,j} \}$ and
note that, by the continuity of $\varepsilon^i$, $z^{i,j} < \sigma^i$
almost surely.  Let
\[
\mathcal{R}^i = \{\{\tau^i(x^{i,j}), \tau^i(z^{i,j})\}, 1 \le j \le s^i\}.
\]
Then $\mathcal{R}^i$ is a collection of unordered pairs of points in
the $\R$-tree $\mathcal{T}^i$.  We wish to glue these points together
in order to obtain an $\R$-graph, as in Section~\ref{subsec:gluing}.  We define a
new equivalence relation $\sim$ by declaring $x \sim x'$ in
$\mathcal{T}^i$ if $\{x,x'\} \in \mathcal{R}^i$.  Then let $\mathcal{X}^i$
be $\mathcal{T}^i /\!\!\sim$, let $d^i$ be the quotient metric
\cite{burago01}, and let $\mu^i$ be the push-forward of $\hat{\mu}^i$ to
$\mathcal{X}^i$.  Then set $\mathscr{G}_{\lambda}^i =(\mathcal{X}^i,
d^i, \mu^i)$ and $\mathscr{G}_{\lambda} =
(\mathscr{G}_{\lambda}^i, i \ge 1)$. We note that for each $i \ge 1$,
the measure $\hat{\mu}^i$ is almost surely concentrated on the leaves
of $\mathcal{T}^i$. As a consequence, $\mu^i$ is almost surely
concentrated on the leaves of $\mathcal{X}^i$.

Given an $\R$-graph $\rX$, write $r(\rX)$ for the minimal
length of a core edge in $\rX$.  
\nomenclature[Rx]{$r(\mathrm{X})$}{Minimal length of a core edge in $\rX$.}
Then $r(\rX) = \inf \{ d(u,v): u,v\in
k(\rX) \}$ whenever $\ker(\rX)$ is non-empty. We use the convention
that $r(\rX)=\infty$ if $\core(\rX)=\varnothing$ and $r(\rX)=\ell(c)$
if $\rX$ has a unique embedded cycle $c$.  Recall also that $s(\rX)$ denotes
the surplus of $\rX$.

\begin{theorem}\label{thm:gnp-converge}
Fix $\lambda \in \R$. Then as $n \to \infty$, we have the following joint convergence
\begin{align*}
G^n_{\lambda} & \convdist \mathscr{G}_{\lambda}\, , \\
(s(G^{n,i}_{\lambda}),i \ge 1) & \convdist (s(\mathscr{G}^i_{\lambda}),i \ge 1)\, , \mbox{ and} \\
(r(G^{n,i}_{\lambda}),i \ge 1) & \convdist (r(\mathscr{G}^i_{\lambda}),i \ge 1)\, .
\end{align*}
The first convergence takes place in the space $(\mathbb{L}_4,
\mathrm{dist}_{\ghp}^4)$.  The others are in the sense of
finite-dimensional distributions.
\end{theorem}

Let $\ell_2^{\downarrow} = \{x = (x_1, x_2, \ldots): x_1 \ge x_2 \ge
\ldots \ge 0, \sum_{i=1}^{\infty} x_i^2 < \infty\}$.  Corollary 2 of \cite{aldous97} gives the following joint convergence
\begin{equation}\label{eq:aldous_joint_convergence}
\begin{aligned}
  (\mathrm{mass}(G_{\lambda}^{n,i}), i \ge 1) & \convdist (\mathrm{mass}(\mathscr{G}_{\lambda}^i), i \ge 1), \text{ and} \\
  (s(G^{n,i}_{\lambda}),i \ge 1) & \convdist
  (s(\mathscr{G}^i_{\lambda}),i \ge 1),
\end{aligned}
\end{equation}
where the first convergence is in $(\ell_2^{\downarrow}, \|\cdot \|_2)$ and the second is in the sense of finite-dimensional distributions.  (Of course, $\mathrm{mass}(\mathscr{G}_{\lambda}^i) = \sigma^i$ and $s(\mathscr{G}_{\lambda}^i) = s^i$.)  Theorem 1 of \cite{ABBrGo09a} extends this to give that, jointly,
\begin{equation}\label{eq:abbrgo}
(G_{\lambda}^{n,i}, i \ge 1) \convdist (\mathscr{G}_{\lambda}^i, i \ge 1)
\end{equation}
in the sense of $\dist_{\gh}^4$, where for $\bX,\bY \in
\mathring{\mathcal{M}}^{\N}$, $\dist_{\gh}^4(\bX,\bY) =
\left(\sum_{i=1}^{\infty} \dgh(\rX^i, \rY^i)^4 \right)^{1/4}$.  We
need to improve this convergence from $\dist_{\gh}^4$ to
$\dist_{\ghp}^4$.  First we show that we can get GHP convergence
componentwise.  We do this in two lemmas.

\begin{lemma} \label{lem:measured_identifications} Suppose that
  $(\mathcal{T}, d, \mu)$ and $(\mathcal{T}', d', \mu')$ are measured
  $\R$-trees, that $\{(x_i, y_i), 1 \le i \le k\}$ are pairs of points
  in $\mathcal{T}$ and that $\{(x_i', y_i'), 1 \le i \le k\}$ are
  pairs of points in $\mathcal{T}'$.  
Then if $(\hat{\mathcal{T}}, \hat{d}, \hat{\mu})$ and
$(\hat{\mathcal{T}}', \hat{d}', \hat{\mu}')$ are the measured metric
spaces obtained by identifying $x_i$ and $y_i$ in $\mathcal{T}$ and
$x_i'$ and $y_i'$ in $\mathcal{T}'$, for all $1 \le i \le k$, we have
\[
\dghp((\hat{\mathcal{T}}, \hat{d}, \hat{\mu}), (\hat{\mathcal{T}}',
\hat{d}', \hat{\mu}')) \le 
(k+1)\, \dghp^{2k,1}((\mathcal{T}, d,\mathbf{x} ,\mu), (\mathcal{T}',
d',\mathbf{x}' ,\mu'))
\]
where $\mathbf{x}=(x_1,\ldots,x_k,y_1,\ldots,y_k)$ and similarly for
$\mathbf{x}'$. 
\end{lemma}

\begin{proof}
  Let $C$ and $\pi$ be a correspondence and a measure which realise
  the Gromov--Hausdorff--Prokhorov distance between $(\mathcal{T},
  d,\mathbf{x}, \mu)$ and $(\mathcal{T}', d',\mathbf{x}', \mu')$;
  write $\delta$ for this distance. 
  Note that by definition,
  $(x_i,x'_i)\in C$ and $(y_i,y'_i)\in C$ for $1\leq i\leq k$.  Now
  make the vertex identifications in order to obtain
  $\hat{\mathcal{T}}$ and $\hat{\mathcal{T}}'$; let $p: \mathcal{T}
  \to \hat{\mathcal{T}}$ and $p': \mathcal{T}' \to \hat{\mathcal{T}}'$
  be the corresponding canonical projections.  Then
\[
\hat{C} = \{(p(x), p'(x')): (x,x') \in \mathcal{R}'\}
\]
is a correspondence between $\hat{\mathcal{T}}$ and
$\hat{\mathcal{T}}'$.  Let $\hat{\pi}$ be the push-forward of the
measure $\pi$ by $(p,p')$.  Then $D(\hat{\pi};\hat{\mu},\hat{\mu}')
\le \delta$ and $\hat{\pi}(\hat{C}^c) \le \delta$.
Moreover, by Lemma 21 of \cite{ABBrGo09a}, we have
$\dis(\hat{C}) \le (k+1)\delta$.  The claimed result follows.
\end{proof}

\begin{lemma} \label{lem:componentwise}
Fix $i \ge 1$.  Then as $n \to \infty$,
\[
G_{\lambda}^{n,i} \convdist \mathscr{G}_{\lambda}^i
\]
in $(\mathcal{M}, \dghp)$.
\end{lemma}

\begin{proof}
This proof is a fairly straightforward modification of the proof of Theorem 22 in
\cite{ABBrGo09a}, so we will only sketch the argument. 
  Consider the component $\mathbb{G}_{\lambda}^{n,i}$.  Since we have
  fixed $\lambda$ and $i$, let us drop them from the notation and
  simply write $\mathbb{G}^n$ for the component, and similarly for
  other objects.  Write $n^{2/3} \Sigma^n \in \Z_+$ for the size of
  $\mathbb{G}^n$ and $S^n \in \Z_+$ for its surplus. We can list the
  vertices of this graph in depth-first order as $v_0, v_1, \ldots,
  v_{n^{2/3} \Sigma^n-1}$.  Let $H^{n}(k)$ be the graph distance of
  vertex $v_k$ from $v_0$.  Then it is easy to see that
  $n^{-2/3}H^{n}$ encodes a tree $T^n$ on vertices $v_0, v_1, \ldots,
  v_{n^{2/3} \Sigma^n-1}$ with metric $d_{T^n}$ such that
  $d_{T^n}(v_k, v_0) = H^n(k)$.  We endow $T^n$ with a measure by 
  letting each vertex of $T^n$ have mass $n^{-2/3}$.
  
  Next, let the pairs $\{i_1,
  j_1\}, \{i_2, j_2\}, \ldots, \{i_{S^n}, j_{S^n}\}$ give the
  indices of the surplus edges required to obtain $\mathbb{G}^{n}$ from
  $T^n$, listed in increasing order of first co-ordinate. In other words, to build 
  $\mathbb{G}^n$ from $T^n$, we add an  edge between $v_{i_k}$ and $v_{j_k}$ for 
  each $1 \le k \le S^n$ (and re-multiply distances by $n^{1/3}$). Recall that to get $G^{n}$ from $\mathbb{G}^{n}$, we 
  rescale the graph distance by $n^{-1/3}$ and assign mass $n^{-2/3}$ to each
  vertex.  It is straightforward that $G^{n}$ is at GHP distance at
  most $n^{-1/3}S^n$ from the metric space $\hat{G}^{n}$ obtained from
  $T^n$ by \emph{identifying} vertices $v_{i_k}$ and $v_{j_k}$ for all
  $1 \le k \le S^n$.

From the proof of Theorem 22 of \cite{ABBrGo09a}, we have jointly
\begin{align*}
(\Sigma^n, S^n) & \convdist (\sigma,s) \\
(n^{-1/3} H^{n}(\fl{n^{2/3} t}), 0 \le t < \Sigma^n) & \convdist (2\varepsilon(t), 0 \le t < \sigma) \\
\{ \{n^{-2/3} i_k, n^{-2/3} j_k\}, 0 \le k \le S^n\} & \convdist \{\{x^{k}, z^{k}\}, 1 \le k \le s\}.
\end{align*}
By Skorokhod's representation theorem, we may work on a probability
space where these convergences hold almost surely.  Consider the
$\R$-tree $(\mathcal{T}, d_{\mathcal{T}})$ encoded by $2 \varepsilon$
and recall that $\tau$ is the canonical projection $[0,\sigma] \to
\mathcal{T}$.  We extend $\tau$ to a function on $[0,\infty)$ by
letting $\tau(t) = \tau(t \wedge \sigma)$.  Let $\eta^n: [0, \infty)
\to \{v_0, v_1, \ldots, v_{n^{2/3} \Sigma^n - 1}\}$ be the function
defined by $\eta^n(t) = v_{\fl{n^{2/3} t} \wedge (n^{2/3}\Sigma^n
  -1)}$.  Set
\[
C^n = \{(\eta^n(t), \tau(t')): t,t' \in [0,\Sigma^n \vee
\sigma],|t-t'|\leq \delta_n\},
\]
where $\delta_n$ converges to $0$ slowly enough, that is, 
$$\delta_n\geq \max_{1\leq k\leq
  s}|n^{-2/3}i_k-x^k|\vee|n^{-2/3}j_k-z^k|\, .$$ Then $C^n$
is a correspondence between $T^n$ and $\mathcal{T}$ that contains
$(v_{i_k},x^k)$ and $(v_{j_k},z^k)$ for every $k\in \{1,2,\ldots,s\}$,
and with distortion going to $0$.  Next, let $\pi^n$ be the
push-forward of Lebesgue measure on $[0,\Sigma^n \wedge \sigma]$ under
the mapping $(\eta^n,\tau)$.  Then the discrepancy of $\pi^n$ with
respect to the uniform measure $\mu^n$ on $T^n$ and the image $\mu$ of Lebesgue
measure by $\tau$ on $\mathcal{T}$ is $|\Sigma^n - \sigma|$, and
$\pi^n((C^n)^c) = 0$.  

For all large enough $n$, $S^n = s$, so let us assume henceforth that
this holds. 
Then, writing
$\mathbf{v}=(v_{i_1},\ldots,v_{i_s},v_{j_1},\ldots,v_{j_s})$ and
$\mathbf{x}=(x^1,\ldots,x^k,z^1,\ldots,z^k)$, we have
\[
\dghp^{2s,1}((T^n,\mathbf{v},\mu_n),
(\mathcal{T},\mathbf{x},\mu)) \le \pran{\frac{1}{2}\dis(C^n)} \vee |\Sigma^n - \sigma| \to 0
\]
almost surely, as $n \to \infty$.
By Lemma~\ref{lem:measured_identifications} we thus have
$\dghp(\hat{G}^n, \mathscr{G}) \to 0$ almost surely, as $n \to
\infty$.  Since $\dghp(G^n,\hat{G}^n) \le n^{-1/3}S_n \to 0$, it follows that $\dghp(G^n, \mathscr{G}) \to 0$ as well. 
\end{proof}

\begin{proof}[Proof of Theorem~\ref{thm:gnp-converge}]
By (\ref{eq:aldous_joint_convergence}), (\ref{eq:abbrgo}), Lemma~\ref{lem:componentwise} and Skorokhod's representation theorem, we may work in a probability space in which the convergence 
in (\ref{eq:aldous_joint_convergence}) and in (\ref{eq:abbrgo}) occur almost surely, and in which for every $i \ge 1$ we almost surely have
\begin{equation} \label{eq:coordconv}
\dghp(G_{\lambda}^{n,i}, \mathscr{G}_{\lambda}^{i}) \to 0
\end{equation}
as $n \to \infty$.  Now, for each $i \ge 1$,
\[
\dghp(G_{\lambda}^{n,i}, \mathscr{G}_{\lambda}^{i}) \le 2 \max\{\diam(G_{\lambda}^{n,i}), \diam(\mathscr{G}_{\lambda}^i), \mathrm{mass}(G_{\lambda}^{n,i}), \mathrm{mass}(\mathscr{G}_{\lambda}^i)\}.
\]
The proof of Theorem 24 from \cite{ABBrGo09a} shows that almost surely 
\[
\lim_{N \to \infty} \sum_{i=N}^{\infty} \diam(\mathscr{G}^i_{\lambda})^4 = 0, 
\]
and (\ref{eq:abbrgo}) then implies that almost surely 
\[
\lim_{N \to \infty} \limsup_{n \to \infty} \sum_{i=N}^{\infty} \diam(G^{n,i}_{\lambda})^4 = 0\, .
\]
The $\ell_2^{\downarrow}$ convergence of the masses entails that almost surely
\[
\lim_{N \to \infty} \sum_{i=N}^{\infty} \mathrm{mass}(\mathscr{G}_{\lambda}^i)^4 = 0
\]
and (\ref{eq:aldous_joint_convergence}) then implies that almost surely 
\[
\lim_{N \to \infty} \limsup_{n \to \infty} \sum_{i=N}^{\infty} \mathrm{mass}(G^{n,i}_{\lambda})^4 = 0\, .
\]
Hence, on this probability space, we have
\begin{align*}
  & \lim_{N \to \infty} \limsup_{n \to \infty} \sum_{i=N}^{\infty} \dghp(G_{\lambda}^{n,i},\mathscr{G}^i_{\lambda})^4 \\
  & \le 16 \lim_{N \to \infty} \limsup_{n \to \infty} \sum_{i=N}^{\infty}
  ( \diam(G_{\lambda}^{n,i})^4 + \diam(\mathscr{G}^i_{\lambda})^4 +
  \mathrm{mass}(G_{\lambda}^{n,i})^4 +
  \mathrm{mass}(\mathscr{G}^i_{\lambda})^4 ) = 0
\end{align*}
almost surely.  Combined with (\ref{eq:coordconv}), this implies that in this space, almost surely
\[
\lim_{n \to \infty} \mathrm{dist}_{\ghp}^4(G_{\lambda}^n, \mathscr{G}_{\lambda}) = 0.
\]

The convergence of $(s(G_{\lambda}^{n,i}),i \ge 1)$ to $(s(\mathscr{G}_{\lambda}^i),i \ge 1)$ follows from (\ref{eq:aldous_joint_convergence}). 

If $i$ is such that $s(\mathscr{G}_{\lambda}^i)=1$ then, by (\ref{eq:aldous_joint_convergence}), we almost surely have $s(G_{\lambda}^{n,i})=1$ for all $n$ sufficiently large. In this case, $r(G_{\lambda}^{n,i})$ and $r(\mathscr{G}_{\lambda}^i)$ are the lengths of the unique cycles in $G_{\lambda}^{n,i}$ and in $\mathscr{G}_{\lambda}^{i}$, respectively. Now, $G_{\lambda}^{n,i} \to \mathscr{G}_{\lambda}^i$ almost surely in $(\mathring{\mathcal{M}},\dgh)$, and it follows easily that in this space, $r(G_{\lambda}^{n,i}) \to r(\mathscr{G}_{\lambda}^i)$ almost surely, for $i$ such that $s(\mathscr{G}_{\lambda}^i)=1$. 

Finally, by Theorem~4 of \cite{luczak1994srg}, $\min(r(G_{\lambda}^{n,i}):s(G_{\lambda}^{n,i}) \ge 2)$ is bounded away from zero in probability. So by Skorokhod's representation theorem, we may assume our space is such that almost surely 
\[
\liminf_{n \to \infty} \min(r(G^{n,i}_{\lambda}):s(G_{\lambda}^{n,i}) \ge 2) > 0. 
\]
In particular, it follows from the above that, for any $i \ge 1$ with
$s(\mathscr{G}_{\lambda}^i) \ge 2$, there is almost surely $r > 0$
such that $\mathscr{G}_{\lambda}^{i} \in \mathcal{A}^r$ and
$G^{n,i}_{\lambda} \in \mathcal{A}^r$ for all n sufficiently
large. Corollary \ref{sec:cutt-safely-point-1} (i) then yields that in
this space, $r(G^{n,i}_{\lambda}) \to r(\mathscr{G}_\lambda^i)$ almost
surely.

Together, the two preceding paragraphs establish the final claimed convergence. 
For completeness, we note that this final convergence may also be deduced without recourse to the results of \cite{luczak1994srg}; here is a brief sketch, using the notation of the previous lemma. 
It is easily checked that the points of the kernels
of $G^{n,i}_\lambda$ and $\mathscr{G}^i_\lambda$ correspond to the
identified vertices $(v_{i_k},v_{j_k})$ and $(x^k,z^k)$, and those 
vertices of degree at least $3$ in the subtrees of
$T^n,\mathcal{T}$ spanned by the points $(v_{i_k},v_{j_k}),1\leq k\leq
s$ and $(x^k,z^k),1\leq k\leq s$ respectively. These trees are
combinatorially finite trees (i.e., they are finite trees with edge-lengths), so the
convergence of the marked trees $(T^n,\mathbf{v})$ to
$(\mathcal{T},\mathbf{x})$ entails in fact the convergence of the same
trees marked not only by $\mathbf{v},\mathbf{x}$ but also by the points of degree
$3$ on their skeletons. Write $\mathbf{v}',\mathbf{x}'$ for these enlarged collections of points. Then one concludes by noting that $r(G_\lambda^{n,i})$ (resp.\
$r(\mathscr{G}_\lambda^i)$) is the minimum quotient distance, after
the identifications $(v_{i_k},v_{j_k})$ (resp.\ $(x^k,z^k)$) between
any two distinct elements of $\mathbf{v}'$ (resp.\ $\mathbf{x}'$).
This entails that $r(G_{\lambda}^{n,i})$ converges almost surely to
$r(\mathscr{G}_{\lambda}^i)$ for each $i \ge 1$.
\end{proof}

The above description of the sequence $\mathscr{G}_{\lambda}$ of
random $\R$-graphs does not make the distribution of the cores and
kernels of the components explicit.  (Clearly the kernel of
$\mathscr{G}_{\lambda}^i$ is only non-empty if
$s(\mathscr{G}_{\lambda}^i) \ge 2$ and its core is only non-empty if
$s(\mathscr{G}_{\lambda}^i) \ge 1$.)  Such an explicit 
distributional description was provided in \cite{ABBrGo09b}, and will be partially detailed below in Section~\ref{sec:properties}. 

\subsection{Convergence of the minimum spanning forest}

Recall that $\mathbb{M}(n,p)$ is the minimum spanning forest of $\mathbb{G}(n,p)$ and that we write
\nomenclature[Mlambdan]{$\mathbb{M}_{\lambda}^n,M_{\lambda}^n,\mathscr{M}_{\lambda}$}{Sequence of components $(\mathbb{M}^{n,i}_{\lambda},i \ge 1)$ of $\mathbb{M}(n,1/n + \lambda/n^{4/3})$; see same section for measured metric space version $M_{\lambda}^n$ and limit $\mathscr{M}_{\lambda}$.}
\[
\mathbb{M}_{\lambda}^n = (\mathbb{M}_{\lambda}^{n,i}, i \ge 1)
\]
for the components of $\mathbb{M}(n,1/n + \lambda/n^{4/3})$ listed in decreasing order of size.  For each $i \ge 1$ we write $M_{\lambda}^{n,i}$ for the measured metric space obtained from $\mathbb{M}_{\lambda}^{n,i}$ by rescaling the graph distance by $n^{-1/3}$ and giving each vertex mass $n^{-2/3}$.  We let
\[
M_{\lambda}^n = (M_{\lambda}^{n,i}, i \ge 1).
\]
Recall the cutting procedure introduced in Section~\ref{sec:breaking-cycles-r}, and that 
for an $\R$-graph $\rX$, we write $\cut(\rX)$ for a random variable with distribution $\mathcal{K}^{\infty}(\rX, \,\cdot)$. 
For $i \ge 1$, if $s(\mathscr{G}_{\lambda}^i) = 0$, let $\mathscr{M}_{\lambda}^i = \mathscr{G}_{\lambda}^i$.  Otherwise, let $\mathscr{M}_{\lambda}^i = \cut(\mathscr{G}_{\lambda}^i)$, where the cutting mechanism is run independently for each $i$. We note for later use that the mass measure on $\mathscr{M}_{\lambda}^i$ is almost surely concentrated on the leaves of $\mathscr{M}_{\lambda}^i$, since this property holds for $\mathscr{G}_{\lambda}^i$, and $\mathscr{G}_{\lambda}^i$ may be obtained from $\mathscr{M}_{\lambda}^i$ by making an almost surely finite number of identifications. 

\begin{theorem}\label{thm:msf-converge}
Fix $\lambda \in \R$.  Then as $n \to \infty$,
\[
M_{\lambda}^n \convdist \mathscr{M}_{\lambda}
\]
in the space $(\mathbb{L}_4, \mathrm{dist}_{\ghp}^4)$. 
\end{theorem}
\begin{proof}
Write 
\[
I = \sup \{i \ge 1: s(\mathscr{G}^i_{\lambda}) > 1\},
\]
with the convention that $I=0$ when $\{i \ge 1: s(\mathscr{G}^i_{\lambda}) > 1\}=\emptyset$. Likewise, for $n \ge 1$ let 
$I_n = \{i \ge 1: s(G^{n,i}_{\lambda}) > 1\}$. 
We work in a probability space in which the convergence statements of Theorem~\ref{thm:gnp-converge} are all almost sure. 
In this probability space, by Theorem 5.19 of \cite{janson00random} we have that 
$I$ is almost surely finite and that $I_n \to I$ almost surely.

By Theorem~\ref{thm:gnp-converge}, almost surely 
  $r(G_\lambda^{n,i})$ is bounded away from zero for all $i\geq 1$. 
It follows from Theorem~\ref{sec:breaking-cycles-r-1} that almost surely 
for every $i \ge 1$ we have 
\[ 
\dghp(\cut(G_{\lambda}^{n,i}),\cut(\mathscr{G}^i_{\lambda})) \to 0. 
\] 
Propositions~\ref{prop:cut-close} and~\ref{prop:cut-dist} then imply that we may work in a probability space in which almost surely, for every $i \ge 1$,
\begin{equation}\label{eq:msf-converge}
\dghp(M_{\lambda}^{n,i},\mathscr{M}^i_{\lambda}) \to 0.
\end{equation}
Now, for each $i \ge 1$, we have
\[
\dghp(M_{\lambda}^{n,i},\mathscr{M}^i_{\lambda})  \le 2 \max(\diam(M_{\lambda}^{n,i}),\diam(\mathscr{M}^i_{\lambda}), \mathrm{mass}(M_{\lambda}^{n,i}), \mathrm{mass}(\mathscr{M}_{\lambda}^i)).
\]
Moreover, for each $i \ge I$ the right-hand side is bounded above by
\[
4\max(\diam(G_{\lambda}^{n,i}),\diam(\mathscr{G}^i_{\lambda}), \mathrm{mass}(G_{\lambda}^{n,i}), \mathrm{mass}(\mathscr{G}_{\lambda}^i)). 
\]
Since $I$ is almost surely finite, as in the proof of Theorem~\ref{thm:gnp-converge} we thus almost surely have that
\begin{align*}
&  \lim_{N \to \infty} \limsup_{n \to \infty} \sum_{i=N}^{\infty} \dghp(M_{\lambda}^{n,i},\mathscr{M}^i_{\lambda})^4 \\
 & \le 64 \lim_{N \to \infty} \limsup_{n \to \infty} \sum_{i=N}^{\infty} (
\diam(G_{\lambda}^{n,i})^4 + \diam(\mathscr{G}^i_{\lambda})^4 + \mathrm{mass}(G_{\lambda}^{n,i})^4 + \mathrm{mass}(\mathscr{G}^i_{\lambda})^4 ) = 0,
\end{align*}
which combined with (\ref{eq:msf-converge}) shows that in this space, almost surely
\[
\lim_{n \to \infty} \mathrm{dist}_{\ghp}^4(M_{\lambda}^n, \mathscr{M}_{\lambda}) = 0. \qedhere
\]
\end{proof}

\subsection{The largest tree in the minimum spanning forest}

In this section, we study the largest component
$\mathbb{M}_\lambda^{n,1}$ of the minimum spanning
forest $\mathbb{M}_\lambda$ obtained by partially running Kruskal's algorithm, as well as
its analogue $\mathbb{G}_\lambda^{n,1}$ for the
random graph. It will be useful to consider the random variable
$\Lambda^n$ which is the smallest number $\lambda\in \R$ such that
$\mathbb{G}_\lambda^{n,1}$ is a subgraph of
$\mathbb{G}_{\lambda'}^{n,1}$ for every $\lambda'>\lambda$. 
\nomenclature[Lambdan]{$\Lambda^n$}{Last time a new component takes the lead; see 
Proposition~\ref{sec:largest-tree-minimum} for subsequential limit $\Lambda$.}
In other
words, in the race of components, $\Lambda^n$ is the last instant
where a new component takes the lead. It follows from Theorem 7 of
\cite{luczak90component} that $(\Lambda^n,n\geq 1)$ is tight, that is
\begin{equation}\label{eq:lzbd}
\lim_{\lambda\to \infty} \limsup_{n \to \infty} \p{\Lambda^n > \lambda} = 0.
\end{equation}
(This result is stated in \cite{luczak90component} for the other
Erd\H{o}s--R\'enyi random graph model, $\mathbb{G}(n,m)$, rather
than $\mathbb{G}(n,p)$, but it is standard that results for the
former have equivalents for the latter; see \cite{janson00random}
for more details.) 

In the following, if $x\mapsto f(x)$ is a real function, we write $f(x)=\oe(x)$ if there
exist positive, finite constants $c,c',\eps, A$ such that 
\nomenclature[Oe]{$\mathrm{oe}(x)$}{Say $f(x)=\mathrm{oe}(x)$ if $\lvert f(x)\rvert \leq c\exp(-c' x^{\eps})$ 
for some $c,c',\eps$ and $x$ large.}
$$|f(x)|\leq c\exp(-c' x^\eps)\, , \qquad \mbox{ for every }x>A\, .$$
In the following lemma, we write $\dhau(M^{n,1}_\lambda,M^n)$ for the
Hausdorff distance between $M^{n,1}_\lambda$ and $M^n$, seen as
subspaces of $M^n$. 
\nomenclature[Dh]{$\dhau$}{Hausdorff distance.}
Obviously, $\dgh(M_\lambda^{n,1},M^n)\leq
\dhau(M_\lambda^{n,1},M^n)$.

\begin{lemma}\label{lem:cauchy-property2}For any $\epsilon\in (0,1)$ and
  $\lambda_0$ large enough, we have
$$\limsup_{n\to\infty}
\p{\dhau(M_\lambda^{n,1},M^n)\ge \frac{1}{\lambda^{1-\epsilon}}\, \Big|\,
  \Lambda^n\leq \lambda_0}=\oe(\lambda)\, .$$
\end{lemma}
In the course of the proof of Lemma~\ref{lem:cauchy-property2}, we
will need the following estimate on the length of the longest path
outside the largest component of a random graph within the critical
window.

\begin{lemma}\label{lem:long_path-improved} For all $0<\eps<1$ there
  exists $\lambda_0$ such that for all $\lambda \ge \lambda_0$ and all $n$
  sufficiently large, the probability that a connected component of
  $\mathbb G_\lambda^n$ aside from $\mathbb{G}_\lambda^{n,1}$ contains
  a simple path of length at least $n^{1/3}/\lambda^{1-\epsilon}$ is
  at most $e^{-\lambda^{\epsilon/2}}$.
\end{lemma}
The proof of Lemma~\ref{lem:long_path-improved} follows precisely the
same steps as the proof of Lemma~3 (b) of \cite{abbr09}, which is
essentially the special case $\eps =1/2$.\footnote{In \cite{abbr09} it was sufficient for the purpose of
  the authors to produce a path length bound of
  $n^{1/3}/\lambda^{1/2}$, but their proof does imply the present
  stronger result. For the careful reader, the key point is that the
  last estimate in Theorem 19 of \cite{abbr09} is a specialisation of
  a more general bound, Theorem 11 (iii) of
  \cite{luczak98random}. Using the more general bound in the proof is
  the only modification required to yield the above result.}  Since no
new idea is involved, we omit the details.
\begin{proof}[Proof of Lemma~\ref{lem:cauchy-property2}]
Fix $f_0 > 0$ and for $i \geq 0$, let $f_i = (5/4)^i \cdot f_0$. Let $t=t(n)$ be the smallest $i$ for which $f_i \geq n^{1/3}/\log n$. Lemma 4 of \cite{abbr09} (proved via Prim's algorithm) states that
\[
\E{\diam(\mathbb{M}^n) - \diam(\mathbb{M}^{n,1}_{f_t})} = O(n^{1/6} (\log n)^{7/2});
\]
this is established by proving the following stronger bound, which will be useful in the sequel:
\begin{equation} \label{eqn:tdistance}
\p{\dhau(M^{n,1}_{f_t},M^n) > n^{-1/6}(\log n)^{7/2}} \le \frac 1 n.
\end{equation}
Let $B_i$ be the event that some component of $\mathbb{G}^n_{f_i}$ aside from $\mathbb{G}^{n,1}_{f_i}$ contains a simple path with more than $n^{1/3}/f_i^{1-\epsilon}$ edges and let
\[
I^n = \max\{i\leq t: B_i~\mbox{occurs}\}.
\]
Lemma~\ref{lem:long_path-improved} entails that, for $f_0$ sufficiently large, for all $n$, and all $0 \leq i \leq t-1$, 
\[
\p{i\le I^n \le t} \leq \sum_{\ell \ge i} e^{-f_i^{\epsilon/2}}\le 2 e^{-f_i^{\epsilon/2}},
\]
where the last inequality holds for all $i$ sufficiently large. 
For fixed $i < t$, if $\Lambda^n \le f_i$ then for all $\lambda \in [f_i,f_t]$ we have
\[
\dhau(M_{\lambda}^{n,1},M_{f_t}^{n,1}) \leq \dhau(M_{f_i}^{n,1},M_{f_t}^{n,1}).
\]
If, moreover, $I^n \le i$, then we have 
\begin{equation}\label{dghtreebd}
\dhau(M_{f_i}^{n,1},M_{f_t}^{n,1}) 
\leq \sum_{j=i+1}^t f_j^{\epsilon-1} \leq \frac{f_i^{\epsilon-1}}{1-(4/5)^{1-\epsilon}}< 10 f_i^{\epsilon-1},
\end{equation}
the latter inequality holding for $\epsilon<1/2$.

Finally, fix $ \lambda \in \R$ and let $i_0=i_0(\lambda)$ be such that
$\lambda\in [f_{i_0},f_{i_0+1})$. Since $f_t\to \infty$ as
$n\to\infty$, we certainly have $i_0< t$ for all $n$ large
enough. Furthermore,
\begin{align*}
\lefteqn{\p{\dhau(M_{\lambda}^{n,1},M^n) \geq \frac 1 {\lambda^{1-\epsilon}}\,\Big|
\,\Lambda^n\leq \lambda_0}} \\
& \le \p{\dhau(M_{\lambda}^{n,1}, M_{f_t}^{n,1}) \ge \frac 1 2 \frac 1 {\lambda^{1-\epsilon}}\,\Big|
\,\Lambda^n\leq\lambda_0} + \p{\dhau(M_{f_t}^{n,1}, M^n) \ge \frac 1 2 \frac 1 {\lambda^{1-\epsilon}}\,\Big|
\,\Lambda^n>\lambda_0} \\
& \leq \frac{1}{\p{
\Lambda^n\leq \lambda_0}}\left(\p{\dhau(M_{f_{i_0}}^{n,1},M_{f_t}^{n,1}) >
\frac{10}{f_{i_0}^{1-\epsilon/2}}}  + \frac 1 n \right),
\end{align*}
for all $\lambda$ large enough and all $n$ such that $2 \lambda\le
n^{1/6} (\log n)^{-7/2}$, by (\ref{eqn:tdistance}). It then follows from \eqref{dghtreebd} and
the tightness of $(\Lambda^n,n\geq 1)$ that there exists a constant
$C\in (0,\infty)$ such that for all $\lambda_0$
large enough,
\begin{align*}
\p{\dhau(M_{\lambda}^{n,1},M^n) \geq \frac 1 {\lambda^{1-\epsilon}}\,\Big|
\,\Lambda^n\leq \lambda_0}
& \leq \frac{1}{\p{\Lambda^n\leq \lambda_0}}\left(\p{i_0(\lambda) \le I^n \le t} + \frac 1 n \right)\\
& \le C \left(e^{-f_{i_0(\lambda)}^{\epsilon/2}}+
\frac 1 n \right).
\end{align*}
Letting $n$ tend to infinity proves the lemma.
\end{proof}
We are now in a position to prove a partial version of our main
result.  
In what follows, we write $\mathring{M}^n$, $\mathring{M}^{n,1}_{\lambda}$ and $\mathring{\mathscr{M}}^1_{\lambda}$ for the metric spaces obtained from 
$M^n$, $M^{n,1}_{\lambda}$ and $\mathscr{M}^1_{\lambda}$ by ignoring their measures. 
\nomenclature[Mnzzz]{$\mathring{M}^n,\mathring{M}^{n,1}_{\lambda},\mathring{\mathscr{M}}, \mathring{\mathscr{M}}^1_{\lambda}$}{Spaces obtained from $M^n$, $M^{n,1}_{\lambda},\mathscr{M},\mathscr{M}^1_{\lambda}$ by ignoring measures.}
\begin{lemma} \label{lem:firstdraft}
There exists a random compact metric space $\mathring{\mathscr{M}}$ such that, as $n \to \infty$,
\[
\mathring{M}^n \convdist \mathring{\mathscr{M}} \quad\text{in~} (\mathring{\mathcal{M}}, \dgh).
\]
Moreover, as $\lambda \to \infty$,
\[
\mathring{\mathscr{M}}_{\lambda}^1 \convdist \mathring{\mathscr{M}} \quad \text{in~}(\mathring{\mathcal{M}}, \dgh).
\]
\end{lemma}
\begin{proof}
Recall that the metric space $(\mathring{\mathcal{M}},\dgh)$ is complete and separable.  Theorem~\ref{thm:msf-converge} entails that
\[
\mathring{M}_{\lambda}^{n,1} \convdist \mathring{\mathscr{M}}_{\lambda}^1
\]
as $n \to \infty$ in $(\mathring{\cM}, \dgh)$.  The stated results then
follow from this, Lemma~\ref{lem:cauchy-property2} and the principle
of accompanying laws (see Theorem 3.1.14 of \cite{Stroock93} or
Theorem 9.1.13 in the second edition).
\end{proof}

Let $\hat{M}^{n,1}_{\lambda}$ be the measured metric space obtained from $M^{n,1}_{\lambda}$ by rescaling so that the total mass is one (in $M^{n,1}_{\lambda}$ we gave each vertex mass $n^{-2/3}$; now we give each vertex mass $|V(\mathbb{M}^{n,1}_{\lambda})|^{-1}$). 
\nomenclature[Mhatlambda]{$\hat{M}^{n,1}_{\lambda}$}{$M^{n,1}_{\lambda}$ renormalised to have mass one.}
\begin{proposition}\label{lem:mst-lim}
For any $\eps > 0$,
\[
\lim_{\lambda \to \infty} \limsup_{n \to \infty} \p{ \dghp(\hat{M}^{n,1}_{\lambda},M^n) \geq \eps} =0. 
\]
\end{proposition}

In order to prove this proposition, we need some notation and a lemma.
Let $\mathbb{F}^{n}_{\lambda}$ be the subgraph of $\mathbb{M}^n$ with
edge set $E(\mathbb{M}^n) \setminus
E(\mathbb{M}_{\lambda}^{n,1})$. 
\nomenclature[Fnlambda]{$\mathbb{F}^{n}_{\lambda}$}{Subgraph of $\mathbb{M}^n$ with
edges $E(\mathbb{M}^n) \setminus E(\mathbb{M}_{\lambda}^{n,1})$; its components 
are $\{\mathbb{F}_{\lambda}^n(v), v \in V(\mathbb{M}_{\lambda}^{n,1})\}$.} 
Then $\mathbb{F}^n_{\lambda}$ is a
forest which we view as rooted by taking the root of a component to be
the unique vertex in that component which was an element of
$\mathbb{M}^{n,1}_{\lambda}$. For $v \in V(\mathbb{M}_{\lambda}^{n,1})$,
let $S^{n}_{\lambda}(v)$ be the number of nodes in the component
$\mathbb{F}_{\lambda}^n(v)$ of $\mathbb{F}^n_{\lambda}$ rooted at $v$.
\nomenclature[Snlambda]{$S^{n}_{\lambda}(v)$}{Size of $\mathbb{F}_{\lambda}^n(v)$.}
The fact that the random variables $(S^{n}_{\lambda}(v), v \in
V(\mathbb{M}_{\lambda}^{n,1}))$ are exchangeable will play a key role
in what follows.

\begin{lemma} \label{lem:exchbitssmall}
For any $\delta > 0$,
\begin{equation}
\lim_{\lambda \to \infty} \limsup_{n \to \infty} \p{ \max_{v \in V(\mathbb{M}_{\lambda}^{n,1})} S^{n}_{\lambda}(v) > \delta n} = 0.
\end{equation}
\end{lemma}

\begin{proof}
Let $U_{\lambda}^n$ be the event that vertices 1 and 2 lie in the same component of $\mathbb{F}_{\lambda}^n$.  Note that, conditional on $\max_{v \in V(\mathbb{M}_{\lambda}^{n,1})} S^{n}_{\lambda}(v) > \delta n$, the event $U_{\lambda}^n$ occurs with probability at least $\delta^2/2$ for sufficiently large $n$.  So, in order to prove the lemma it suffices to show that
\begin{equation} \label{eqn:12diffcmpts}
\lim_{\lambda \to \infty} \limsup_{n} \p{U_{\lambda}^n} = 0.
\end{equation}

In order to prove (\ref{eqn:12diffcmpts}), we consider the following modification of Prim's algorithm.  We build the MST conditional on the collection $\mathbb{M}_{\lambda}^n$ of trees.  We start from the component containing vertex 1 in $\mathbb{M}_{\lambda}^{n}$.  To this component, we add the lowest weight edge connecting it to a new vertex.  This vertex lies in a new component of $\mathbb{M}_{\lambda}^{n}$, which we add in its entirety, before again seeking the lowest-weight edge leaving the tree we have so far constructed.  We continue in this way until we have constructed the whole MST.  (Observe that the components we add along the way may, of course, be singletons.)  Note that if we think of Prim's algorithm as a discrete-time process, with time given by the number of vertices added so far, then this is simply a time-changed version which looks only at times when we add edges of weight \emph{strictly greater than $1/n+\lambda/n^{4/3}$}.  This is because when Prim's algorithm first touches a component of $\mathbb{M}_{\lambda}^n$, it necessarily adds all of its edges before adding any edges of weight exceeding $1/n + \lambda/n^{4/3}$.  For $i \ge 0$, write $C_i$ for the tree constructed by the modified algorithm up to step $i$ and let $e_i$ be the edge added at step $i$.  The advantage of the modified approach is that, for each $i \ge 1$, we can calculate the probability that the endpoint of $e_i$ which does not lie in $C_{i-1}$ touches $\mathbb{M}_{\lambda}^{n,1}$, given that it has not at steps $0,1,\ldots,i-1$.  Recall that, at each stage of Prim's algorithm, we add the edge of minimal weight leaving the current tree.  We are thinking of this tree as a collection of components of $\mathbb{M}_{\lambda}^n$ connected by edges of weight strictly greater than $1/n + \lambda/n^{4/3}$.  In general, different sections of the tree built so far are subject to different conditionings depending on the weights of the connecting edges and the order in which they were added.  In particular, the endpoint of $e_i$ contained in $C_{i-1}$ is more likely to be in a section with a lower weight-conditioning.  However, the \emph{other} endpoint of $e_i$ is equally likely to be any of the vertices of $\{1,2,\ldots,n\} \setminus C_{i-1}$ because all that we know about them is that they lie in (given) components of $\mathbb{M}^{n}_{\lambda}$.

Formally, let $k=n - 1-|E(\mathbb{M}^{n}_{\lambda})|$. Let $C_0$ be the component containing 1 in $\mathbb{M}_{\lambda}^n$.  Recursively, for $1 \le i \le k$, let 
\begin{itemize}
\item $e_i$ be the smallest-weight edge leaving $C_{i-1}$ and 
\item $C_i$ be the component containing 1 in the graph with edge-set $E(\mathbb{M}_{\lambda}^n) \cup \{e_1, \ldots, e_i\}$.
\end{itemize}
The graph with edge-set $E(\mathbb{M}_{\lambda}^n) \cup \{e_1, \ldots,
e_k\}$ is precisely $\mathbb{M}^n$.  Let $I_1$ be the first index for
which $V(\mathbb{M}^{n,1}_{\lambda}) \subset V(C_{I_1})$, so that
$I_1$ is the time at which the component containing $1$ attaches to
$\mathbb{M}^{n,1}_{\lambda}$. For each $1 \le i \le k$, the endpoint
of $e_i$ not in $C_{i-1}$ is uniformly distributed among all vertices
of $\{1,\ldots,n\}\setminus C_{i-1}$.  So, conditionally given
$\mathbb{M}_\lambda^n,e_1,\ldots,e_{i-1}$ and on $\{I_1\geq i\}$, the
probability that $I_1$ takes the value $i$ is
$|V(\mathbb{M}_{\lambda}^{n,1})| / (n - |V(C_{i-1})|)$. Therefore, 
\begin{equation}
\label{eq:6}
\p{I_1>i\, |\, \mathbb{M}_\lambda^n}\leq
\left(1-\frac{|V(\mathbb{M}_\lambda^{n,1})|}{n}\right)^i\, .
\end{equation}
By Theorem 2 of \cite{NaPe07} (see also Lemma 3 of \cite{luczak90component}), for all $\delta > 0$, 
\begin{equation} \label{eqn:luczak} \lim_{\lambda \to \infty}
  \limsup_{n \to \infty} \p{ \left|
      \frac{|V(\mathbb{M}^{n,1}_{\lambda})|}{2\lambda n^{2/3}} - 1
    \right| > \delta} = 0.
\end{equation}
Using \eqref{eq:6} and \eqref{eqn:luczak}, it follows that for any $\delta > 0$, there exists $B > 0$ such that 
\begin{equation}\label{eq:mst-lim2}
\lim_{\lambda \to \infty} \limsup_{n \to \infty} \p{I_1 >
  Bn^{1/3}/\lambda} < \delta\, .
\end{equation}
Next, let $Z$ be a uniformly random element of
$\{1,\ldots,n\}\setminus V(\mathbb{M}^{n,1}_{\lambda})$, and let
$L_{\lambda}$ be the size of the component of
$\mathbb{M}^{n}_{\lambda}$ that contains $Z$.  Theorem~A1 of
\cite{janson2008susceptibility} shows that
\[
\lim_{\lambda \to \infty} \limsup_{n \to \infty} \E{\frac{\sum_{i=2}^{\infty} |V(\mathbb{M}_{\lambda}^{n,i})|^2}{n^{4/3}/\lambda}} < \infty,
\]
which implies that 
\[
\lim_{\lambda \to \infty} \limsup_{n \to \infty} \E{\frac{L_{\lambda}}{n^{1/3}/\lambda}} < \infty.
\]
For each $i
\ge 1$, given that $i < I_1$, the difference $|V(C_i)| - |V(C_{i-1})|$
is stochastically dominated by $L_{\lambda}$, so that 
$$\E{|V(C_{I_1-1})|\ind_{\{ I_1\leq  i\}}}\leq i\E{L_\lambda}\, .$$
By \eqref{eq:mst-lim2} and Markov's inequality, there
exists $B' > 0$ such that
\begin{equation}\label{eq:mst-lim3}
\lim_{\lambda \to \infty} \limsup_{n \to \infty} \p{|V(C_{I_1-1})| > B'n^{2/3}/\lambda^2} < \delta. 
\end{equation}
The graph $C_*$ with edge-set $E(C_{I_1-1}) \cup \{e_{I_1}\}$ forms
part of the component containing 1 in $\mathbb{F}_{\lambda}^n$;
indeed, the endpoint of $e_{I_1}$ not contained in $C_{I_1-1}$ is the
root of this component.  Write $v_1$ for this root vertex.  Now
consider freezing the construction of the MST via the modified version
of Prim's algorithm at time $I_1$ and constructing the rest of the MST
using the modified version of Prim's algorithm starting now from
vertex 2.  Let $\ell = n - 1 - |E(\mathbb{M}_{\lambda}^n)| - I_1$.
Let $D_0$ be the component containing 2 in the graph with edge-set
$E(\mathbb{M}_{\lambda}^n) \cup \{e_1, \ldots, e_{I_1}\} $.
Recursively, for $1 \le j \le \ell$, let
\begin{itemize}
\item $f_j$ be the smallest-weight edge leaving $D_{j-1}$ and
\item $D_j$ be the component containing 2 in the graph with edge-set $E(\mathbb{M}_{\lambda}^n) \cup \{e_1, \ldots, e_{I_1}\} \cup \{f_1, \ldots, f_j\}$.
\end{itemize}
Let $I_2$ be the first index for which $f_{I_2}$ has an endpoint in $V(\mathbb{M}_{\lambda}^{n,1}) \setminus \{v_1\}$, and 
let $J_1$ be the first index for which $f_{J_1}$ has an endpoint in $V(C_*)$. 

Recall that $U_{\lambda}^n$ is the event that 1 and 2 lie in the same component of $\mathbb{F}_{\lambda}^n$.  If $U_{\lambda}^n$ occurs then we necessarily have $J_1 < I_2$. 
To prove (\ref{eqn:12diffcmpts}) it therefore suffices to show that 
\begin{equation}\label{eq:mst-lim5}
\lim_{\lambda \to \infty} \limsup_{n \to \infty} \p{J_1 < I_2} = 0. 
\end{equation}
In order to do so, we first describe how the construction of $C_{I_1}$ conditions the locations of attachment of the edges $f_i$. 
As in the introduction, for $e \in E(K_n)$, $W_{e}$ is the weight of edge $e$, and unconditionally these weights are i.i.d.\ Uniform$[0,1]$ random variables. 

Write $A_0=V(C_0)$, and for $1 \le i \le I_1$, let $A_i = V(C_i)\setminus V(C_{i-1})$. In particular, $A_{I_1}=V(\bbM_{\lambda}^{n,1})$. After $C_{I_1}$ is constructed, for each $0 \le i \le I_1$, the conditioning on edges incident to $A _i$ is as follows.
\begin{itemize}
\item[(a)] Every edge between $V(C_{i-1})$ and $A_i$ has weight at least $W_{e_i}$. 
\item[(b)] For each $i < j \le I_1$, every edge between $A_i$ and $[n]\setminus V(C_j)$ has weight at least $\max\{ W_{e_j}, i < k \le j\}$.
\end{itemize}
In particular, (b) implies all edges from $A_i$ to $[n]\setminus V(C_{I_1})$ are conditioned to have weight at least $\max\{W_{e_j}, i < k \le I_1\}$. This entails that components which are added later have lower weight-conditioning. In particular, there is no conditioning on edges from $A_{I_1}=V(\bbM_{\lambda}^{n,1})$ to $[n]\setminus V(C_{I_1})$ (except the initial conditioning, that all such edges have weight at least $1/n + \lambda/n^{4/3}$, which comes from conditioning on $\mathbb{M}^n_{\lambda}$). 

It follows that under the conditioning imposed by the construction of $C_{I_1}$, it is {\em not} the case that for $1 \le j < \ell$, the endpoint of $f_{j+1}$ \emph{outside} $D_j$ is uniformly distributed among $\{1,\ldots,n\}\setminus D_j$.
However, the conditioning precisely biases these endpoints away from the sets $A_i$ with $i < I_1$ (but not away from $A_{I_1}=V(\mathbb{M}^{n,1}_{\lambda})$). 
As a consequence, for each $1 \le j \le \ell$, conditional on the edge set $E(\mathbb{M}_{\lambda}^n) \cup \{e_1, \ldots, e_{I_1}\} \cup \{f_1, \ldots, f_j\}$ and on the event $\{J_1 \ge j\} \cap \{I_2 \ge j\}$,
the probability that $j = I_2$ is at least $(|V(\mathbb{M}_{\lambda}^{n,1})| - 1)/(n-|V(D_{j-1})|)$ and the probability that $j=J_1$ is at most $|V(C_{*})|/(n-|V(D_{j-1})|)$.  Hence,
\[
\p{I_2 > i \,|\, \bbM_{\lambda}^n} \le \left(1 - \frac{|V(\bbM_{\lambda}^{n,1})| - 1}{n} \right)^i
\]
and so, by (\ref{eqn:luczak}), we obtain that for any 
$\delta > 0$ there exists $B'' > 0$ such that 
\begin{equation} \label{eq:ibound}
\lim_{\lambda \to \infty} \limsup_{n \to \infty} \p{I_2 > B''n^{1/3}/\lambda} < \delta.
\end{equation}
Moreover,
\[
\p{J_1 > i\, |\, \bbM_{\lambda}^n} \ge \pran{1 - \frac{|V(C_*)|}{n - |V(D_{J_1-1})|}}^i.
\]
Note that $|V(C_*)| = |V(C_{I_1 - 1})| + 1$.  Also, just as for the components $C_i$, given that $i < J_1$, the difference $|V(D_i)| - |V(D_{i-1})|$ is stochastically dominated by $L_{\lambda}$ and so we obtain the analogue of (\ref{eq:mst-lim3}): there exists $B'''$ such that
\[
\lim_{\lambda \to \infty} \limsup_{n \to \infty} \p{|V(D_{J_1 - 1})| > B''' n^{2/3}/\lambda^2} < \delta.
\]
Hence, from this and (\ref{eq:mst-lim3}), we see that there exists $B''''$ such that
\begin{equation} \label{eq:jbound}
\lim_{\lambda \to \infty} \limsup_{n \to \infty} \p{J_1 \le \lambda^2 n^{1/3}/B''''} < \delta. 
\end{equation}
Together, (\ref{eq:ibound}) and (\ref{eq:jbound}) establish (\ref{eq:mst-lim5}) and complete the proof.
\end{proof}

Armed with this lemma, we now turn to the proof of Proposition~\ref{lem:mst-lim}.

\begin{proof}[Proof of Proposition~\ref{lem:mst-lim}]
  Fix $\eps > 0$ and let $N^n_{\eps}$ be the minimal number of open
  balls of radius $\eps/4$ needed to cover the (finite) space $M^n$.
  This automatically yields a covering of $M_{\lambda}^{n,1}$ by
  $N^n_{\eps}$ open balls of radius $\eps/4$ since $M_\lambda^{n,1}$ is
  included in $M^n$. From this covering, we can easily construct a new
  covering $B_{\lambda}^{n,1}, \ldots, B_{\lambda}^{n,N^n_{\eps}}$ of
  $M_\lambda^{n,1}$ by sets of diameter at most $\eps/2$ which are
  pairwise disjoint. Let
$$\tilde{B}_\lambda^{n,i}=\bigcup_{v\in
  B_\lambda^{n,i}}\mathbb{F}_\lambda^n(v)\, ,\qquad 1\leq i\leq
N^n_\eps\, ,$$ and let
$C^n_\lambda=\bigcup_{i=1}^{N^n_\eps}(B_\lambda^{n,i}\times
\tilde{B}_\lambda^{n,i})\, ,$ which defines a correspondence between
$M_\lambda^{n,1}$ and $M^n$. Moreover, its
distortion is clearly at most $2\dhau(M_\lambda^{n,1},M^n)+\eps$. 
Therefore, 
by
Lemma~\ref{lem:cauchy-property2}, 
\begin{equation} \label{eq:distortionsmall}
\lim_{\lambda \to \infty} \limsup_{n \to \infty} \p{\dis(C_{\lambda}^n) > 2\eps} = 0.
\end{equation}
Next, write $V^n_{\lambda} =
|V(\mathbb{M}_{\lambda}^{n,1})|$ and take an arbitrary relabelling of
the elements of $V(\mathbb{M}_{\lambda}^{n,1})$ by $\{1,2,\ldots,
V_n^{\lambda}\}$.  Since $(S_{\lambda}^n(1), \ldots,
S_{\lambda}^n(V_{\lambda}^n))$ are exchangeable, Theorem 16.23 of
Kallenberg~\cite{Kallenberg} entails that for any
$\delta > 0$,
\begin{equation} \label{eq:Kallenberg} \lim_{\lambda \to \infty}
  \limsup_{n \to \infty} \p{\max_{1 \le i \le V_{\lambda}^n}
    \left|\sum_{j=1}^i \frac{S_{\lambda}^n(j)}{n} -
      \frac{i}{V_{\lambda}^n} \right| > \delta} = 0
\end{equation}
as soon as we have that for all $\delta > 0$,
\[
\lim_{\lambda \to \infty} \limsup_{n \to \infty} \p{\max_{1 \le i \le V_{\lambda}^n} S_{\lambda}^n(i) > \delta n} = 0,
\]
which is precisely the content of Lemma~\ref{lem:exchbitssmall}. 

Now define a measure $\pi^n$ on $M_\lambda^{n,1}\times M^n$ by 
$$\pi^n(\{(u,v)\})=\frac{1}{V_\lambda^n|\tilde{B}_\lambda^{n,i}|}\wedge
\frac{1}{n|B_\lambda^{n,i}|}\, ,\qquad (u,v)\in B_\lambda^{n,i}\times
\tilde{B}_\lambda^{n,i}\, ,\quad 1\leq i\leq N^n_\eps\, .$$
Note that
$\pi^n_{\lambda}((C_{\lambda}^n)^c) = 0$ by definition.
Moreover, the marginals of $\pi^n$ are given by 
$$\pi^n_{(1)}(\{u\})=\frac{1}{V_\lambda^n}\wedge
\frac{|\tilde{B}_\lambda^{n,i}|}{n|B_\lambda^{n,i}|}
\, ,\qquad u\in B_\lambda^{n,i}\, ,\quad 1\leq i\leq N^n_\eps\, ,$$
and 
$$\pi^n_{(2)}(\{v\})=\frac{1}{n}\wedge
\frac{|B_\lambda^{n,i}|}{V_\lambda^n|\tilde{B}_\lambda^{n,i}|}\,
  ,\qquad v\in \tilde{B}_\lambda^{n,i}\, ,\quad 1\leq i\leq
  N^\eps_n\, .$$
Therefore, the discrepancy $D(\pi^n_\lambda)$ of $\pi_{\lambda}^n$ with respect
to the uniform measures on $M^n$ and $M_\lambda^{n,1}$ is at most
\[
N_{\eps}^n \max_{1 \le i \le N_{\eps}^n} \left|\frac{|\tilde{B}_\lambda^{n,i}|}{n} - \frac{|B_{\lambda}^{n,i}|}{V_{\lambda}^n} \right|
\]
which (by relabelling the elements of $M_{\lambda}^{n,1}$ so that the
vertices in each $B_\lambda^{n,i}$ have consecutive labels and using
exchangeability) is bounded above by
\[
2  N_{\eps}^n \max_{1 \le i \le V_{\lambda}^n} \left|\sum_{j=1}^i \frac{S_{\lambda}^n(j)}{n} - \frac{i}{V_{\lambda}^n} \right|.
\]
Then
\begin{align*}
& \p{\dghp(\hat{M}_{\lambda}^{n,1}, M^n) \ge \eps} \\
& \qquad \le \p{\dis(C_{\lambda}^n) > 2\eps} + \p{D(\pi_{\lambda}^n) > \eps} \\
& \qquad \le \p{\dis(C_{\lambda}^n) > 2\eps} + \p{N_{\eps}^n \max_{1 \le i \le V_{\lambda}^n} \left|\sum_{j=1}^i \frac{S_{\lambda}^n(j)}{n} - \frac{i}{V_{\lambda}^n} \right|> \frac{\eps}{2}} \\
& \qquad \le \p{\dis(C_{\lambda}^n) > 2\eps} + \p{\max_{1 \le i \le V_{\lambda}^n} \left|\sum_{j=1}^i \frac{S_{\lambda}^n(j)}{n} - \frac{i}{V_{\lambda}^n} \right|> \frac{\eps}{2K}} + \p{N_{\eps}^n > K}.
\end{align*}
But now recall that $N_{\eps}^n$ is the minimal number of open balls
of radius $\eps/4$ needed to cover $M^n$.  Let $N_{\eps}$ be the same quantity for
$\mathring{\mathscr{M}}$. Then by
Lemma~\ref{lem:firstdraft}, $\mathring{M}^n \convdist \mathring{\mathscr{M}}$, which
easily implies that $\limsup_{n\to\infty}\p{N_{\eps}^n>K}\leq
\p{N_\eps> K}$.  In particular, by (\ref{eq:distortionsmall})
and (\ref{eq:Kallenberg})
\[
\lim_{\lambda \to \infty} \limsup_{n \to \infty} \p{\dghp(\hat{M}_{\lambda}^{n,1}, M^n) \ge \eps} \le \p{N_{\eps} > K}
\]
and the right-hand side converges to 0 as $K \to \infty$.
\end{proof}

Let $\hat{\mathscr{M}}_{\lambda}^1$ be the measured metric space
obtained from $\mathscr{M}_{\lambda}^1$ by renormalising the measure
to be a probability.
\nomenclature[Mhatlambda]{$\hat{\mathscr{M}}_{\lambda}^1$}{$\mathscr{M}_{\lambda}^1$ renormalised to have mass one.}
\begin{theorem}\label{thm:big_statement}
There exists a random compact measured metric space $\mathscr{M}$ of total mass 1 such that as $n \to \infty$,
\[
M^n \convdist \mathscr{M}
\]
in the space $(\mathcal{M}, \dghp)$.  Moreover, as $\lambda \to \infty$,
\[
\hat{\mathscr{M}}_{\lambda}^1 \convdist \mathscr{M}
\]
in the space $(\mathcal{M}, \dghp)$. 
Finally, writing $\mathscr{M}=(X,d,\mu)$, we have $(X,d) \eqdist \mathring{\mathscr{M}}$ in $(\mathcal{M},\dgh)$, where 
$\mathring{\mathscr{M}}$ is as in \refL{lem:firstdraft}.
\end{theorem}

\begin{proof} Recall that the metric space $(\cM,\dghp)$ is a complete and separable.  Theorem~\ref{thm:msf-converge} entails that
\[
\hat{M}_{\lambda}^{n,1} \convdist \hat{\mathscr{M}}_{\lambda}^1
\]
as $n \to \infty$ in $(\cM, \dghp)$.  The stated results then follow from this, Proposition~\ref{lem:mst-lim} and the principle of accompanying laws (see Theorem 3.1.14 of \cite{Stroock93} or Theorem 9.1.13 in the second edition).
\end{proof}

Finally, we observe that, analogous to the fact that $M_\lambda^{n,1}$ is a subspace of $M^n$, 
we can view $\mathscr{M}_\lambda^1$ as a subspace of $\mathscr{M}$. 
(We emphasize that this does not follow from Theorem~\ref{thm:big_statement}.) 
To this end, we briefly introduce the marked Gromov--Hausdorff topology of
\cite[Section 6.4]{miermont09}. Let 
$\mathcal{M}_*$ be the set of ordered pairs of the form $(\rX,Y)$, where $\rX=(X,d)$ is a
compact metric space and $Y\subset X$ is a compact subset of $X$ (such
pairs are considered up to isometries of $\rX$). 
\nomenclature[Mdghpstar]{$\mathcal{M}_*$}{Set of pairs $(\rX,Y)$ where $\rX$ is a compact metric space and $Y \subset X$ is compact; see same section for the marked Gromov--Hausdorff topology.}
A sequence
$(\rX_n,Y_n)$ of such pairs converges to a limit $(\rX,Y)$ if
there exist correspondences $C_n\in C(X_n,X)$ whose restrictions to
$Y_n\times Y$ are correspondences between $Y_n$ and $Y$, and such
that $\dis(C_n)\to 0$. (In particular, this implies that $Y_n$
converges to $Y$ for the Gromov--Haudorff distance, when these spaces
are equipped with the restriction of the distances on
$\rX_n,\rX$.) Moreover, a set $\mathcal{A}\subset \mathcal{M}_*$ 
is relatively compact if and only if $\{\rX:(\rX,Y)\in \mathcal{A}\}$
is relatively compact for the Gromov--Hausdorff topology.

Recall the definition of the tight sequence of random variables
$(\Lambda^n,n\geq 1)$ at the beginning of this section.  
By taking subsequences, we may assume that we have the joint convergence in 
distribution 
$$(((\mathring{M}^n,\mathring{M}_\lambda^{n,1}),\lambda\in \Z),\Lambda^n)\convdist
(((\tilde{\mathscr{M}},\tilde{\mathscr{M}}_\lambda^1),\lambda\in
\Z),\Lambda)\, ,$$ 
for the product topology on $\mathcal{M}_*^{\Z}\times \R$.\footnote{This 
is a slight abuse of notation, in the sense
that the limiting spaces $\tilde{\mathscr{M}}$ on the right-hand side
should, in principle, depend on $\lambda$, but obviously these
spaces are almost surely all isometric.}  This coupling of course has the properties that 
$\tilde{\mathscr{M}}\eqdist\mathring{\mathscr{M}}$ and that 
$\tilde{\mathscr{M}}_\lambda^1\eqdist\mathring{\mathscr{M}}_\lambda^1$ for every
$\lambda\in \Z$.  Combining this with Lemma \ref{lem:cauchy-property2}
we easily obtain the following.

\begin{proposition}
  \label{sec:largest-tree-minimum}
  There exists a probability space on which one may define a triple
  \[
  (\tilde{\mathscr{M}},(\tilde{\mathscr{M}}_\lambda^1,\lambda\in
  \Z),\Lambda)
  \]
  with the following properties: (i) $\Lambda$ is an a.s.\ finite random
  variable; (ii) $\tilde{\mathscr{M}}\eqdist\mathring{\mathscr{M}}$, 
  $\tilde{\mathscr{M}}_\lambda^1\eqdist\mathring{\mathscr{M}}_\lambda^1$ and
  $(\tilde{\mathscr{M}},\tilde{\mathscr{M}}_\lambda^1)\in
  \mathcal{M}_*$ for every $\lambda\in \Z$; and (iii) 
  for every $\eps\in (0,1)$ and $\lambda_0>0$ large enough,
$$\p{\dhau(\tilde{\mathscr{M}},\tilde{\mathscr{M}}_\lambda^1)>\lambda^{\eps-1}\, \Big|\,\Lambda\leq
  \lambda_0}=\oe(\lambda)\, .$$ In particular,
$(\tilde{\mathscr{M}},\tilde{\mathscr{M}}_\lambda^1)\convdist
(\tilde{\mathscr{M}},\tilde{\mathscr{M}})$ as $\lambda\to\infty$ for
the marked Gromov-Hausdorff topology. 
\end{proposition}

\section{Properties of the scaling limit} \label{sec:properties}

In this section we give some properties of the limiting metric space $\mathscr M$. We start with some general properties that $\mathscr M$ shares with the Brownian CRT of Aldous \cite{aldouscrt91,aldouscrtov91,aldouscrt93}:

\begin{theorem}\label{thm:basic_properties}
$\mathscr{M}$ is a measured $\R$-tree which is almost surely binary and whose mass measure is concentrated on its leaves.
\end{theorem}

\begin{proof}
  By the second distributional convergence in
  Theorem~\ref{thm:big_statement}, we may (and will) work in a space
  in which we almost surely have $\lim_{\lambda \to \infty}
  \dghp(\hat{\mathscr{M}}^1_{\lambda},\mathscr{M}) = 0$.  Since it is the
  Gromov--Hausdorff limit of the sequence of $\R$-trees $\mathscr
  M_\lambda^1$, $\mathscr M$ is itself an $\R$-tree (see for instance
  \cite{evans2006rpr}).  For fixed $\lambda \in \R$, each component of
  $\mathscr{M}_{\lambda}$ is obtained from $\mathscr G_\lambda$, the
  scaling limit of $G_\lambda^n$, using the cutting process. 
  From the construction of $\mathscr G_\lambda$ detailed in Section~\ref{sec:scaling-limit-erdhos}, 
  it is clear that $\mathscr G_\lambda^1$ almost surely does not contain points of 
  degree more than three, and so $\mathscr M_\lambda^1$ is almost surely
  binary.

  Next, let us work with the coupling
  $(\tilde{\mathscr{M}},(\tilde{\mathscr{M}}_\lambda^1,\lambda\in
  \Z))$, of Proposition
  \ref{sec:largest-tree-minimum}. We can assume, using the
  last statement of this proposition and the Skorokhod representation
  theorem, that
  $(\tilde{\mathscr{M}},\tilde{\mathscr{M}}_\lambda^1)\to
  (\tilde{\mathscr{M}},\tilde{\mathscr{M}})$ a.s.\ in
  $\mathcal{M}_*$.  Now suppose that $\tilde{\mathscr{M}}$ has a point
  $x_0$ of degree at least $4$ with positive probability. On this
  event, we can find four points $x_1,x_2,x_3,x_4$ of the skeleton of
  $\tilde{\mathscr{M}}$, each having degree $2$, and such that the geodesic
  paths from $x_i$ to $x_0$ have strictly positive lengths and meet
  only at $x_0$.  But for $\lambda$ large enough, $x_0,x_1,\ldots,x_4$
  all belong to $\tilde{\mathscr{M}}_\lambda^1$, as well as the
  geodesic paths from $x_1,\ldots,x_4$ to $x_0$. This contradicts the
  fact that $\mathscr{M}_\lambda^1$ is binary. Hence, $\mathscr{M}$ is
  binary almost surely.

Let $x$ and $x^{\lambda}$ be sampled according to the probability measures on $\mathscr{M}$ 
and on $\hat{\mathscr{M}}^1_{\lambda}$, respectively. 
For the remainder of the proof we abuse notation by writing $(\mathscr{M},x)$ and 
$(\hat{\mathscr{M}}^1_{\lambda},x^{\lambda})$ for the marked spaces 
(random elements of $\mathcal{M}^{1,1}$) obtained by marking at the points 
$x$ and $x^{\lambda}$. 
Then we may,
in fact, work in a space in which almost surely
\[
\lim_{\lambda \to \infty} \dghp^{1,1}((\mathscr{M},x),(\hat{\mathscr{M}}^1_{\lambda},x^{\lambda})) = 0\, . 
\]
As noted earlier, the mass measure on $\hat{\mathscr{M}}^1_{\lambda}$ is
almost surely concentrated on the leaves of $\hat{\mathscr{M}}^1_{\lambda}$,
and it follows that for each fixed $\lambda$, $x^{\lambda}$ is almost
surely a leaf.  Let 
\[
\Delta(x) = \sup\left\{ \min(f^{-1}(x),t-f^{-1}(x))~\text{for}~f:[0,t] \to M~\mbox{a geodesic with }x \in \im(f)\right\}\, ,
\]
so, in particular, $\Delta(x)=0$ precisely if $x$ is a leaf. For each fixed
$\lambda$, since $x^{\lambda}$ is almost surely a leaf, it is
straightforward to verify that almost surely
\[
\dghp^{1,1}((\mathscr{M},x),(\hat{\mathscr{M}}^1_{\lambda},x^{\lambda})) \ge \Delta(x)/2.
\]
But then taking $\lambda \to \infty$ along any countable sequence shows that $\Delta(x)=0$ almost surely. 
\end{proof}

To distinguish $\mathscr M$ from Aldous' CRT, we look at a natural
notion of fractal dimension, the \emph{Minkowski (or box-counting) 
dimension} \cite{Falconer90}. Given a compact metric space $\rX$ and $r
> 0$, let $N(\rX, r)$ be the minimal number of open balls of radius $r$
needed to cover $\rX$. 
\nomenclature[Nxr]{$N(\rX, r)$}{Minimal number of open balls of radius $r$
needed to cover $\rX$.}

We define the lower and upper Minkowski dimensions by
$$\underline{\dim}_{\mathrm{M}}(\rX) = \liminf_{r\downarrow 0} \frac{\log N(\rX,r)}{\log(1/r)}\qquad \text{and}\qquad \overline{\dim}_{\mathrm{M}}(\rX) = \limsup_{r\downarrow 0} \frac{\log N(\rX,r)}{\log(1/r)}.$$
If $\underline{\dim}_{\mathrm{M}}(\rX)=\overline{\dim}_{\mathrm{M}}(\rX)$, then this value
is called the Minkowski dimension and is denoted $\dim_{\mathrm{M}}(\rX)$.
\nomenclature[Dimm]{$\dim_{\mathrm{M}}(\rX)$}{Minkowski dimension of $\rX$; see same section for 
$\underline{\dim}_{\mathrm{M}}(\rX)$ and $\overline{\dim}_{\mathrm{M}}(\rX)$.}
\begin{proposition}\label{prop:box-counting}
The Minkowski dimension of $\mathscr{M}$ exists and is equal to $3$
almost surely. 
\end{proposition}

Since the Brownian CRT $\mathscr{T}$ satisfies
$\dim_{\mathrm{M}}(\mathscr{T})=2$ almost surely (\cite[Corollary
5.3]{duqlegprep}), we obtain the following result, which gives a
negative answer to a conjecture of Aldous \cite{aldous90mst}.
\begin{corollary}\label{cor:notCRT}
For any random variable
$A>0$, the laws of $\mathscr{M}$ and of $A\mathscr{T}$, the metric space 
$\mathscr{T}$ with distances rescaled by $A$, are mutually singular.
\end{corollary}

The proof relies on an explicit description of the components of
$\mathscr{G}_\lambda$, given in \cite{ABBrGo09b}. We only give a
partial statement, since that is all that we need here.  Note
that, given $s(\mathscr{G}_\lambda^1)=k\geq 2$, the
kernel $\mathrm{ker}(\mathscr{G}_\lambda^1)$ is a 3-regular multigraph
with $3k-3$ edges and hence $2(k-1)$ vertices.  Fix $\lambda\in \R$
and $k\geq 2$, and $K$ a $3$-regular multigraph with $3k-3$ edges.
Label the edges of $K$ by $\{1,2,\ldots, 3k-3\}$ arbitrarily. 
  \begin{description}
  \item[Construction 1. ] Independently sample random variables $\Gamma_k \sim
    \mathrm{Gamma}((3k-2)/2,1/2)$ and $(Y_1, Y_2, \ldots, Y_{3k-3}) \sim
    \mathrm{Dirichlet}(1,1, \ldots,1)$.
    Attach a line-segment of length $Y_j \sqrt{\sigma \Gamma_k}$ in
    the place of edge $j$ in $K$, for $1 \le j \le 3k-3$.
  \item[Construction 2. ] Sample 
    $(X_1,X_2,\ldots,X_{3k-3}) \sim \mathrm{Dirichlet}(1/2,1/2,\ldots,1/2)$
    and, given $(X_1,\ldots,X_{3k-3})$, let
    $(\mathscr{T}^{(1)},\ldots,\mathscr{T}^{(3k-3)})$ be independent CRT's
    with masses given by $(\sigma X_1,\ldots,\sigma X_{3k-3})$ respectively. For
    $1\leq i\leq 3k-3$, let $(x_i,x_i')$ be two independent points in
    $\mathscr{T}^{(i)}$, chosen according to the normalized mass measure. 
    Take the metric gluing of $(\mathscr{T}^{(i)},1\leq i\leq 3k-3)$ induced by the
    graph structure of $K$, by viewing $x_i,x'_i$ as the extremities of
    the edge $i$.
\end{description}

Here we should recall some of the basic properties of the CRT
$\mathscr{T}$, referring the reader to, e.g., \cite{legall05b} for more
details. If $\varepsilon=(\varepsilon(s),0\leq s\leq 1)$ is a standard normalized
Brownian excursion then $\mathscr{T}$ is the quotient space of $[0,1]$
endowed with the pseudo-distance
$d_\varepsilon(s,t)=2(\varepsilon(s)+\varepsilon(t)-2\inf_{s\wedge t\leq u\leq s\vee
  t}\varepsilon(u))$, by the relation $\{d_\varepsilon=0\}$. It is seen as a
measured metric space by endowing it with the mass measure which
is the image of Lebesgue measure on $[0,1]$ by the canonical
projection $p:[0,1]\to \mathscr{T}$. It is also naturally rooted at
the point $p(0)$. Likewise, the CRT with mass $\sigma$, denoted by
$\mathscr{T}_\sigma$, is coded in a similar fashion by (twice) a
Brownian excursion conditioned to have duration $\sigma$. By scaling
properties of Brownian excursion, this is the same as multiplying
distances by $\sqrt{\sigma}$ in $\mathscr{T}$, and multiplying the
mass measure by $\sigma$.

\begin{proposition}
  \label{sec:prop-scal-limit-1}  
  The metric space obtained by Construction 1 (resp.\ Construction 2) has
  same distribution as $\core(\mathscr{G}_\lambda^1)$ (resp.\
  $\mathscr{G}_\lambda^1$), given 
  $\mathrm{mass}(\mathscr{G}_{\lambda}^1) = \sigma$,
  $s(\mathscr{G}_{\lambda}^1) = k$ and
  $\mathrm{ker}(\mathscr{G}_\lambda^1)=K$.
\end{proposition}

The proof of Proposition \ref{prop:box-counting} builds on this result
and requires a couple of lemmas.  Recall the notation $\rX_x$ from Section
\ref{sec:breaking-cycles-r}.
\begin{lemma}
  \label{sec:prop-scal-limit-2}Let $\rX=(X,d,x)$ be a safely pointed
  $\R$-graph and fix $r>0$. Then $N(\rX,r)\leq N(\rX_x,r)\leq N(\rX,r)+2$.
\end{lemma}
This lemma will be proved in Section \ref{sec:cutting-procedure}, where
we give a more precise description of $\rX_x$. 
The next lemma is a concentration
result for the mass and surplus of $\mathscr{G}_\lambda^1$. This
should be seen as a continuum analogue of similar results in
\cite{luczak90component,NaPe07}. We stress that these bounds are far from
being sharp, and could be much improved by a more careful analysis. In
the rest of this section, if $(Y(\lambda),\lambda\geq 0)$ is a family
of positive random variables and $(f(\lambda),\lambda\geq 0)$ is a
positive function, 
we write $Y(\lambda)\asymp f(\lambda)$ if for all $a > 1$, 
\nomenclature[Aaaa]{$\asymp$}{$Y(\lambda)\asymp f(\lambda)$ if for all $a > 1$, $\p{Y(\lambda) \not \in [f(\lambda)/a,af(\lambda)]}=\oe(\lambda)$.}
\[\p{Y(\lambda) \not \in 
  [f(\lambda)/a,af(\lambda)]}
  =\oe(\lambda)\, .
\]
Note that this only constrains the above probability for large $\lambda$.

\begin{lemma}
  \label{sec:prop-scal-limit}
It is the case that 
$$\mathrm{mass}(\mathscr{G}_\lambda^1)\asymp 2\lambda\qquad \mbox{ and
}\qquad 
s(\mathscr{G}_\lambda^1)\asymp\frac{2\lambda^3}{3}\, .$$
\end{lemma}

\proof We use the construction of $\mathscr{G}_\lambda$ described in
Section \ref{sec:scaling-limit-erdhos}. Recall that $(W(t),t\geq 0)$
is a standard Brownian motion, that $W_\lambda(t)=W(t)+\lambda
t-t^2/2$, and that $B_{\lambda}(t)=W_{\lambda}(t)-\min_{0 \le s \le t} W_{\lambda}(s)$. 
Note that, letting
\begin{equation}
  \label{eq:8}
  A_\lambda=\{|W(t)|\leq (2\lambda)\vee t\mbox{ for all }t\geq 0\}\, ,
\end{equation} 
we have $\p{A_\lambda^c}=\oe(\lambda)$.  Considering first $t \le 2\lambda$, 
by symmetry, the reflection principle and scaling we have that
\begin{align*}
\p{\sup_{0\leq t\leq 2\lambda}|W(t)|>
  \lambda}&\leq 2\p{\sup_{0\leq t\leq 2\lambda}W(t)>
  \lambda}\\
&=2\p{|W(2\lambda)|> \lambda}
\\
&=2\p{|W(2)|>\sqrt{\lambda}}\,
,
\end{align*}
and this is $\oe(\lambda)$ since $W(2)$ is Gaussian. Turning to $t > 2\lambda$, note that letting $W'=(W(u+2\lambda)-W(2\lambda),u\geq 0)$, 
then $W'$ is a standard Brownian motion by the Markov property. Hence, on the event
$\{\sup_{0\leq t\leq 2\lambda}|W(t)|\leq \lambda\}$, the probability
that $|W(t)|>t$ for some $t\geq 2\lambda$ is at most $\p{\exists u\geq
  0:|W'(u)|\geq u+\lambda}\leq 2 \p{\max_{u\geq 0}(W'(u)-u)\geq
  \lambda}$. We deduce that $\p{A_\lambda^c}=\oe(\lambda)$
from the fact that $\max_{u\geq 0}(W'(u)-u)$ has an exponential
distribution, see e.g.\ \cite{revyor}.

On $A_\lambda$, 
$$-\frac{t^2}{2}+\lambda t-((2\lambda)\vee t) \leq W_\lambda(t)\leq
-\frac{t^2}{2}+\lambda t+(2\lambda)\vee t \, ,\qquad t\geq 0\, ,$$
from which it is elementary to obtain that if $\lambda\geq 4$, the following properties hold. 
\begin{itemize}
\item[(i)] The excursion $\varepsilon$ of $B_\lambda$ that straddles the time
  $\lambda$ has length in $[2\lambda-8,2\lambda+8]$.
\item[(ii)] All other excursions of $B_\lambda$ have length at most $6$. 
\item[(iii)] The area of $\varepsilon$ is in
  $[2\lambda^3/3-4\lambda^2,2\lambda^3/3+8\lambda^2]$.
\end{itemize}
Note that (i) and (ii) imply that, for $\lambda\geq 8$, on
$A_\lambda$, the excursion $\varepsilon$ of $B_\lambda$ is the longest, 
which we previously called $\varepsilon^1$, and which encodes the
component $\mathscr{G}_\lambda^1$ of $\mathscr{G}_\lambda$. This
implies that $\mathrm{mass}(\mathscr{G}_\lambda^1)\asymp 2\lambda$, 
since $\mathrm{mass}(\mathscr{G}_\lambda^1)$ is precisely the length 
of $\varepsilon^1$. Finally, recall that, given $\varepsilon^1$,
$s(\mathscr{G}_\lambda^1)$ has a Poisson distribution with parameter
equal to the area of $\varepsilon^1$. Therefore, standard large deviation
bounds together with (iii) imply that
$s(\mathscr{G}_\lambda^1)\asymp 2\lambda^3/3$. 
\endproof

\begin{proof}[Proof that $\underline{\dim}_{\mathrm{M}}(\mathscr{M})\geq 3$ almost surely.]
In this proof, we always
work with the coupling from Proposition \ref{sec:largest-tree-minimum},
but for convenience omit the decorations from the notation, e.g., writing 
$\mathscr{M}$ in place of $\mathring{\mathscr{M}}$ or of $\tilde{\mathscr{M}}$. 
In particular, this allows us to view $\mathscr{M}_\lambda^1$ as a subspace 
of $\mathscr{M}$ for every $\lambda\in \Z$. 

Since $\mathscr{M}_\lambda^1$ is obtained from 
$\mathscr{G}_\lambda^1$ by performing the cutting operation of Section
\ref{sec:breaking-cycles-r}, 
Lemma~\ref{sec:prop-scal-limit-2} implies that for every $r>0$, 
\begin{equation}
  \label{eq:12}
\p{N(\mathscr{M},1/\lambda)<r}\leq \p{N(\mathscr{M}_\lambda^1,1/\lambda)<r}\leq
\p{N(\mathscr{G}_\lambda^1,1/\lambda)<r}\, .
\end{equation}
Next, by viewing
$\core(\mathscr{G}_\lambda^1)$ as a graph with edge-lengths, we obtain
that $N(\mathscr{G}_\lambda^1,1/\lambda)$ is at least equal to the
number $N'(1/\lambda)$ of edges of $\core(\mathscr{G}_\lambda^1)$ that
have length at least $2/\lambda$, since the open balls with radius
$1/\lambda$ centred at the midpoints of these edges are pairwise
disjoint.

Now fix $\sigma>0,k\geq 2$ and a $3$-regular multigraph $K$ with
$3k-3$ edges, and recall the notation of {\bf Construction 1}. 
Given that
$\mathrm{mass}(\mathscr{G}_\lambda^1)=\sigma,s(\mathscr{G}_\lambda^1)=k$
and $\ker(\mathscr{G}_\lambda^1)=K$,  the edge-lengths of
$\core(\mathscr{G}_\lambda^1)$ are given by
$Y_i\sqrt{\sigma\Gamma_k},1\leq i\leq 3k-3$, and we conclude that
(still conditionally)
$$N'(1/\lambda)\eqdist|\{i\in
\{1,\ldots,3k-3\}:Y_i\sqrt{\sigma\Gamma_k}> 2/\lambda\}|\, .$$ Note that
this does not depend on $K$ but only on $\sigma$ and on $k$. Now $\Gamma_k\sim \mathrm{Gamma}((3k-2)/2,1/2)$ can
be represented as the sum of $3k-2$ independent random variables with
distribution $\mathrm{Gamma}(1/2,1/2)$, which have mean $1$, and by
standard large deviation results this implies that
$$\sup_{k\in [\lambda^3/2,\lambda^3]}\p{\Gamma_k<
  \lambda^3}=\oe(\lambda)\, .$$ Hence, by first conditioning on
$\mathrm{mass}(\mathscr{G}_\lambda^1),s(\mathscr{G}_\lambda^1)$ and
  using Lemma \ref{sec:prop-scal-limit}, for any given $c>0$, 
\begin{align}
 \p{N'(1/\lambda)<c\lambda^3}&\leq \sup_{\substack{\sigma\geq \lambda\\k\in
      [\lambda^3/2,\lambda^3]}}\p{N'(1/\lambda)<c\lambda^3\, |\,
    \mathrm{mass}(\mathscr{G}_\lambda^1)=\sigma,s(\mathscr{G}_\lambda^1)=k}+\oe(\lambda)\nonumber\\
 & \leq \sup_{\substack{\sigma\geq \lambda\\k\in
      [\lambda^3/2,\lambda^3]}}\p{|\{i\in
    \{1,\ldots,3k-3\}:Y_i\sqrt{\sigma\Gamma_k}>2/\lambda\}|<c\lambda^3}+\oe(\lambda)
\nonumber\\
  &\leq \sup_{k\in
      [\lambda^3/2,\lambda^3]}\p{|\{i\in
    \{1,\ldots,3k-3\}:Y_i>2/\lambda^3\}|<c\lambda^3}+\oe(\lambda) \nonumber
\end{align}
We now use that $(Y_1,\ldots,Y_{3k-3})\sim
\mathrm{Dirichlet}(1,\ldots,1)$ is distributed as
$(\gamma_1,\ldots,\gamma_{3k-3})/(\gamma_1+\ldots+\gamma_{3k-3})$, where
$\gamma_1,\ldots,\gamma_{3k-3}$ are independent Exponential$(1)$ random
variables. 
Standard large deviations results 
for gamma random variables imply that 
$$\sup_{k\in[\lambda^3/2,\lambda^3]}\p{\gamma_1+\ldots+\gamma_{3k-3}> 4\lambda^3} =
\oe(\lambda)\, .$$
From this we obtain
\begin{align}
\lefteqn{  \sup_{k\in
      [\lambda^3/2,\lambda^3]}\p{|\{i\in
    \{1,\ldots,3k-3\}:Y_i>2/\lambda^3\}|<c\lambda^3}}\nonumber\\
&\leq 
\sup_{k\in
      [\lambda^3/2,\lambda^3]}\p{|\{i\in
    \{1,\ldots,3k-3\}:\gamma_i>8\}|<c\lambda^3}+\oe(\lambda)\nonumber
\end{align}
and this is $\oe(\lambda)$ for $c<e^{-8}$, since $|\{i\in \{1,\ldots,3k-3\}:\gamma_i>8\}|$ is 
Bin$(3k-3,e^{-8})$ distributed. 
   
It follows that for such $c$, 
$\p{N'(1/\lambda)<c \lambda^3} = \oe(\lambda)$, 
which with \eqref{eq:12} implies that 
$$\p{N(\mathscr{M},1/\lambda)<c\lambda^3/2}=\oe(\lambda)\, .$$
We obtain by the Borel--Cantelli Lemma that $N(\mathscr{M},1/\lambda)\geq
c\lambda^3/2$ for all $\lambda \in \Z$ sufficiently large. 
By sandwiching $1/r$ between
consecutive integers, this yields that almost surely
$$\underline{\dim}_{\mathrm{M}}(\mathscr{M})=\liminf_{r\to0}\frac{\log
  N(\mathscr{M},r)}{\log (1/r)}\geq 3\, .$$ 
\end{proof}
We now prove the upper bound from Proposition~\ref{prop:box-counting}. 
\begin{proof}[Proof that $\overline{\dim}_{\mathrm{M}}(\mathscr{M})\leq 3$ almost surely.]
Recall the definition of $\Lambda$ from Proposition \ref{sec:largest-tree-minimum}. 
Fix $\lambda_0 > 0$ and an integer $\lambda > \lambda_0$. 
We work conditionally on the event $\{\Lambda\leq \lambda_0\}$. 
Next, fix $\eps>0$. If $B_1,\ldots,B_N$ is a covering of $\mathscr{M}_\lambda^1$
by balls of radius $1/\lambda^{1-\eps}$ then, since
$\mathscr{M}_\lambda^1\subset \mathscr{M}$, the centres
$x_1,\ldots,x_N$ of these balls are elements of $\mathscr{M}$. 
On the event 
$\{\dhau(\mathscr{M}_\lambda^1,\mathscr{M})<1/\lambda^{1-\eps}\}$,
whose complement has conditional probability $\oe(\lambda)$ by Proposition
\ref{sec:largest-tree-minimum}, the balls with centres
$x_1,\ldots,x_N$ and radius $2/\lambda^{1-\eps}$ then form a covering of
$\mathscr{M}$. Hence
\begin{align}\label{eq:11}
 \p{N(\mathscr{M},2/\lambda^{1-\eps})>5\lambda^3 \, |\, \Lambda\leq
    \lambda_0}&\leq
  \frac{\p{N(\mathscr{M}_\lambda^1,1/\lambda^{1-\eps})>5\lambda^3}}{\p{\Lambda\leq
      \lambda_0}}+\oe(\lambda)\nonumber\\
  &\leq
  \frac{\p{N(\mathscr{G}_\lambda^1,1/\lambda^{1-\eps})+2s(\mathscr{G}_\lambda^1)>5\lambda^3}}{\p{\Lambda\leq
      \lambda_0}}+\oe(\lambda)\nonumber\\
&\leq  \frac{\p{N(\mathscr{G}_\lambda^1,1/\lambda^{1-\eps})>3\lambda^3}}{\p{\Lambda\leq
      \lambda_0}}+\oe(\lambda)
\end{align}
where in the penultimate step we used Lemma
\ref{sec:prop-scal-limit-2} and the fact that $\mathscr{M}_\lambda^1$
is obtained from $\mathscr{G}_\lambda^1$ by performing
$s(\mathscr{G}_\lambda^1)$ cuts, and in the last step we
used the fact that $s(\mathscr{G}_\lambda^1)\asymp 2\lambda^3/3$ from
Lemma \ref{sec:prop-scal-limit}. 

To estimate $N(\mathscr{G}_\lambda^1,1/\lambda^{1-\eps})$, we now use
{\bf Construction 2} to obtain a copy of $\mathscr{G}_\lambda^1$ conditioned to satisfy 
$\mathrm{mass}(\mathscr{G}_\lambda^1)=\sigma,s(\mathscr{G}_\lambda^1)=k$
and $\ker(\mathscr{G}_\lambda^1)=K$, where $K$ is a $3$-regular
multigraph with $3k-3$ edges. Recall that we glue $3k-3$ Brownian CRT's $(\mathscr{T}^{(1)}_{\sigma X_1},\ldots,\mathscr{T}^{(3k-3)}_{\sigma X_{3k-3}})$ along the edges of $K$.
These CRT's are conditionally independent given their masses $\sigma X_1,\ldots,\sigma X_{3k-3}$, and $(X_1,\ldots,X_{3k-3})$ has Dirichlet$(1/2,\ldots,1/2)$
distribution. (Here we include the mass in the notation because it will vary later on.) 
If each of these trees has diameter less than $1/\lambda^{1-\eps}$, then clearly we can cover
the glued space by $3k-3$ balls of radius $1/\lambda^{1-\eps}$, each centred in a distinct tree $\mathscr{T}^{(i)}_{\sigma X_i},1\leq i\leq
3k-3$. Therefore, by first conditioning on
$\mathrm{mass}(\mathscr{G}_\lambda^1)$ and on $s(\mathscr{G}_\lambda^1)$, and
then using Lemma \ref{sec:prop-scal-limit},
\begin{align}\label{eq:1729}
\lefteqn{\p{N(\mathscr{G}_\lambda^1,1/\lambda^{1-\eps})>3\lambda^3}}\nonumber\\
&\leq
\sup_{\substack{\sigma\leq 3\lambda\\k\in
    [\lambda^3/2,\lambda^3]}}  
\p{N(\mathscr{G}_\lambda^1,1/\lambda^{1-\eps})>3\lambda^3\, |\,
  \mathrm{mass}(\mathscr{G}_\lambda^1)=\sigma,s(\mathscr{G}_\lambda^1)=k}+\oe(\lambda)\nonumber\\
&\leq \sup_{\substack{\sigma\leq 3\lambda\\k\in
    [\lambda^3/2,\lambda^3]}}  
\p{\max_{1\leq i\leq 3k-3}\diam(\mathscr{T}^{(i)}_{\sigma X_i})>1/\lambda^{1-\eps}}+\oe(\lambda)
\end{align}

We can represent $(X_1,\ldots,X_{3k-3})$ as
$(\gamma_1,\ldots,\gamma_{3k-3})/(\gamma_1+\ldots+\gamma_{3k-3})$, where
$\gamma_1,\ldots,\gamma_{3k-3}$ are i.i.d.\ random variables with
distribution $\mathrm{Gamma}(1/2,1)$. 
Hence 
\begin{align*}
 \p{\max_{1\leq i\leq 3k-3}X_i> 1/\lambda^{3-\eps}}&\leq
  \p{\gamma_1+\ldots+\gamma_{3k-3}<\lambda^{3-\eps/2}}
  +\p{\max_{1\leq i\leq {3k-3}}\gamma_i>\lambda^{\eps/2}}\\
  &\leq
  \p{\gamma_1+\ldots+\gamma_{3k-3}<\lambda^{3-\eps/2}}+1-\left(1-\p{\gamma_1>\lambda^{\eps/2}}\right)^{3k-3}\, .
\end{align*}
Standard large deviations results for gamma random variables 
then entail that for all $\eps > 0$, 
\[
\sup_{k\in [\lambda^3/2,\lambda^3]}\p{\max_{1\leq i\leq 3k-3}X_i> 1/\lambda^{3-\eps}}=\oe(\lambda)\, ,
\]
which in turn implies that 
\begin{align}
\lefteqn{\sup_{\substack{\sigma\leq   3\lambda\\ k\in
    [\lambda^3/2,\lambda^3]}}\p{\max_{1\leq i\leq
    3k-3}\diam(\mathscr{T}^{(i)}_{\sigma X_i})>1/\lambda^{1-\eps}} }\nonumber\\
&\leq 
\sup_{k\in [\lambda^3/2,\lambda^3]}\p{\max_{1\leq i\leq
    3k-3}X_i>1/\lambda^{3-\eps}}
+\p{\max_{1\leq i\leq
    3\lambda^3}\diam(\mathscr{T}^{(i)}_{3/\lambda^{2-\eps}})>1/\lambda^{1-\eps}}\nonumber
\end{align}
where we used that, by scaling,
$\diam(\mathscr{T}_\sigma)$ is stochastically increasing in $\sigma$,
and 
$(\mathscr{T}^{(i)}_{3/\lambda^{2-\eps}}),1\leq i\leq \lfloor
3\lambda^3\rfloor$ are independent CRT's, each with mass 
$3/\lambda^{3-\eps}$.
Using this bound and Brownian scaling, it follows that 
\begin{equation}\label{eq:14}
\sup_{\substack{\sigma\leq   3\lambda\\ k\in
    [\lambda^3/2,\lambda^3]}}\p{\max_{1\leq i\leq
    3k-3}\diam(\mathscr{T}^{(i)}_{\sigma X_i})>1/\lambda^{1-\eps}} 
    \leq
\oe(\lambda)+1-\left(1-\p{\diam(\mathscr{T}>\lambda^{\eps/2}/\sqrt{3})}\right)^{3\lambda^3}\, .
\end{equation}
Next, 
it is well-known that the height of $\mathscr{T}$, that is, the
maximal distance from the root to another point, is theta-distributed:
$$\p{\mathrm{height}(\mathscr{T})\geq x}=\sum_{k\geq
  1}(-1)^{k+1}e^{-k^2x^2}\leq e^{-x^2}\, .$$ 
  Since $\diam(\mathscr{T})\leq 2\,
\mathrm{height}(\mathscr{T})$ it follows that 
$$\p{\diam(\mathscr{T})\geq x}=\oe(x)\, .$$
We obtain that
\eqref{eq:14} is $\oe(\lambda)$, and (\ref{eq:1729}) then yields that
$\p{N(\mathscr{G}_\lambda^1,1/\lambda^{1-\eps})>3\lambda^3}=\oe(\lambda)$.
By \eqref{eq:11}, we then have 
$$\p{N(\mathscr{M},2/\lambda^{1-\eps})>5\lambda^3\, |\,
  \Lambda\leq \lambda_0}=\oe(\lambda)\, .$$ Therefore, the
Borel--Cantelli Lemma implies that
$N(\mathscr{M},2/\lambda^{1-\eps})\leq 5\lambda^3$ a.s.\ for every
integer $\lambda>\lambda_0$ large enough. This implies that,
conditionally on $\{\Lambda\leq \lambda_0\}$,
$\overline{\dim}_{\mathrm{M}}(\mathscr{M})\leq 3+\eps$ almost surely for every
$\eps>0$, by sandwiching $1/r$ between integers in $\limsup_{r\to
  0}\log N(\mathscr{M},r)/\log(1/r)$. Since $\Lambda$ is almost
surely finite and $\lambda_0$ was arbitrary, this then holds
unconditionally for any $\eps>0$. \end{proof} This concludes the proof
of Proposition \ref{prop:box-counting}.

\section{The structure of $\R$-graphs} \label{sec:RtreesRgraphs}

In this section, we investigate $\R$-graphs and prove
the structure theorems claimed in Section~\ref{sec:r-trees-r}. 

\subsection{Girth in $\R$-graphs}
In this section, $\rX=(X,d)$ is an $\R$-graph. The {\em girth} of $\rX$ is defined by 
\nomenclature[Girx]{$\mathrm{gir}(\rX)$}{The girth of $\rX$ is $\inf \{\len(c):c \mbox{ is an embedded cycle in }X\}$.}
\[
\mathrm{gir}(\rX)=\inf \{\len(c):c \mbox{ is an embedded cycle in }X\}\, .
\]
If $(X,d)$ is an $\R$-graph, then by definition $(B_{\eps(x)}(x),d)$
is an $\R$-tree for every $x\in X$ and for some function
$\eps:X\to(0,\infty)$. The balls $(B_{\eps(x)}(x),x\in X)$ form an
open cover of $X$. By extracting a finite sub-cover, we see that there
exists $\eps>0$ such that for every $x\in X$, the space
$(B_\eps(x),d)$ is an $\R$-tree. 
\nomenclature[Rx]{$R(\rX)$}{Largest $\eps$ such that $B_{\eps(x)}(x)$ is an $\R$-tree for all $x \in X$.}
We let $R(\rX)$ be the supremum of
all numbers $\eps>0$ with this property. It is immediate that
$\mathrm{gir}(\rX) \ge 2R(X)>0$. In fact, it is not difficult
to show that $\mathrm{gir}(\rX)=4R(\rX)$ and that $(B_{R(\rX)}(x),d)$
is an $\R$-tree. More precisely, the closed ball
$(\overline{B}_{R(\rX)}(x),d)$ is also a (compact) $\R$-tree, since it
is the closure of the corresponding open ball. These facts are not
absolutely crucial in the arguments to come, but they make some proofs
more elegant, so we will take them for granted and leave their
proofs to the reader, who is also referred to Proposition~2.2.15 of
\cite{papa}.

\begin{proposition} 
  \label{sec:r-graphs} 
  If $f\in \mathcal{C}([a,b],X)$ is a local geodesic in $\rX$, 
  then for every $t\in [a,b]$, the restriction of 
  $f$ to $[t-R(\rX),t+R(\rX)]\cap[a,b]$ is a geodesic. In particular, 
  if $c$ is an embedded cycle and $x\in c$, then $c$ contains a 
  geodesic arc of length $2R(\rX)$ with mid-point $x$. 
\end{proposition}

\proof The function $f$ is
injective on any interval of length at most $2R(\rX)$, since otherwise
we could exhibit an embedded cycle with length at most
$2R(\rX)=\mathrm{gir}(\rX)/2$. In particular, $f$ is injective on the
interval $[t-R(\rX),t+R(\rX)]\cap[a,b]$, and takes values in the
$\R$-tree $(\overline{B}_{R(\rX)}(f(t)),d)$, so that its image is a
geodesic segment, and since $f$ is parameterized by arc-length, its
restriction to the above interval is an isometry. This proves the
first statement. 

For the second statement, note that every injective path $f\in
\mathcal{C}([a,b],X)$ parameterized by arc-length is a local geodesic
since, for every $t$, the path $f$ restricted to
$[t-R(\rX),t+R(\rX)]\cap [a,b]$ is an injective path in the $\R$-tree
$(\overline{B}_{R(\rX)}(f(t)),d)$ parameterized by arc-length, and hence
is a geodesic. If now $g:\mathbb{S}_1\to X$ is an injective continuous
function inducing the embedded cycle $c$, it suffices to apply the
previous claim to a parametrisation by arc-length mapping $0$ to $x$
of the function $t\mapsto g(e^{2\mathrm{i}\pi t})$.
\endproof

\subsection{Structure of the core}
In this section, $\rX=(X,d)$ is again an $\R$-graph. 
Recall that $\core(\rX)$ is the union of all arcs with endpoints in
embedded cycles.  

\begin{proposition}
  \label{sec:skeleton-core-kernel}
  The set $\core(\rX)$ is a finite union of embedded cycles and simple
  arcs that are disjoint from the embedded cycles except at their
  endpoints. Moreover, the space $(\core(\rX),d)$ is an $\R$-graph with no
  leaves.
\end{proposition}

\proof Assume, for a contradiction, that the union of all embedded cycles
cannot be written as a finite union of embedded cycles. Then we can
find an infinite sequence $c_1,c_2,\ldots$ of embedded cycles such
that $c_i\setminus(c_1\cup\cdots\cup c_{i-1})$ is non-empty for every
$i\geq 0$, and thus contains at least one point $x_i$. Up to
taking subsequences, one can assume that $x_i$ converges to some point $x$, 
and that $d(x,x_i)<R(\rX)/2$ for every $i\geq 1$. Let $\gamma_i'$ be a
geodesic from $x$ to $x_i$: this geodesic takes its values in the
$\R$-tree $\overline{B}_{R(\rX)}(x)$. Since $x_i\in c_i$, by
Proposition \ref{sec:r-graphs} we can find two geodesic paths starting
from $x_i$, meeting only at $x_i$, with length $R(\rX)-d(x,x_i)$, and
taking values in $c_i\cap \overline{B}_{R(\rX)}(x)$. At least one of
these paths $\gamma_i''$ does not pass through $x$, and so the
concatenation $\gamma_i$ of $\gamma_i'$ and $\gamma''_i$ is an
injective path parameterized by arc-length starting from $x$ and with
length $R(\rX)$. So it is, in fact, a geodesic path, since it takes its
values in $\overline{B}_{R(\rX)}(x)$. We let $y_i$ be the endpoint of
$\gamma_i$, so that $d(x,y_i)=R(\rX)$ for every $i\geq 1$.  Now, we
observe that if $i<j$, the paths $\gamma_i$ and $\gamma_j$ both start from
the same point $x$, but since $\gamma''_i$ takes values in $c_i$, 
since $\gamma''_j$ passes through $x_j\notin c_i$, and since $d(x,x_i)\vee
d(x,x_j)\leq R(\rX)/2$, these paths are disjoint outside the ball
$B_{R(\rX)/2}(x)$. This implies that $d(y_i,y_j)\geq R(\rX)$ for every
$i<j$, and contradicts the compactness of $X$.

Therefore, the union $X_0$ of all embedded cycles is closed and has
finitely many connected components. By definition, $\core(\rX)$ is the
union of $X_0$ together with all simple arcs with endpoints in
$X_0$. Obviously, in this definition, we can restrict our attention to
simple arcs having only their endpoints in $X_0$.  So let $x,y\in X_0$
with $x\neq y$ be linked by an simple arc $A$ taking its values
outside $X_0$, except at its endpoints. Necessarily, $x$ and $y$ must
be in disjoint connected components of $X_0$, because otherwise there would
exist a path from $x$ to $y$ in $X_0$ whose concatenation with
$\gamma$ would create an embedded cycle not included in
$X_0$. Furthermore, there can exist at most one arc $A$, or else it
would be easy to exhibit an embedded cycle not included in
$X_0$. So we see that $\core(\rX)$ is a finite union of simple
arcs and embedded cycles, which is obviously connected, and is thus 
closed. Any point in $\core(\rX)$ has degree at least $2$ by
definition.

It remains to check that the intrinsic metric on $\core(\rX)$ is
given by $d$ itself. Let $x,y\in \core(\rX)$ and $\gamma$ be a
geodesic path from $x$ to $y$. Assume that $\gamma$
takes some value $z=\gamma(t)$ outside $\core(\rX)$. Let
$t_1=\sup\{s\leq t:\gamma(s)\in \core(\rX)\}$ and $t_2=\inf\{s\geq
t:\gamma(s)\in \core(\rX)\}$, so that
$\gamma((t_1,t_2))\cap\core(\rX)=\varnothing$. Since $\core(\rX)$ is
connected, we can join $\gamma(t_1)$ and
$\gamma(t_2)$ by a simple arc included in $\core(\rX)$, and the union
of this arc with $\gamma((t_1,t_2))$ is an embedded cycle not
contained in $\core(\rX)$, a contradiction. 
\endproof

\subsection{The kernel of $\R$-graphs with no leaves}\label{sec:core-metric-gluing}

In this section, $\rX$ is an $\R$-graph with no leaves. 
We now start to prove Theorem \ref{sec:structure-r-graphs-4} on the
structure of such $\R$-graphs. 
The set $k(\rX)=\{x\in X:\deg_X(x)\geq 3\}$ of branchpoints of $X$ forms the vertex set of $\ker(\rX)$. 

\begin{proposition}
  \label{sec:structure-r-graphs-1}
  The set $k(\rX)$ is
  finite, and $\deg_X(x)<\infty$ for every $x\in k(\rX)$.
\end{proposition}

\proof By Proposition~\ref{sec:skeleton-core-kernel}, the number of cycles of $\rX$ is finite. We assume that $\rX$ is not acyclic, and argue by induction on
the maximal number of independent embedded cycles, that is, of embedded cycles
$c_1,\ldots,c_k$ such that $c_i\setminus(c_1\cup\ldots\cup
c_{k-1})\neq\varnothing$ for $1\leq i\leq k$. Plainly, we may and will
assume that for all $i\in \{1,2,\ldots,k\}$, either $c_i$ is disjoint from $c_1\cup \ldots\cup
c_{i-1}$, or $c_i\setminus (c_1\cup \ldots \cup c_{i-1})$ is a simple
arc of the form $\gamma((0,1))$, where $\gamma:[0,1]\to X$ satisfies
$\gamma(0),\gamma(1)\in c_1\cup\ldots\cup c_{i-1}$. 
The result is trivial if $\rX$
is unicyclic ($k=1$). Suppose $X$ has $k$ independent embedded cycles
$c_1,\ldots,c_k$ as above. Consider the smallest connected subset $X'$ of $X$
containing $c_1,\ldots,c_{k-1}$: this subset is the union of
$c_1,\ldots,c_{k-1}$ with some simple arcs having only their endpoints
as elements of $c_1\cup\ldots\cup c_{k-1}$, and is a closed subset
of $X$. 

If $c_k$ does not intersect $X'$, then there exists a unique simple
arc with one endpoint $a$ in $c_k$ and the other endpoint $b$ in $X'$,
and disjoint from $c_k\cup X'$ elsewhere. Then $a,b$ must be elements
of $k(\rX)$: $a$ is the only element of $k(\rX)$ in $c_k$, we have
$\deg_X(a)=3$, and $\deg_X(b)=\deg_{X'}(b)+1$. Therefore, the
number of points in $k(\rX)$ is at most $2+k(\rX')$, where $\rX'$ is
the set $X'$ endowed with the intrinsic metric inherited from $\rX$. This is an $\R$-graph
without leaves and with (at most) $k-1$ independent cycles. 

If on the other hand $c_k\cap X'\neq \varnothing$, then by assumption we
have $X=X'\cup A$, where $A$ is a sub-arc of $c_k$ disjoint from $X'$
except at its endpoints $a,b$. The latter are elements of $k(\rX)$,
and satisfy $\deg_{X}(a)\leq \deg_{X'}(a)+2$ and similarly for $b$
(note that $a,b$ may be equal). After we remove $A\setminus
\{a,b\}$ from $X$, we are left with an $\R$-graph $\rX'$ (in the
induced metric) without leaves, and with at most $k-1$ independent cycles. 

The result follows by induction on $k$. 
\endproof

If $\rX$ is unicyclic, then $X$ is in fact identical to its unique
embedded cycle $c$. In this case, $k(\rX)=\varnothing$, and we let
$e(\rX)=\{c\}$. If $\rX$ has at least two distinct embedded cycles,
then the previous proof entails that $k(\rX)\neq \varnothing$, and
more precisely that every embedded cycle contains at least one point
of $k(\rX)$. The set $X\setminus k(\rX)$ has finitely many connected
components (in fact, there are precisely $\frac 1 2 \sum_{x\in
  k(\rX)}\deg_X(x)$ components, as the reader is invited to verify),
which are simple arcs of the form $\gamma((0,1))$, where
$\gamma:[0,1]\to X$ is such that $\gamma$ is injective on $[0,1)$, such that
$\gamma(0),\gamma(1)\in k(\rX)$, and such that $\gamma((0,1))\cap
k(\rX)=\varnothing$.  We let $e(\rX)$ be the set of the
closures of these connected components, i.e.\ the arcs $\gamma([0,1])$
with the above notation, which are called the \emph{kernel edges}. 
\nomenclature[Ex]{$e(\rX)$}{The edges of the kernel $\ker(\rX)$.}
The multigraph $\ker(\rX)=(k(\rX),e(\rX))$ is the \emph{kernel} of $\rX$, where
the vertices incident to $e\in e(\rX)$ are, of course, the endpoints of
$e$.  An orientation of the edge $e$ is the choice of a
parametrisation $\gamma:[0,1]\to X$ of the arc $e$ or its reversal
$\gamma(1-\cdot\,)$, considered up to reparametrisations by increasing
bijections from $[0,1]$ to $[0,1]$. If $e$ is given an
orientation, then its endpoints are distinguished as the source and
target vertices, and are denoted by $e^-,e^+$, respectively.  The
  next proposition then follows from the definition of $k(\rX)$.


\begin{proposition}
\label{sec:structure-r-graphs-3}
The kernel of a non-unicyclic $\R$-graph without leaves is a
multigraph of minimum degree at least 3.
\end{proposition}

Finally, we prove Theorem \ref{sec:structure-r-graphs-4}.  Assume that
$\rX$ is a non-unicyclic $\R$-graph without leaves, and let $\ell(e):e\in
e(\rX)$ be the lengths of the kernel edges.  Note that if $x,y\in k(\rX)$, then
\[
d(x,y)=\inf\Bigg\{\sum_{i=1}^k \ell(e_i):(e_1,\ldots,e_k) \mbox{
  a chain from }x \mbox{ to }y \mbox{ in }G\Bigg\}\, ,
  \] 
 where $(e_1,\ldots,e_k)$ is a chain from $x$ to $y$ if it is possible to
orient $e_1,\ldots,e_k\in e(\rX)$ in such a way that $e_1^-=x,e_k^+=y$
and $e_i^+=e_{i+1}^-$ for every $i\in \{1,\ldots,k-1\}$.  Of course,
it suffices to restrict the infimum to those chains that are simple,
in the sense that they do not visit the same vertex twice. Since there
are finitely many simple chains, the above infimum is, in fact,
a minimum.  Next, if $x$ and $y$ are elements of $e$ and $e'$ respectively,
consider an arbitrary orientation of $e,e'$. Then a shortest path from
$x$ to $y$ either stays in $e$ (in this case $e=e'$), or
passes through at least one element of $k(\rX)$ incident to $e$, and likewise for $e'$.  Therefore,
$$d(x,y)=d_e(x,y)\wedge \min_{s,t\in \{-,+\}} 
\Big\{d_e(x,e^s)+d(e^s,(e')^{t})+d_{e'}((e')^{t},y)\Big\}\, ,$$
where we let $d_e(a,b)$ be the length of the arc of $e$ between $a$
and $b$ if $a,b\in e$, and $\infty$ otherwise. It is shown in 
 \cite[Section 3]{burago01} that this formula gives the
distance for the metric gluing of the graph with edge-lengths
$(k(\rX),e(\rX),(\ell(e),e\in e(\rX)))$. This proves Theorem~\ref{sec:structure-r-graphs-4}.

\subsection{Stability of the kernel in the Gromov--Hausdorff
  topology}\label{sec:convergence-core}

In this section, we show that kernels of $\R$-graphs are stable
under small perturbations in the Gromov--Hausdorff metric, under an
assumption which says, essentially, that the girth is uniformly bounded
away from $0$. 

Recall from Section \ref{sec:breaking-cycles-r} that
$\mathcal{A}_r$ is the set of measured $\R$-graphs $\rX$ such that 
$$
\min_{e\in e(\rX)}\ell(e)\geq r, \mbox{ }\sum_{e\in
  e(\rX)}\ell(e)\leq 1/r\mbox{ and }s(\rX)\leq
1/r\, ,$$ where it is understood in this definition that the
unicyclic $\R$-graphs (those with surplus $1$) are such that their
unique embedded cycle has length in $[r,1/r]$.  
It follows that the sets
$\mathcal{A}_r,0<r<1$, are decreasing, with union the set of all
measured $\R$-graphs. If $[X,d]$ is an $\R$-graph, we write $[X,d]\in
\AA_r$ if $[X,d,0]\in \AA_r$.  Note that an element $\rX\in
\mathcal{A}_r$ has $\mathrm{gir}(\rX)\geq r$.  Likewise, we let
$\mathcal{A}_r^\bullet$ be the set of (isometry equivalence classes
of) safely pointed measured $\R$-graphs $(X,d,x,\mu)$ with
$(X,d,\mu)\in \mathcal{A}_r$ (recall Definition
\ref{sec:cutting-cycles-an}), and say that a pointed $\R$-graph $(X,d,x) \in \AA_r^{\bullet}$ if $(X,d,x,0) \in \AA_r^{\bullet}$. 
\nomenclature[Arbullet]{$\mathcal{A}_r^\bullet$}{Set of safely pointed elements of $\mathcal{A}_r$.}

A subset $A$ of $X$ is said to be in correspondence with a subset $A'$
of $X'$ via $C\subset X\times X'$ if $C\cap (A\times A')$ is a
correspondence between $A$ and $A'$.  Let $\rX$ and $\rX'$ be
$\R$-graphs with surplus at least $2$. Given $C \in
C(\rX,\rX')$, for $\eps > 0$ we say
that $C$ is a \emph{$\eps$-overlay} (of $\rX$ and $\rX'$) if $\dis(C)
< \eps$, 
\nomenclature[Epsilonoverlay]{$\eps$-overlay}{See definition in text.}
and there exists a multigraph isomorphism $\chi$ between
$\ker(\rX)$ and $\ker(\rX')$ such that:
\begin{enumerate}
\item
For every $v\in k(\rX)$, 
$(v,\chi(v)) \in C$.
\item For every $e\in e(\rX)$, the edges $e$ and $\chi(e)$ are
  in correspondence via $C$, and 
$$|\ell(e)-\ell(\chi(e))|\leq
  \eps\, .$$
\end{enumerate}
If $s(\rX)=s(\rX')=1$, an $\eps$-overlay is a correspondence with
distortion at most $\eps$, such that the unique embedded cycles $c,c'$
of
$\rX$ and $\rX'$ are in correspondence via $C$, and
$|\ell(c)-\ell(c')|\leq \eps$. 
Finally, if $s(\rX)=s(\rX')=0$ then an $\eps$-overlay is just a correspondence of distortion 
at most~$\eps$.

\begin{proposition}\label{lem:deltaover}
Fix $ r \in(0, 1)$. For every $\eps>0$ there exists $\delta>0$ such that 
if $\rX=(X,d)$ and $\rX'=(X',d')$ are elements of $\AA_r$ 
and $C \in C(\rX,\rX')$ has $\dis(C) \le \delta$, then there exists an $\eps$-overlay $C' \in C(\rX,\rX')$ with $C \subset C'$.
\end{proposition}

We say that a sequence of graphs with
edge-lengths $((V_n,E_n,(l_n(e),e\in E_n)),n\geq 1)$ converges to the
graph with edge-lengths $(V,E,(l(e),e\in E))$ if $(V_n,E_n)$ and
$(V,E)$ are isomorphic 
for all but finitely many $n\geq 1$, through an isomorphism $\chi_n$
such that $l_n(\chi_n(e))\to l(e)$ as
$n\to\infty$ for every $e\in E$. We now state some consequences of Proposition~\ref{lem:deltaover} which 
are used in the proof of Theorem~\ref{thm:gnp-converge} and in Section~\ref{sec:cuttingrgs}, before proceeding to the proof of Proposition~\ref{lem:deltaover}. 
Recall the definition of the distances $\dghp^{k,l}$ from Section~\ref{sec:notions-convergence}.

\begin{corollary}
\label{sec:cutt-safely-point-1}
Fix $r\in (0 , 1)$. Let $(\rX^n=(X^n,d^n,\mu^n),n\geq 1)$ and $\rX=(X,d,\mu)$
be elements of $\mathcal{A}_r$. Suppose that
$\dghp(\rX^n,\rX)\to 0,$ as $n\to\infty$.
\\
\noindent{\rm (i)} Then $\ker(\rX^n)$
 converges to $\ker(\rX)$ as a graph with edge-lengths. As a
 consequence, $r(\rX^n) \to r(\rX)$, and writing $\Ell^n$ (resp.\ $\Ell$) for the
 restriction of the length measure of $\rX^n$ (resp.\ $\rX$) to
 $\conn(\rX^n)$ (resp.\ $\conn(\rX)$), 
 it holds that
$$\dghp^{0,2}((X^n,d^n,\mu^n,\Ell^n),(X,d,\mu,\Ell))
\underset{n\to\infty}{\longrightarrow} 0\, .$$

\noindent{\rm (ii)} 
Let $x^n$ be a
random variable in $X^n$ with distribution
$\Ell^n/\Ell^n(\conn(\rX^n))$ and $x$ be a random variable in
$\rX$ with distribution $\Ell/\Ell(\conn(\rX))$. Then as $n \to \infty$, 
$$(X^n,d^n,x^n,\mu^n)
\convdist(X,d,x,\mu)\, $$
in the space $(\cM^{1,1},\dghp^{1,1})$.
\end{corollary}

The above results rely on the following lemma. Given metric spaces
$(X,d)$ and $(X',d')$, $C\subset X\times X'$ and $r>0$,
let 
\[
C_r=\big\{(y,y')\in X\times X'~:~d(x,y)\vee d'(x',y')\le r~\text{for~}(x,x')\in C\big\}. 
\]
$C_r$ is the
$r$-enlargement of $C$ with respect to the product distance. 
\nomenclature[Cr]{$C_r$}{The $r$-enlargement of correspondence $C$.}
Note that if $C$ is a
correspondence between $X$ and $X'$, then $C_r$ is also a
correspondence for every $r>0$. Moreover, $\dis(C_r)\leq
\dis(C)+4r$.  A mapping $\phi:[a,b]\to [a',b']$ is called {\em bi-Lipschitz} if
$\phi$ is a bijection such that $\phi$ and $\phi^{-1}$ are Lipschitz, and we
call the quantity
$$K(\phi)=\inf\Big\{K>1:K^{-1}|x-y|\leq |\phi(x)-\phi(y)|\leq K|x-y|\mbox{ for
 every }x,y\in [a,b]\Big\}$$
the {\em bi-Lipschitz constant} of $\phi$. 
\nomenclature[Kphi]{$K(\phi)$}{Bi-Lipschitz constant of $\phi$.}
By convention, we let $K(\phi)=\infty$
if $\phi$ is not a bijection, or not bi-Lipschitz. 

\begin{lemma}
  \label{sec:cutting-r-graphs}
  Fix $r\in (0,1)$ and let $(X,d),(X',d')\in\mathcal{A}_r$. Suppose there
  exists a correspondence $C$ between $X$ and $X'$ such that
  $\dis(C)<r/56 $.  

  Let $x,y\in X$ be two distinct points in $X$, and let $f$ be a local
  geodesic from $x$ to $y$. Let $x',y'\in X'$ be such that
  $(x,x'),(y,y')\in C$. Then there exists a local geodesic $f'$ from
  $x'$ to $y'$ with 
$$\len(f')\leq \bigg(1+\frac{64 \dis(C)}{r\wedge \len(f)}\bigg)\cdot \len(f)\,
,$$ 
and a bi-Lipschitz mapping $\phi:[0,\len(f)]\to
  [0,\len(f')]$ such that $(f(t),f'(\phi(t)))\in C_{8\dis(C)}$ for
  every $t\in [0,\len(f)]$, and
$$K(\phi)\leq \bigg(1-\frac{64 \dis(C)}{r\wedge \len(f)}\bigg)_+^{-1}\,
.$$
\end{lemma}
Note that the second part of the statement also implies a lower bound
on the length of $f'$, namely, 
$$ \len(f')\geq K(\phi)^{-1}\len(f)\geq \len(f)
\bigg(1-\frac{64 \dis(C)}{r\wedge \len(f)}\bigg)_+\, ,$$
which is, of course, useless when $r \wedge \len(f) \le 64\dis(C)$.

\proof Let us first assume that $0<\len(f)\leq r/8 $, so in particular
$d(x,y)\leq R(\rX)$ and $f$ is the geodesic from $x$ to $y$.  We have
\[
d(x,y)-\dis(C)\leq d'(x',y') \leq d(x,y)+\dis(C)\leq r/8 +\dis(C)<R(\rX')\, ,
\]
so that $x'$ and $y'$ are linked by a unique geodesic $f'$. 
Set $\phi(t)=d'(x',y')t/d(x,y)$ for $0\leq t\leq d(x,y)$. From the
preceding chain of inequalities,
we obtain that 
\[
\len(f')\leq \len(f)+\dis(C)\, ,\qquad \mbox{ and }\quad K(\phi)\leq \bigg(1-\frac{\dis(C)}{d(x,y)}\bigg)_+^{-1}\, .\] 

Fix $z=f(t)\in \im(f)$ and let $z''$ be such that $(z,z'')\in
C$. Then $d'(x',z'')\leq d(x,z)+\dis(C)<r/4$, so that $z''$ belongs to the
$\R$-tree $B_{R(\rX')}(x')$. Let $z'$ be the (unique) point of
$\im(f')$ that is closest to $z''$. Then a path from $x'$ or $y'$ to
$z''$ must pass through $z'$, from which we have
\begin{equation}
  \label{eq:2}
  d'(z'',z')=\frac{d'(x',z'')+d'(y',z'')-d'(x',y')}{2}\leq
\frac{3}{2}\dis(C)\, .
\end{equation}
Therefore,
\[
t-\frac{5}{2}\dis(C)\leq d'(x',z'')-d'(z'',z')\leq d'(x',z')\leq
d'(x',z'')\leq t+\dis(C)\, ,
\]
so after a short calculation we get that 
\[
\Big|d'(x',z')-\frac{d'(x',y')}{d(x,y)}t\Big|\leq
\frac{7}{2}\dis(C)\, .
\]
From this we obtain 
\[d'(z',f'(\phi(t)))=d'\pran{f'(d'(x',z')),f'\pran{\frac{d'(x',y')}{d(x,y)}t}}=\Big|d'(x',z')-
\frac{d'(x',y')}{d(x,y)}t\Big|\leq \frac{7}{2}\dis(C)\, ,
\] so that, in
conjunction with \eqref{eq:2}, we have $(f(t),f'(\phi(t)))\in
C_{5\dis(C)}$

We next assume that $\len(f)>r/8 $.  Fix an integer $N$ such that
$r/16 <\len(f)/N\leq r/8 $ and let $t_i=i\, \len(f) /N $ and $x_i=f(t_i)$
for $0\leq i\leq N$. By Proposition \ref{sec:r-graphs}, since $f$ is a
local geodesic and $t_{i+1}-t_{i-1}<R(\rX)$ for every $i\in
\{1,\ldots,N-1\}$, 
the restriction
$f|_{[t_{i-1},t_{i+1}]}$ must be a shortest path and so
\[
d(x_{i-1},x_{i+1})=\frac{2\len(f)}{N}\leq r/4\, ,\qquad
d(x_i,x_{i+1})=\frac{\len(f)}{N}\in [r/16 , r/8 ]\, .
\] 
Letting $x''_i$ be a point such that $(x_i,x''_i)\in C$ (where we
always make the choice $x''_0=x'$ and
$x''_N=y'$), we have $d'(x''_i,x''_{i+1})\leq d(x_i,x_{i+1})+\dis(C)< R(\rX')$, so
that we can consider the unique geodesic $f''_i$ between
$x''_i$ and $x''_{i+1}$.  The concatenation of the paths
$f''_0,f''_1,\ldots,f''_{N-1}$ is not necessarily a local geodesic,
but by excising certain parts of it we will be able to recover a local
geodesic between $x'$ and $y'$.  For each
$i\in\{1,\ldots,N-1\}$, the sets $\im(f''_{i-1})$ and $\im(f''_i)$ are
included in the $\R$-tree $B_{R(\rX')}(x''_i)$, and the concatenation
of $f''_{i-1}$ and $f''_i$ is a path from $x''_{i-1}$ to $x''_{i+1}$
which, as such, must contain the image of the geodesic $g_i$ between
these points. Let $x'_i$ be the unique point of $\im(g_i)$ that is
closest to $x''_i$, and let $x'_0=x',x'_N=y'$. Then
\[
d'(x'_i,x''_i)=\frac{d'(x''_{i-1},x''_i)+d'(x''_{i+1},x''_i)-d'(x''_{i-1},x''_{i+1})}{2}\leq 
\frac{3}{2}\dis(C)\, ,
\]
so that, for $i\in \{0,1,\ldots,N\}$, 
\[
d(x_i,x_{i+1})-\dis(C)\leq d'(x''_i,x''_{i+1})\leq
d'(x'_i,x'_{i+1})\leq d'(x''_i,x''_{i+1})+3\dis(C)\leq
d(x_i,x_{i+1})+4\dis(C)\, .
\]
If $x_{i+1}' \in \im(f_{i-1}'')$ then 
\[
d'(x_{i-1}'',x_{i+1}'') \le d'(x_{i-1}'',x_i'')+\frac{3}{2}\dis(C) \le \frac{\len(f)}{N}+\frac{5}{2}\dis(C)\, .
\]
However, since $(x_{i-1},x_{i-1}''),(x_{i+1},x_{i+1}'') \in C$ and $\dis(C) < r/56 < 2\len(f)/(7N)$, we have 
\[
d'(x_{i-1}'',x_{i+1}'') \ge \frac{2\len(f)}{N} - \dis(C) > \frac{\len(f)}{N}+\frac{5}{2}\dis(C),
\]
so, in fact, $x_{i+1}' \not\in{\im(f_{i-1}'')}$ and, in particular, $x_{i+1}'$ does not lie on the shortest path between $x_{i-1}'$ and $x_i'$. 
From this it follows that if $f'$ denotes the concatenation of the
geodesic $f'_i$ between $x'_i$ and $x'_{i+1}$, for $0\leq i\leq N-1$,
then $f'$ is a local geodesic between $x'$ and $y'$. Its length is
certainly bounded by the sum of the lengths of the paths $f_i''$, so that 
\[
\len(f')\leq \sum_{i=1}^Nd(x_i,x_{i+1})+4N\dis(C)\leq
\len(f)+\frac{64 \len(f)\dis(C)}{r}\, ,
\]
as claimed. 
Next, we let
$\phi_i(t)=d'(x'_i,x'_{i+1})t/d(x_i,x_{i+1})$, so that 
\[
K(\phi_i)\leq \bigg(1-\frac{4\dis(C)}{d(x_i,x_{i+1})}\bigg)^{-1}\leq
\bigg(1-\frac{64 \dis(C)}{r}\bigg)^{-1}\, .
\]
If $\phi:[0,\len(f)]\to
[0,\len(f')]$ is the concatenation of the mappings $\phi_i,0\leq i\leq
N-1$, then $\phi$ is bi-Lipschitz with the same upper-bound for
$K(\phi)$ as for each $K(\phi_i)$.  Finally, we note that
$f'\circ\phi$ is the concatenation of the paths
$f_i'\circ\phi_i$. If $z=f_i(t)$, we let $z''$ be such that
$(z,z'')\in C$ and let $z'$ be the point in $\im(f_i')$ that is
closest to $z''$. Then similar arguments to before entail that
$d'(z',z'')\leq 3\dis(C)$, so that 
\[
t-4\dis(C)\leq d'(x'_i,z')\leq d'(x'_i,z'')\leq t+\dis(C)\, ,
\]
which implies that
\[
d'(z',f'_i(\phi'_i(t)))=\bigg|d'(x'_i,z')-\frac{d'(x'_i,x'_{i+1})}{d(x_i,x_{i+1})}t\bigg|\leq
5\dis(C)\, ,
\]
and we conclude that $(f_i(t),f_i'(\phi_i(t)))\in C_{8\dis(C)}$, and
so that $(f(s),f'(\phi(s)))\in C_{8\dis(C)}$ for every $s\in
[0,\len(f)]$. 
\endproof

\begin{proof}[Proof of Proposition \ref{lem:deltaover}] 
We will prove this result only when one of 
$\rX$ and $\rX'$ (and then, in fact, both) has surplus at least $2$, leaving the 
similar and simpler case of surplus $1$ to the reader 
(the case of surplus $0$ is trivial). Also, we may assume without loss of generality 
that $\eps < r/4$. 
  
Fix $\eps\in (0,r/4)$, and fix any $\delta  \in (0,\eps r^2/128)$. 
Also, fix $\rX,\rX' \in \AA_r$ 
and a correspondence $C \in
C(X,X')$ with $\dis(C) <\delta$.  
List the
elements of $k(\rX)$ as $v_1,\ldots,v_n$, and fix elements
$v_1'',\ldots,v_n''$ of $X'$ with $(v_i,v_i'') \in C$ for each $1 \le
i \le n$. Since $\dis(C)<\delta$ and $v_1,\ldots,v_n$ are pairwise at
distance at least $r$, $v_1'',\ldots,v_n''$ are pairwise at distance at
least $r-2\delta > r/2$ and, in particular, are all distinct.  Next,
for every $e\in e(\rX)$, say with $e^+=v_i,e^-=v_j$, 
fix a local geodesic $f_e$ between $v_i$ and $v_j$ with $\im(f_e)=e$ and $f_e(0)=e^-$.
By Lemma \ref{sec:cutting-r-graphs}, there
exists a geodesic $f''_e$ from $v''_i$ to $v''_j$ and a
bi-Lipschitz mapping $\phi_e:[0,\ell(e)]\to [0,\len(f''_e)]$ with
$$K(\phi_e)\leq \pran{1-\frac{64 \delta}{r}}^{-1}< 2\, ,$$
and such that
$(f_e(t),f''_e(\phi_e(t)))\in C_{8\delta}$ for every $t\in
[0,\ell(e)]$. In particular, it follows that $\len(f''_e)>r/2$.
Then we claim that for $\delta$ small enough, the following two properties hold.
\begin{enumerate}
\item For every $e\in e(\rX)$, the path $(f''_e(t),\eps/8 \leq t\leq
  \len(f''_e)-\eps/8 )$ is injective.
\item 
For $e_1, e_2\in e(\rX)$ with $e_1\neq e_2$, we have 
\[
\{f''_{e_1}(t): \eps/8 \leq t\leq \len(f''_{e_1})-\eps/8 \}\cap
  \{f''_{e_2}(t):\eps/8 \leq t\leq \len(f''_{e_2})-\eps/8 \}=\varnothing\, .
  \]
\end{enumerate}
To establish the first property, suppose that
$f''_e(t)=f''_e(t')$ for some $e\in
e(\rX)$ and distinct $t,t'\in [\eps/8,\len(f''_e)-\eps/8]$.
For concreteness, let us assume that $e^-=v_i$ and $e^+=v_j$.  Since $f''_e$
is a local geodesic, this implies that $|t-t'|\geq R(\rX')\geq r/4$. 
Moreover, since 
$(f_e(\phi_e^{-1}(t)),f''_e(t)),(f_e(\phi_e^{-1}(t')),f_e''(t'))\in
C_{8\delta}$ and since $\delta < \eps/128$, we have
\begin{equation}
  \label{eq:3}
  d(f_e(\phi_e^{-1}(t)),f_e(\phi_e^{-1}(t')))\leq
d'(f''_e(t),f_e''(t'))+8\delta=8\delta<\eps/16 \, .
\end{equation}
On the other hand, we have 
\[
|\phi_e^{-1}(t)-\phi_e^{-1}(t')|\geq K(\phi_e)^{-1}|t-t'|\geq
\frac{1}{2}|t-t'|\geq r/8 > \eps/2\, ,
\]
and since $\phi_e^{-1}(0)=0$ and $\phi_e^{-1}(\len(f_e''))=\ell(e)$, 
\[
|\phi_e^{-1}(t)|\geq K(\phi_e)^{-1} t> \eps/16 \, ,\qquad
|\phi_e^{-1}(t)-\ell(e)|> \eps/16 \, ,
\]
and similarly for $t'$.
But if $s,s'\in
[\eps/16 ,\ell(e)-\eps/16 ]$, then $d(f_e(s),f_e(s'))\geq 
(\eps/8 )\wedge |s-s'|$, because a path from $f_e(s)$ to $f_e(s')$ is
either a subarc of $e$, or passes through both vertices $v_i$ and
$v_j$. It follows that
$d(f_e(\phi_e^{-1}(t)),f_e(\phi_e^{-1}(t'))\geq \eps/8 $, in
contradiction with \eqref{eq:3}. This yields that property 1 holds.

The argument for property 2 is similar: for every $t_1\in
(\eps/8 ,\len(f''_{e_1})-\eps/8 )$ and
$t_2\in(\eps/8 ,\len(f''_{e_2})-\eps/8 )$, there exist
$x_1 \in e_1$ and $x_2 \in e_2$ such that $(x_1,f''_{e_1}(t_1))$ and
$(x_2,f''_{e_2}(t_2))$ are in $C_{8\delta}$. Then the distance from
$x_1,x_2$ to $k(\rX)$ is at least $\eps/16 $ so that
$d(x_1,x_2)\geq \eps/8 $. From this, we deduce that
$d'(f''_{e_1}(t_1),f''_{e_2}(t_2))\geq d(x_1,x_2)-8\delta>0$.

Next, for every $i\in \{1,\ldots,n\}$, consider the points $f''_e(\eps/8 ), e\in
e(\rX)$ for which $e^-=v_i$, as well as the points
$f''_e(\len(f''_e)-\eps/8 )$ for which $e^+=v_i$. These points are on the
boundary of the ball $\overline{B}_{\eps/8 }({v''_i})$, which we recall is an
$\R$-tree. Let $T_i$ be the subtree of $\overline{B}_{\eps/8 }({v''_i})$
spanned by these points. Then property 1 above shows that 
\[
\bigcup_{1\leq i\leq n}T_i \cup \bigcup_{e \in e(\rX)} \{f''_e(t),\eps/8\leq t\leq \len(f''_e)-\eps/8\}
\]
induces a closed subgraph of $(X',d')$ without leaves, and so this subgraph is in fact a subgraph of 
$\core(\rX')$. Furthermore, property 2 implies that the points of
degree at least $3$ in this subgraph can only belong to
$\bigcup_{1\leq i\leq n}T_i$. Since any such point is then an element
of $k(\rX')$ and $\diam (T_i)\leq \eps/4 <r$, we see that each $T_i$ can
contain at most one element of $k(\rX')$. On the other hand, each
$T_i$ must contain at least one element of $k(\rX')$ because  
$T_i$ has at least three leaves (since $v_i$ has degree at least $3$). 
Thus, each $T_i$ contains exactly one element of $k(\rX')$, which we denote by $v_i'$. 
Next, for $e \in e(\rX')$, 
if $e^-=v_i,e^+=v_j$, then we let
$f'_e$ be the simple path from $v'_i$ to $v'_j$ that has a non-empty
intersection with $f''_e$. It is clear that this path is well-defined
and unique. Letting $\chi(v_i)=v'_i$ for $1 \le i \le n$, and letting $\chi(e)=\im(f'_e)$ for $e \in e(\rX)$, we have therefore 
defined a multigraph homomorphism from $\ker(\rX)$ to $\ker(\rX')$, and this
homomorphism is clearly injective. By symmetry of the roles of $X$ and
$X'$, we see that $|k(\rX)|=|k(\rX')|$ and $|e(\rX)|=|e(\rX')|$,
and so $\chi$ must, in fact, be a multigraph isomorphism.

Finally, since $\len(f_e) = \ell(e) \le 1/r$, we have 
\[
|\ell(e) - \len(f''_e)| = |\len(f_e)-\len(f''_e)|\leq
\frac{64\delta}{r}
\len(f_e)
\leq
\frac{64\delta}{r^2} < \frac{\eps}{2}\, ,
\] 
by our choice of $\delta$.
But, by construction,
$|\len(f''_e)-\ell(\chi(e))|<\eps/2$, since the endpoints of $\chi(e)$ each have distance at most $\eps/4$ from an endpoint of $f''(e)$. It follows that 
$|\ell(e)-\ell(\chi(e))|<\eps$.
Finally, since every point of $\chi(e)$ is within distance $\eps/4$ of $f_e''$ and 
$e=\im(f_e)$ and $\im(f_e'')$ are in correspondence via $C_{8\delta}$, it follows 
that $e$ and $\chi(e)$ are in correspondence via $C_{8\delta+\eps/4}$. Since 
$\dis(C_{8\delta+\eps/4}) < \dis(C)+16\delta+\eps/2 < \eps$, this completes the proof.
\end{proof}

\begin{proof}[Proof of Corollary \ref{sec:cutt-safely-point-1}] 
  Again we only consider the case $s(\rX)>1$, the case $s(\rX)=1$
  being easier (and the case $s(\rX)=0$ trivial).  
 
  Let $(\rX_n,n \ge 1)$ and $\rX$ be as in the statement of Corollary~\ref{sec:cutt-safely-point-1}. 
  Let $(C^n,n \ge 1)$ and $(\pi^n,n \ge 1)$ be sequences of correspondences and of measures, respectively, 
  such that $\dis(C^n),\pi^n((C^n)^c)$ and $D(\pi^n;\mu^n,\mu)$ each converge to $0$ as $n\to\infty$.
  The fact
  that $\ker(\rX^n)$ converges to $\ker(\rX)$ as a graph with
  edge-lengths is then an immediate consequence of Proposition
  \ref{lem:deltaover}: for each $n$ sufficiently large, simply 
  replace $C^n$ by $C^n \cup \hat{C}^n$, where $\hat{C}^n$ is 
  an $\eps_n$-overlay of $\rX$ and $\rX^n$, 
  for some sequence $\eps_n\to 0$ (we may assume $\eps_n \ge \dis(C^n)$).
  We continue to write $C^n$ instead of $C^n \cup \hat{C}^n$, and note that enlarging 
  $C^n$ diminishes $\pi_n((C^n)^c)$.

In particular, we obtain that for all large enough $n$, there is an isomorphism
$\chi_n$ from $(k(\rX),e(\rX))$ to $(k(\rX^n),e(\rX^n))$ such that
$\ell(e)-\ell(\chi_n(e))$ converges to $0$ for every $e\in e(\rX)$. 
The fact that $r(\rX^n) \to r(\rX)$ is immediate. 
We now fix a particular orientation of the edges, and view 
$\chi_n$ as an isomorphism of oriented graphs, in the sense that
$\chi_n(e^-)=\chi_n(e)^-$. 

For each $e \in e(\rX)$, let $f_e$ be a local geodesic between $e^-$
and $e^+$ with $f_e(0)=e^-$ and $f_e(\ell(e))=e^+$ and, for each $n
\ge 1$ and $e \in e(\rX^n)$, define $f_e$ accordingly. Then for each
$n$ sufficiently large, define a mapping $\Phi_n$ with domain
$\mathrm{dom}(\Phi_n) = \bigcup_{e \in e(\rX)}f_e([0,\ell(e)-\eps_n])$
by setting $\Phi_n(f_e(t))=f_{\chi_n(e)}^n(t)$ for each $e \in e(\rX)$
and each $0 \le t \le \ell(e)-\eps_n$.

By considering a small enlargement of $C^n$, or, equivalently, by letting $\eps_n$ tend to zero sufficiently slowly, we
may assume without loss of generality that $(x,\Phi_n(x))\in C^n$ for
all $x \in \mathrm{dom}(\Phi_n)$.
This comes from the fact that $e$
and $\chi_n(e)$ are in correspondence via $C^n$; we leave the details of this
verification to the reader. It follows that the relation $\{(x,\Phi_n(x)):x \in \mathrm{dom}(\Phi_n)\}$ is a subset of $C^n$.

Let $e_c(\rX)$ be the set of edges $e \in e(\rX)$ whose removal
from $e(\rX)$ does not disconnect
$\ker(\rX)$. Clearly, $\conn(X)\subseteq k(\rX)\cup\bigcup_{e\in
  e_c(\rX)}e$, and the measure $\Ell$ is carried by $\bigcup_{e\in
  e_c(\rX)}e$ (in fact, it is carried by the subset of
points of $\bigcup_{e\in e_c(\rX)}e$ with degree $2$, by Proposition
\ref{sec:skeleton-core-kernel-2}). Let $\Ell'$ be the restriction
of $\Ell$ to the set $\mathrm{dom}(\Phi_n)$, which has total mass $\sum_{e\in
  e_c(\rX)}(\ell(e)-\eps_n)$. We
consider the push-forward $\rho_n$ of $\Ell'$ by the mapping
$x\mapsto (x,\Phi_n(x))$ from $X$ to $X\times X^n$. Then the second
marginal of $\rho_n$ is the restriction of $\Ell^n$ to
$\bigcup_{e\in e_c(\rX)}\im(f^n_e)$, 
so that 
$$D(\rho_n;\Ell,\Ell^n)\leq  
\sum_{e\in e(\rX)}(\eps_n+|\ell(e)-\ell(\chi_n(e))|) \, .$$ The latter
converges to $0$ by the convergence of the edge-lengths. It only
remains to note that $\rho_n(X\times X^n\setminus C^n)=0$ by
construction. This yields (i).

Finally, (i) implies 
that $(X^n,d^n,\mu^n,\Ell^n/\Ell^n(\conn(\rX^n)))$ converges to
$(X,d,\mu,\Ell/\Ell(\conn(\rX)))$ in the metric $\dghp^{0,2}$
introduced in Section \ref{sec:notions-convergence}, and (ii)
then follows from Proposition \ref{sec:grom-hausd-prokh}. 
\end{proof}

\section{Cutting safely pointed $\R$-graphs}\label{sec:ckcc}

In this section, we will consider a simple cutting procedure on
$\R$-graphs, and study how this procedure is perturbed by small
variations in the Gromov--Hausdorff distance.

\subsection{The cutting procedure}\label{sec:cutting-procedure}

Let $(X,d)$ be an $\R$-graph, and let $x\in \conn(\rX)$. We endow the
connected set $X\setminus\{x\}$ with the intrinsic distance
$d_{X\setminus \{x\}}$: more precisely, $d_{X\setminus \{x\}}(y,z)$ is
defined to be the minimal length of an injective path not visiting
$x$. This is indeed a minimum because there are finitely many
injective paths between $y$ and $z$ in $X$, as a simple consequence of
Theorem \ref{sec:structure-r-graphs-4} applied to $ \core(\rX)$.  The
space $(X\setminus \{x\},d_{X\setminus \{x\}})$ is not complete, so we
let $(X_x,d_x)$ be its metric completion as in Section~\ref{sec:breaking-cycles-r}. This space is connected, and
thus easily seen to be an $\R$-graph. We call it
the \emph{$\R$-graph $(X,d)$ cut at the point $x$}. 

From now on, we will further assume that $\deg_X(x)=2$, 
so that $(X,d,x)$ is safely pointed as in Definition~\ref{sec:cutting-cycles-an}.
In this case, one can provide a more detailed description
of $(X_x,d_x)$.  A Cauchy sequence $(x_n,n\geq 1)$ in $(X\setminus
\{x\},d_{X\setminus \{x\}})$ is also a Cauchy sequence in $(X,d)$,
since $d\leq d_{X\setminus \{x\}}$. If its limit $y$ in $(X,d)$ is
distinct from $x$, then it is easy to see that $d_{X\setminus
  \{x\}}(x_n,y)\to 0$, by considering a ball $B_\eps(y)$ not
containing $x$ within which $d=d_{X\setminus\{x\}}$.

So let us assume that $(x_n,n\geq 1)$ converges to $x$ for the
distance $d$. Since $x$ has degree $2$, the $\R$-tree
$B_{R(\rX)}(x)\setminus \{x\}$ has exactly two components, say
$Y_1,Y_2$. It is clear that $d_{X\setminus \{x\}}(z_1,z_2)\geq 2R(\rX)$
for every $z_1\in Y_1,z_2\in Y_2$. Since $(x_n,n\geq 1)$ is a Cauchy
sequence for $(X\setminus \{x\},d_{X\setminus \{x\}})$, we conclude
that it must eventually take all its values in precisely one of $Y_1$ and
$Y_2$, let us say $Y_1$ for definiteness.  Note that the restrictions
of $d$ and $d_{X\setminus \{x\}}$ to $Y_1$ are equal, so that if
$(x'_n,n\geq 1)$ is another Cauchy sequence in $(X\setminus
\{x\},d_{X\setminus \{x\}})$ which converges in $(X,d)$ to $x$ and
takes all but a finite number of values in $Y_1$, then $d_{X\setminus
  \{x\}}(x_n,x'_n)=d(x_n,x'_n)\to 0$, and so this sequence is equivalent
to $(x_n,n\geq 1)$. 

We conclude that the completion of $(X\setminus \{x\},d_{X\setminus
  \{x\}})$ adds exactly two points to $X\setminus \{x\}$,
corresponding to classes of Cauchy sequences converging to $x$ in
$(X,d)$ ``from one side'' of $x$. So we can write $X_x=(X\setminus \{x\})\cup\{x_{(1)},x_{(2)}\}$ and describe $d_x$ as follows: 
\begin{itemize}
\item
 If $y,z\notin \{x_{(1)},x_{(2)}\}$ then $d_x(y,z)$ is the minimal
 length of a path from $y$ to $z$  in $X$ not visiting $x$. 
\item If $y\neq x_{(2)}$ then $d_x(x_{(1)},y)$ is the minimal length of an injective path
  from $x$ to $y$ in $X$ which takes its values in the component $Y_1$
  on some small initial interval $(0,\eps)$, and similarly for
  $d(x_{(2)},y)$ with $y\neq x_{(1)}$.
\item Finally, $d_x(x_{(1)},x_{(2)})$ is the minimal length of an
embedded cycle passing through $x$. 
\end{itemize}

If $(X,d,x,\mu)$ is a pointed measured metric space such that
$(X,d,x)$ is a safely pointed $\R$-graph, and $\mu(\{x\})=0$, then the
space $(X_x,d_x)$ carries a natural measure $\mu'$, such that
$\mu'(\{x_{(1)},x_{(2)}\})=0$ and, for any open subset $A\subseteq
X_x$ not containing $x_{(1)}$ and $x_{(2)}$,  $\mu'(A)=\mu(A)$ if
on the right-hand side we view $A$ as an open subset of $X$. Consequently, there is little risk of ambiguity in using the notation $\mu$ instead of $\mu'$ for
this induced measure. 

We finish this section by proving Lemma \ref{sec:prop-scal-limit-2} on the number of balls required to cover the cut space. 

\begin{proof}[Proof of Lemma \ref{sec:prop-scal-limit-2}]
Let $B_1,B_2,\ldots,B_N$ be a covering of $\rX$ by open balls
of radius $r>0$, centred at $x_1,\ldots,x_N$ respectively.  By definition, any point of $X$ can be joined to the centre of some ball $B_i$ 
by a geodesic path of length $<r$. If such a path does not pass
through $x$, then it is also a geodesic path in $\rX_x$. Now since
$B_1,\ldots,B_N$ is a covering of $\rX$, this implies that any point
$y$ in $X$ can either
\begin{itemize}
\item be joined to some point $x_i$ by a path of length
  $<r$ that does not pass through $x$, 
\item or can be joined to $x$ through a path $\gamma$ of length $<r$. 
\end{itemize}
In the first case, this means that $y$ belongs to the ball 
with centre $x_i$ and radius $r$ in $\rX_x$. In the second case, depending on
whether the initial segment of $\gamma$ belongs to $Y_1$ or $Y_2$,
this means that $y$ belongs to the ball with centre $x_{(1)}$ or
$x_{(2)}$ with radius $r$ in $\rX_x$. This yields a covering of
$\rX_x$ with at most $N+2$ balls, as desired. 

Conversely, it is clear that if $N$ balls are sufficient to cover
$\rX_x$ then the same is true of $\rX$, because distances are smaller
in $\rX$ than in $\rX_x$ (if $x$ is identified with the points
$\{x_{(1)},x_{(2)}\}$). 
\end{proof}

\subsection{Stability of the cutting procedure}\label{sec:stability-cutting}

The following statement will be used in conjunction with Corollary
\ref{sec:cutt-safely-point-1} (ii). Recall the definition of $\mathcal{A}_r^\bullet$ from 
the start of Section~\ref{sec:convergence-core}. 

\begin{theorem} \label{thm:safelypointed} Fix $r \in (0,1)$. Let
  $(X^n,d^n,x^n,\mu^n),n\geq 1$ and $(X,d,x,\mu)$ be 
  elements of $\mathcal{A}_r^\bullet$.  Suppose that
 $$\dghp^{1,1}((X^n,d^n,x^n,\mu^n),(X,d,x,\mu))\underset{n\to\infty}{\longrightarrow} 
0\, ,$$ 
and that $\mu^n(\{x\})=\mu(\{x\})=0$ for every $n$. 
Then
\[
\dghp\pran{(X^n_{x^n},d^n_{x^n},\mu^n),(X_x,d_{x},\mu)}  \underset{n\to\infty}{\longrightarrow} 0\, .
\]
\end{theorem}

Our proof of Theorem~\ref{thm:safelypointed} hinges on two lemmas; to state these lemmas we 
require a few additional definitions. 
Let $\rX=(X,d,x,\mu)\in \mathcal{A}^\bullet_r$ 
and recall the definition of the projection $\alpha:X\to \core(\rX)$. 
For $\eps >0$, write
\[
\widetilde{B}_\eps(x)=\{y\in X: d(\alpha(y) ,x)<\eps\}\, \quad \mbox{
  and }\quad h_\eps(\rX)=\diam (\widetilde{B}_\eps(x))\, ,
\]
so that $B_\eps(x)\subseteq\widetilde{B}_\eps(x)$. The sets
$\widetilde{B}_\eps(x)$ decrease to the singleton $\{x\}$ as
$\eps\downarrow 0$, because $\deg_X(x)=2$. Consequently, $h_\eps(\rX)$
converges to $0$ as $\eps\downarrow 0$.  For $\eps>0$ sufficiently
small, the set $X_{x,\eps}=X\setminus
\widetilde{B}_\eps(x)$, endowed with the intrinsic metric, is an
$\R$-graph. In fact, it is easy to see that for $\eps<R(X)$, this
intrinsic metric is just the restriction of $d_x$ to $X_{x,\eps}$.

\begin{figure}[htb!]
  \centering
 \includegraphics{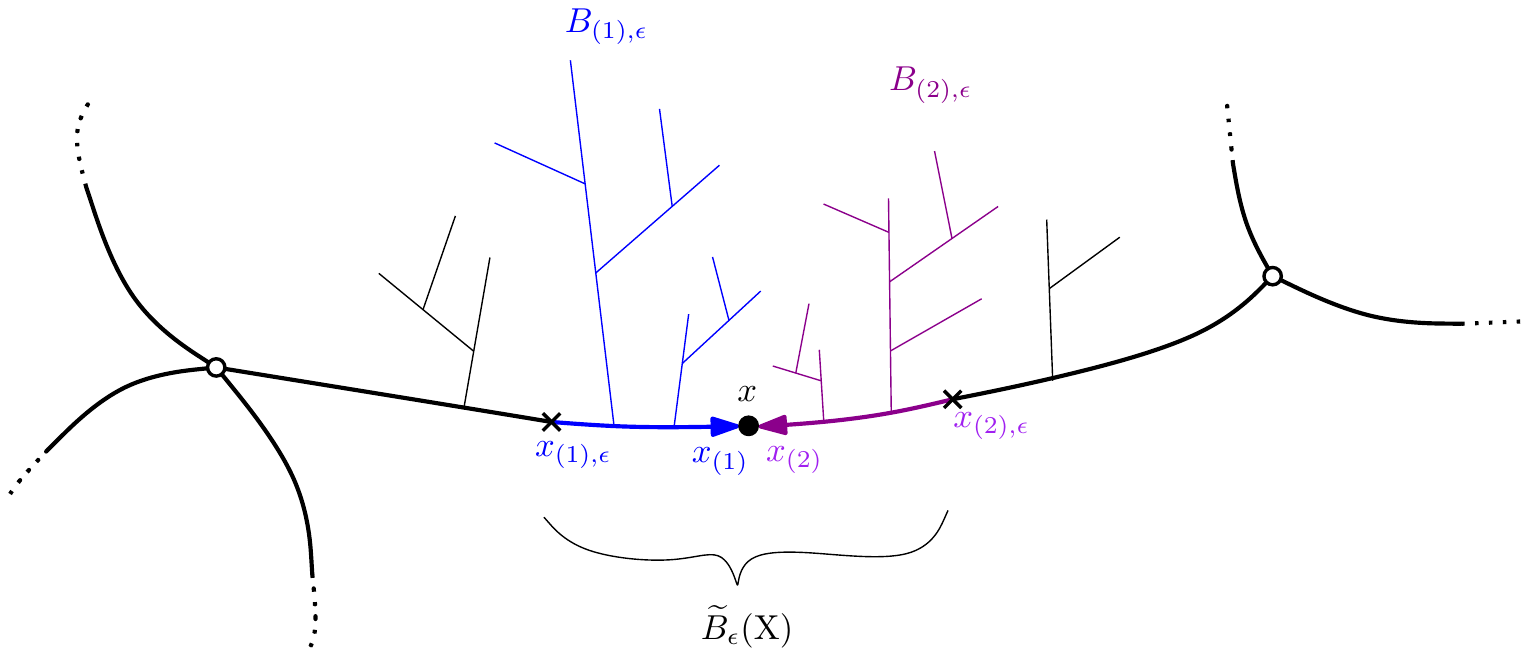} 
  \caption{Part of an $\R$-graph: $\core(\rX)$ is in thicker line.}
  \label{fig:notationcut}
\end{figure}

Let us assume that $\eps<d(x,k(\rX))\wedge R(\rX)$. Let
$x_{(1),\eps},x_{(2),\eps}$ be the two points of $\core(\rX)$ at
distance $\eps$ from $x$, labelled in such a way that, in the notation
of \refS{sec:cutting-procedure}, $x_{(1),\eps}$ is the point closest
to $x_{(1)}$ in $X_x$, or in other words, such that $x_{(1),\eps}\in Y_1$. For $i\in \{1,2\}$
let $P_i$ be the geodesic arc between $x_{(i),\eps}$ and $x$ in $X$. We let $B_{(i),\eps}= \{w\in \widetilde{B}_\eps(x)\setminus \{x\}:\alpha(w)\in P_i\}\cup \{x_{(i)}\}$, which we see as a subset of $X_x$. See Figure~\ref{fig:notationcut} for an illustration.

Now let $\rX' = (X',d',x',\mu')\in\mathcal{A}_r^\bullet$. Just as we defined the space
$\rX_x = (X_x,d_x)$, we define the space
$\rX'_{x'}=(X'_{x'},d'_{x'})$, with $X'_{x'}=(X'\setminus
\{x'\})\cup\{x'_{(1)},x'_{(2)}\}$. We likewise define the sets
$\widetilde{B}_\eps(x')$ and $B'_{(1),\eps},B'_{(2),\eps}$ and the
points $x'_{(1),\eps},x'_{(2),\eps}$ for $\eps <d(x',k(\rX'))\wedge
R(X')$ as above. We will use the same notation $\alpha$ for the
projection $X'\to \core(\rX')$.

\begin{lemma}\label{sec:stab-cutt-proc}
Fix $\delta > 0$. If $C$ is a $\delta$-overlay of $\rX$ and $\rX'$ then 
for every
 $(y,y')\in C$, we have $(\alpha(y),\alpha(y'))\in C_{2\delta}$.
\end{lemma}

\proof
Let $y''$ be such that $(\alpha(y),y'')\in C$. Since $C$ is a $\delta$-overlay, $d'(y'',\core(\rX'))<\delta$. In
particular, if $\alpha(y'')=\alpha(y')$ then we have
$d'(\alpha(y'),y'')< \delta$. Otherwise, a geodesic from $y'$ to
$y''$ must pass through $\alpha(y')$ and $\alpha(y'')$, so that 
$$d'(y',\alpha(y'))+d'(\alpha(y'),y'')=d'(y',y'')\leq
d(y,\alpha(y))+\delta\, .$$
On the other hand, since $C$ is an $\delta$-overlay, we know that
$\core(\rX)$ and $\core(\rX')$ are in correspondence via $C$, which
implies that 
$$d'(y',\alpha(y'))=d'(y',\core(\rX'))>d(y,\core(\rX))-\delta=d(y,\alpha(y))-\delta\,
,$$
so that $d'(\alpha(y'),y'')\leq 2\delta$. In all cases, we have
$(\alpha(y),\alpha(y'))\in C_{2\delta}$, as claimed . 
\endproof

\begin{lemma}\label{sec:stab-cutt-proc-2}
Fix $r \in (0,1)$. For all $\eps > 0$ there exists $\eta > 0$ such that 
if $\dghp^{1,1}(\rX,\rX') < \eta$ then 
\begin{align*}
\dghp(\rX_x,\rX'_{x'}) & \le \mu(\widetilde{B}_\eps(x))+\mu'(\widetilde{B}_\eps(x')) \\
& \quad + 3\max(h_{\eps}(x),h_{\eps}(x'))+7\bigg(2\frac{\diam(\rX)\vee\diam(\rX')}{r}\vee 1\bigg)\eps
\end{align*}
\end{lemma}
\proof Since $\dghp^{1,1}(\rX,\rX') < \eta$ we can find $C_0 \in
C(\rX,\rX')$ with $\dis(C_0) < \eta$ and with $(x,x') \in C_0$, and a
measure $\pi$ with  $D(\pi;\mu,\mu')\leq \eta$ and $\pi(C_0^c) < \eta$. Fix
$\delta> 0$ such that $\delta < \eps/10$ and $\delta < r/56$. By choosing
$\eta< \delta$ sufficiently small, it follows by \refP{lem:deltaover}
that there exists a $\delta$-overlay $C$ of $\rX$ and $\rX'$ with $C_0
\subset C$, so in particular $(x,x') \in C_0$ and $\pi(C^c) < \eta <
\delta$. We also remark that $D(\pi;\mu,\mu')\leq \delta$. 

We next modify $C$ to give a correspondence between $X_{x,\eps}$ and $X'_{x',\eps}$ by letting
\[
C^{(\eps)}=\big(C\cap (X_{x,\eps}\times X'_{x',\eps})\big)\cup A_1\cup A_2\cup A'_1\cup A'_2\, ,
\]
where for $i\in \{1,2\}$, we define
\begin{align*}
A_i & =\{(y,x'_{(i),\eps}):(y,y')\in C\cap (X_{x,\eps}\times B'_{(i),\eps})\}\, , \\
A_i' & = \{(x_{(i),\eps},y'):(y,y') \in C \cap (B_{(i),\eps}\times X'_{x',\eps}) \}\, .
\end{align*}
To verify that $C^{(\eps)}$ is indeed a correspondence between $X_{x,\eps}$ and $X'_{x',\eps}$, it suffices to check that there does not exist $y\in X_{x,\eps}$ for which $(y,x')\in C$, and similarly that there does not exist $y'\in X'_{x',\eps}$ for which $(x,y')\notin C$.  In the first case this is immediate since $d(x,y)\geq \eps$, so for all $y' \in X'$ with $(y,y') \in C$ we have $d'(x',y')\geq \eps-\delta>0$. A symmetric argument handles the second case.

We next estimate the distortion of $C^{(\eps)}$ 
when $X_{x,\eps}$ and $X'_{x',\eps}$ are endowed with the
metrics $d_x$ and $d'_{x'}$ respectively.  To this end, let $(y,y'),(z,z')\in
C^{(\eps)}$.  We have to distinguish several cases. The simplest
case is when $(y,y')$ and $(z,z')$ are, in fact, both in
$C$. In particular, $y,z\in X_{x,\eps}$ and $y',z'\in X'_{x',\eps}$.
Let $f$ be a geodesic from $y$ to $z$ in $X_{x,\eps}$, i.e.\ a local
geodesic in $X_{x,\eps}$ not passing through $x$ and with minimal
length. Let $f'$ be the path from $y'$ to $z'$ associated with $f$ as
in Lemma \ref{sec:cutting-r-graphs}, which we may apply since $\delta
< r/56$. We claim that $f'$ does not pass through $x'$. Indeed, if
it did, then we would be able to find a point $x_0\in \im(f)$ such
that $(x_0,x')\in C_{8\delta}$. Since also $(x,x') \in C_{8\delta}$,
it would follow that
\[
d(x,\im(f))\leq d(x,x_0)\leq
d'(x',x')+8\delta< \eps\, ,
\] 
contradicting the fact that $f$ is a path in
$X_{x,\eps}$. 
By Lemma \ref{sec:cutting-r-graphs}, we deduce that 
\begin{align}
d'_{x'}(y',z')
&\leq d_x(y,z)\bigg(1+\frac{64\delta}{r\wedge
  d_x(y,z)}\bigg)\nonumber\\
&= d_x(y,z)+64\bigg(\frac{d_x(y,z)}{r}\vee 1\bigg)\delta\nonumber\\
&\leq  d_x(y,z)+64\bigg(\frac{2\diam(\rX)}{r}\vee 1\bigg)\delta\, ,\label{eq:7}
\end{align}
where at the last step we use that $d_x(y,z)\leq
\diam(\rX_x)\leq 2\diam(\rX)$.

Let us now consider the cases where $(y,y')\notin C$, still assuming that $(z,z')\in C$. There are two possibilities. 
\begin{enumerate}
\item There exists $y''\in B'_{(i),\eps}$ with $(y,y'')\in C$ and $i\in \{1,2\}$, and so $y'=x'_{(i),\eps}$.
\item
There exists $\overline{y}\in B_{(i),\eps}$ with $(\overline{y},y')\in
C$ and $i \in \{1,2\}$, and so $y=x_{(i),\eps}$. 
\end{enumerate}
Let us consider the first case, assuming $i=1$ for definiteness. The argument leading to \eqref{eq:7} is still valid, with $y''$ replacing $y'$. Using
$d'_{x'}(y'',x'_{(1),\eps})=d'(y'',x'_{(1),\eps})\leq h_\eps(x')$, we obtain
$$d'_{x'}(y',z')\leq d_x(y,z)+64 \bigg(\frac{2\diam(\rX)}{r}\vee 1\bigg)\delta+h_\eps(x')\, .$$

In the second case (still assuming $i=1$ without loss of generality),
we have to modify the argument as follows. We consider a geodesic $f$
from $\overline{y}$ to $z$ in $(X_x,d_x)$. We  $f'$ be the associated
path from $y'$ to $z'$ (again using Lemma \ref{sec:cutting-r-graphs}), and claim that $x' \not\in \im(f')$. Otherwise, $f$ would visit a point at distance less than $8\delta$ from $x$. On the other hand, the point of $\im(f)$ that is closest to $x$ is $\alpha(\overline{y})$. But by Lemma \ref{sec:stab-cutt-proc}, we have
\[
d(x,\alpha(\overline{y}))\geq d'(x',\alpha(y'))-2\delta\geq \eps-2\delta > 8\delta\, .
\]
Finally, since $y=x_{(1),\eps}$ we obtain that $d_x(\overline{y},z)\leq d_x(y,z)+h_\eps(x)$, and the argument leading to \eqref{eq:7} yields
\[
d'_{x'}(y',z')\leq d_x(y,z)+h_\eps(x)+64 \bigg(\frac{2\diam(\rX)}{r}\vee 1\bigg)\delta\, .
\]
Arguing similarly when $(z,z')$ is no longer assumed to belong to $C$, we obtain the following bound for every $(y,y'),(z,z')\in C^{(\eps)}$: 
$$d'_{x'}(y',z')\leq d_x(y,z)+2(h_\eps(x)\vee h_\eps(x'))+64 \bigg(\frac{2\diam(\rX)}{r}\vee 1\bigg)\delta\, .$$
Writing $h_{\eps} = h_{\eps}(x)\vee h_{\eps}(x')$, by symmetry we thus
conclude that
\[
\dis(C^{(\eps)})\leq 2h_{\eps}+64 \bigg(2\frac{\diam(\rX)\vee\diam(\rX')}{r}\vee 1\bigg)\delta\, ,
\]
where the distortion is measured with respect to the metrics $d_x$ and $d'_{x'}$. 
Now let $\hat{C}^{(\eps)}$ be the $h_{\eps}$-enlargement of $C^{(\eps)}$ with respect to $d_x$ and $d'_{x'}$. 
Since $C^{(\eps)}$ is a correspondence between $X_{x,\eps}$ and $X'_{x',\eps}$, and  
all points of $X_x$ (resp.\ $X'_{x'}$) have distance at most $h_{\eps}$ from $X_{x,\eps}$ (resp.\ $X'_{x',\eps}$) under $d_x$ (resp.\ $d'_{x'}$), 
we have that $\hat{C}^{(\eps)}$ is a correspondence between $X_x$ and $X'_{x'}$, of distortion at most 
\[
3h_{\eps}+64\bigg(2\frac{\diam(\rX)\vee\diam(\rX')}{r}\vee 1\bigg)\delta\, . 
\]

Finally, since $D(\pi;\mu,\mu')\leq \delta$ and $\pi(C^c)\leq \delta$, and since 
$(C \cap (X_{x,\eps} \times X'_{x',\eps})) \subset  C^{(\eps)} \subset \hat{C}^{(\eps)}$, 
we have 
\[
\pi((\hat{C}^{(\eps)})^c) \le \pi(C^c) + \pi(\widetilde{B}_{\eps}(x)\times X') + \pi(X \times \widetilde{B}_{\eps}(x')) 
\le \delta + \mu(\widetilde{B}_\eps(x))+\mu'(\widetilde{B}_\eps(x'))\, . 
\]
Since $65 \delta < 6.5 \eps < 7\eps$, the lemma then follows from the
two preceding offset equations and the definition of the distance
$\dghp$.  
\endproof 

\begin{proof}[Proof of Theorem \ref{thm:safelypointed}. ]
  Fix $\eps > 0$. Under the hypotheses of the theorem, for all $n$ large enough, by \refL{sec:stab-cutt-proc-2} we have 
\begin{align*}
\dghp(\rX_x,\rX^n_{x^n})  
\le \;
&\mu(\widetilde{B}_\eps(x))+\mu'(\widetilde{B}_\eps(x^n)) 
 + 3\max(h_{\eps}(x),h_{\eps}(x^n))\\
&+7\bigg(2\frac{\diam(\rX)\vee\diam(\rX^n)}{r}\vee 1\bigg)\eps.  
\end{align*}
It is easily checked that, for all $\eps > 0$,
$$\limsup_{n\to\infty}h_\eps(x^n)\leq h_{2\eps}(x)\, ,\qquad
\limsup_{n\to\infty}\mu^n(\widetilde{B}_\eps(x^n))\leq \mu(\widetilde{B}_{2\eps}(x))\, ,
$$
which both converge to $0$ as $\eps\to 0$.  The result follows.
\end{proof}

\subsection{Randomly cutting $\R$-graphs} \label{sec:cuttingrgs}

Let $\rX=(X,d,x)$ be a safely pointed $\R$-graph, and write $\Ell$
for the length measure restricted to $\conn(\rX)$.  Then
$\rX_x=(X_x,d_x)$ is an $\R$-graph with $s(\rX_x)=s(\rX)-1$. Indeed,
if $e$ is the edge of $\ker(\rX)$ that contains $x$, then it is easy
to see that $\ker(\rX_x)$ is the graph obtained from $\rX$ by first
deleting the interior of the edge $e$, and then taking the kernel of
the resulting $\R$-graph. Taking the kernel
of a graph does not modify its surplus, and so the surplus diminishes by
$1$ during this operation, which corresponds to the deletion of the edge
$e$. Moreover, we see that $\mathcal{A}_r$ is stable under this
operation, in the sense that if $(X,d)\in \mathcal{A}_r$, then for
every $x$ such that $(X,d,x)$ is safely pointed, the space $(X_x,d_x)$
is again in $\mathcal{A}_r$. Indeed, the edges
in $\ker(\rX_x)$ are either edges of $\ker(\rX)$, or a concatenation
of edges in $\ker(\rX)$, and so the minimum
edge-length can only increase. On the other hand, the total core
length and surplus can only decrease. 

Let us now consider the following random cutting procedure for
$\R$-graphs. If $(X,d)$ is an $\R$-graph which is not an $\R$-tree,
then it contains at least one cycle, and by 
Proposition~\ref{sec:skeleton-core-kernel-2}, $\ell$-almost every point of any
such cycle is in $\conn(\rX)$. Consequently, the measure
$\Ell=\ell(\cdot\cap \conn(\rX))$ is non-zero, and we can consider
a point $x$ chosen at random in $\conn(\rX)$ with distribution
$\Ell/\Ell(\conn(\rX))$. Then $(X,d,x)$ is a.s.\ safely pointed
by Proposition \ref{sec:skeleton-core-kernel-2}. Let
$\mathcal{K}(\rX,\cdot)$ be the distribution of $(X_x,d_x)$. By
convention, if $\rX$ is an $\R$-tree, we let
$\mathcal{K}(\rX,\cdot)=\delta_{\{\rX\}}$.  By combining
Corollary \ref{sec:cutt-safely-point-1} (ii) with Theorem
\ref{thm:safelypointed}, we immediately
obtain the following statement.

\begin{proposition}\label{sec:randomly-cutting-r-1}
  Fix $r>0$, and let $(\rX^n,n\geq 1)$ and $\rX$ be elements
  of $\mathcal{A}_r$ such that $\dghp(\rX^n,\rX)\to0$ as
  $n\to\infty$. Then $\mathcal{K}(\rX^n,\cdot) \convdist \mathcal{K}(\rX,\cdot)$ in
  $(\cM,\dghp)$, as $n \to \infty$. 
\end{proposition}

In particular, $\mathcal{K}$ defines a Markov kernel from $\mathcal{A}_r$ to
itself for every $r$. Since each application of this kernel decreases
the surplus by $1$ until it reaches $0$, it makes sense
to define $\mathcal{K}^\infty(\rX,\cdot)$ to be the law of $\mathcal{K}^m(\rX,\cdot)$ for
every $m\geq s(\rX)$, where $\mathcal{K}^m$ denotes the $m$-fold composition of
$\mathcal{K}$. 
The next corollary follows immediately from Proposition
\ref{sec:randomly-cutting-r-1} by induction.

\begin{corollary}\label{cor:randomly-cutting-r-2}
Fix $r>0$, and let $(\rX^n,n\geq 1)$ and $\rX$ be elements of $\mathcal{A}_r$ with $\dghp(\rX^n,\rX)\to0$ as $n\to\infty$. Then 
$\mathcal{K}^{\infty}(\rX^n,\cdot) \convdist \mathcal{K}^{\infty}(\rX,\cdot)$ in $(\cM,\dghp)$, as $n \to \infty$.
\end{corollary}
This proves Theorem~\ref{sec:breaking-cycles-r-1}.

\small
\printnomenclature[3.1cm]
\normalsize

\section*{Acknowledgements}
\addcontentsline{toc}{section}{Acknowledgements} LAB was supported for
this research by NSERC Discovery grant and by an FQRNT Nouveau
Chercheur grant, and thanks both institutions for their support.  LAB
and CG were supported for this research by Royal Society International
Exchange Award IE111100.  NB acknowledges ANR-09-BLAN-0011. CG is grateful to
Universit\'e Paris-Sud for making her professeur invit\'e for March
2011 which enabled work on this project to progress. She is also
grateful for support from the University of Warwick.  GM acknowledges
the support of grants ANR-08-BLAN-0190 and ANR-08-BLAN-0220-01. He
is grateful to PIMS (CNRS UMI 3069) and UBC, Vancouver, for hospitality in the year
2011--2012, and for the support of Universit\'e Paris-Sud. 

\def\polhk#1{\setbox0=\hbox{#1}{\ooalign{\hidewidth
  \lower1.5ex\hbox{`}\hidewidth\crcr\unhbox0}}} \providecommand{\noopsort}[1]{}

\end{document}